\documentclass[onecolumn,showpacs,preprintnumbers,amsmath,amssymb,superscriptaddress]{revtex4-1}
\usepackage{amsmath}
\usepackage{amsfonts}
\usepackage{amssymb}
\usepackage{amsthm}
\usepackage{bm}
\usepackage{indentfirst}
\usepackage{graphicx}
\usepackage{subfigure}
\usepackage{array}
\usepackage{fullpage}
\usepackage[hidelinks]{hyperref}
\usepackage{color}
\usepackage{latexsym}
\usepackage{soul}
\usepackage{etoolbox}
\usepackage{acronym}
\usepackage[super]{nth}
\usepackage[draft,inline,nomargin]{fixme}

\mathchardef\mhyphen="2D

\newcommand{\ignore}[1]{}

\newcommand{\Rmnum}[1]{\uppercase\expandafter{\romannumeral #1\relax}}

\newcommand{\mb}{\mathbf}

\DeclareMathAlphabet{\mathsfsl}{OT1}{cmss}{m}{sl}

\newcommand{\dif}{\,\mathrm{d}}

\newcommand{\ba}{\bm\alpha}
\newcommand{\bx}{\bm\xi}
\newcommand{\ve}{\varepsilon}
\newcommand{\PreserveBackslash}[1]{\let\temp=\\#1\let\\=\temp}
\newcolumntype{C}[1]{>{\PreserveBackslash\centering}p{#1}}
\newcolumntype{R}[1]{>{\PreserveBackslash\raggedleft}p{#1}}
\newcolumntype{L}[1]{>{\PreserveBackslash\raggedright}p{#1}}

\numberwithin{equation}{section}
\newtheorem{thm}{Theorem}[section]

\newtheorem{cor}{Corollary}
\theoremstyle{definition}
\newtheorem{defn}[thm]{Definition}

\newtheorem{rem}[thm]{Remark}

\newcommand{\black}[1]{{\color{black}#1}}

\usepackage{newfloat, algcompatible}
\usepackage{algcompatible}
\DeclareFloatingEnvironment[fileext=loa, listname=List of Algorithms, name=ALGORITHM, placement=tbhp]{algorithm}

\acrodef{QoI}{quantity of interest}
\acrodef{UQ}{uncertainty quantification}
\acrodef{MC}{Monte Carlo}
\acrodef{GP}{Gaussian Process}
\acrodef{gPC}{generalized polynomial chaos}
\acrodef{aPC}{arbitrary polynomial chaos}
\acrodef{amdP}{\emph{polynomials for arbitrary mutually dependent}}
\acrodef{DSRAR}{\emph{Data-driven Sparsity-enhancing Rotation for Arbitrary Randomness}} 
\acrodef{CS}{compressed sensing}
\acrodef{SASA}{solvent-accessible surface area}
\acrodef{PDE}{partial differential equation}
\acrodef{RIC}{restricted isometry constant}
\acrodef{RIP}{restricted isometry property}
\acrodef{PCA}{principal component analysis}
\acrodef{MD}{molecular dynamics}
\acrodef{APBS}{Adaptive Poisson-Boltzmann Solver}
\acrodef{PDF}{probability density function}
\acrodef{GAFF}{General AMBER Force Field}

\usepackage{tikz}
\usetikzlibrary{patterns}

\newcommand{\dashline}{\raisebox{0pt}{\tikz{\draw[-,dashed,line width = 1.0pt](0.,1mm) -- (10.0mm,1mm)}}}

\newcommand{\rectangleopen}{\raisebox{0pt}{\tikz{\draw[solid,line width = 1.0pt](4.mm,0)rectangle (5.8mm,1.8mm);
  }}}

\newcommand{\circleopen}{\raisebox{0pt}{\tikz{\draw[solid, line width = 1.0pt](5mm,0) circle[radius=1.0mm];
}}}

\newcommand{\triangleopen}{\raisebox{0pt}{\tikz{\draw[solid, line width = 1.0pt](4mm,0.15mm) -- (6mm,0.15mm) -- (5mm,1.95mm) -- cycle;
}}}

\newcommand{\diamondopen}{\raisebox{0pt}{
  \tikz{\draw[solid, line width = 1.0pt](-1mm,0mm) -- (0mm,1mm) -- (1mm,0mm) -- (0mm,-1mm) -- cycle;
  }}}

\newcommand{\rectanglesolidline}{\raisebox{0pt}{\tikz{\draw[solid,fill, line width = 1.0pt](4.mm,0)rectangle (5.8mm,1.8mm);
  \draw[-,solid,line width = 1.0pt](0.,1mm) -- (10.0mm,1mm)}}}

\newcommand{\rectangledashline}{\raisebox{0pt}{\tikz{\draw[solid,fill, line width = 1.0pt](4.mm,0)rectangle (5.8mm,1.8mm);
  \draw[-,dashed,line width = 1.0pt](0.,1mm) -- (10.0mm,1mm)}}}

\newcommand{\trianglesolidline}{\raisebox{0pt}{\tikz{\draw[solid,fill, line width = 1.0pt](4mm,0.15mm) -- (6mm,0.15mm) -- (5mm,1.95mm) -- cycle;
  \draw[-,solid,line width = 1.0pt](0.,1.05mm) -- (10.0mm,1.05mm)}}}

\newcommand{\triangledashline}{\raisebox{0pt}{\tikz{\draw[solid,fill, line width = 1.0pt](4mm,0.15mm) -- (6mm,0.15mm) -- (5mm,1.95mm) -- cycle;
  \draw[-,dashed,line width = 1.0pt](0.,1.05mm) -- (10.0mm,1.05mm)}}}

\newcommand{\triangledashdotline}{\raisebox{0pt}{\tikz{\draw[solid,fill, line width = 1.0pt](4mm,0.15mm) -- (6mm,0.15mm) -- (5mm,1.95mm) -- cycle;
  \draw[-,dashdotted,line width = 1.0pt](0.,1.05mm) -- (10.0mm,1.05mm)}}}

\newcommand{\downtrianglesolidline}{\raisebox{0pt}{\tikz{\draw[solid,fill, line width = 1.0pt](4mm,1.95mm) -- (6mm,1.95mm) -- (5mm,0.15mm) -- cycle;
  \draw[-,solid,line width = 1.0pt](0.,1.05mm) -- (10.0mm,1.05mm)}}}

\newcommand{\downtriangledashline}{\raisebox{0pt}{\tikz{\draw[solid,fill, line width = 1.0pt](4mm,1.95mm) -- (6mm,1.95mm) -- (5mm,0.15mm) -- cycle;
  \draw[-,dashed,line width = 1.0pt](0.,1.05mm) -- (10.0mm,1.05mm)}}}

\newcommand{\diamondsolidline}{\raisebox{0pt}{
  \tikz{ \draw[solid,fill, line width = 1.0pt](4mm,0mm) -- (5mm,1mm) -- (6mm,0mm) -- (5mm,-1mm) -- cycle;
  \draw[-,solid,line width = 1.0pt](0.,0mm) -- (10.0mm,0mm)}}}

\newcommand{\diamonddashline}{\raisebox{0pt}{
  \tikz{ \draw[solid,fill, line width = 1.0pt](4mm,0mm) -- (5mm,1mm) -- (6mm,0mm) -- (5mm,-1mm) -- cycle;
  \draw[-,dashed,line width = 1.0pt](0.,0mm) -- (10.0mm,0mm)}}}

\newcommand{\diamonddashdotline}{\raisebox{0pt}{
  \tikz{ \draw[solid,fill, line width = 1.0pt](4mm,0mm) -- (5mm,1mm) -- (6mm,0mm) -- (5mm,-1mm) -- cycle;
  \draw[-,dashdotted,line width = 1.0pt](0.,0mm) -- (10.0mm,0mm)}}}

\usepackage{hyperref}

\begin{document}

 \preprint{}

\title{A data-driven framework for sparsity-enhanced 
surrogates with arbitrary mutually dependent randomness}% Force line breaks with \\

\author{Huan Lei}
\email{huan.lei@pnnl.gov }
\affiliation{Pacific Northwest National Laboratory, Richland, WA 99352.}%
\author{Jing Li}
\thanks{The first two authors contributed equally}
\affiliation{Pacific Northwest National Laboratory, Richland, WA 99352.}%
\author{Peiyuan Gao}
\affiliation{Pacific Northwest National Laboratory, Richland, WA 99352.}%
\author{Panos Stinis}
\affiliation{Pacific Northwest National Laboratory, Richland, WA 99352.}%
\affiliation{Department of Applied Mathematics, University of Washington, Seattle, WA 98195.}%
\author{Nathan A.\ Baker}
\email{nathan.baker@pnnl.gov }
\affiliation{Pacific Northwest National Laboratory, Richland, WA 99352.}%
\affiliation{Division of Applied Mathematics, Brown University, Providence, RI 02912.}%

\date{\today} % It is always \today, today,

\begin{abstract}
    The challenge of quantifying uncertainty propagation in real-world systems is 
rooted in the high-dimensionality of the stochastic input and the frequent lack 
of explicit knowledge of its probability distribution. Traditional approaches 
show limitations for such problems, especially when the size of the training data is limited.
To address these difficulties, we have developed a general framework of constructing 
surrogate models on spaces of stochastic input with arbitrary probability measure 
irrespective of \black{the mutual dependencies between individual components
of the random inputs} and the analytical form.
The present \acf*{DSRAR} framework 
includes a \black{data-driven construction of multivariate polynomial basis for arbitrary mutually
dependent probability measure} and a sparsity enhancement rotation procedure.
This sparsity enhancement method was initially proposed in our previous work \cite{Lei_Yang_MMS_2015} 
for Gaussian density distributions, which may not be feasible for non-Gaussian distributions due 
to the loss of orthogonality after the rotation. To remedy such difficulties, \black{we 
developed a new data-driven approach to construct orthonormal polynomials for \acf*{amdP} randomness,
ensuring the constructed basis} maintains the orthogonality/near-orthogonality 
with respect to the density of the rotated random vector, \black{where 
directly applying the regular polynomial chaos including arbitrary polynomial chaos (aPC) \cite{OLADYSHKIN2012} 
shows limitations due to the assumption of the mutual independence between 
the components of the random inputs.}   
The developed \acs*{DSRAR} framework leads to accurate recovery, with only limited training data, of a sparse representation 
of the target functions. 
The effectiveness of our method is demonstrated in challenging problems such as PDEs and realistic molecular systems 
\textcolor{black}{within high-dimensional conformational space ($O(10)$)} where the underlying density is implicitly represented by a \textcolor{black}{large collection of sample data}, as well as systems {with explicitly given} non-Gaussian probabilistic measures.

\end{abstract}
% \pacs{}% check this
\maketitle

%\listoffixmes

\section{Introduction}\label{sec:intro}

A fundamental problem in \ac{UQ} \cite{Saltelli_book_2008} is to calculate 
the statistical properties of a \ac{QoI} due to various sources of randomness, e.g., \black{
numerical simulations} subject to uncertain parameters, initial conditions and/or boundary 
conditions, as well as experimental measurements in the presence of material heterogeneity, thermal fluctuations. Such sources of uncertainty 
are usually characterized by high-dimensional random variables whose probability 
measures can be either discrete or continuous. In real-world systems, there are usually two crucial
challenges to accurately quantify the propagation of the randomness from the input to the system response.
The first challenge comes from the high-dimensionality of the random inputs. \black{
For such systems, limited computational resources often motivates further dimensionality reduction \cite{Bishop_2006}. 
However, it is often non-trivial to accurately transfer the high-dimensional random space into 
a low-dimensional random space.}
This results in the numerical intractability of quantifying the uncertainty of the \ac{QoI} from training data of limited size. 
The second challenge arises from frequent dependencies and arbitrary distribution of the random inputs. 
Typically, random inputs are \black{represented by random vectors with mutually independent components.}   
For realistic systems, the underlying distribution of the inputs can often involve dependencies that cannot be 
ignored \black{(e.g., see molecule systems in Ref. \cite{Laio_Parrinello_PNAS_2002} and Sec. \ref{sec:mole_example})}.
\black{Moreover, the input distribution could be even unknown} and thus we may only have access to it implicitly through a 
collection of samples. This creates further numerical obstacles in characterizing the random inputs 
as well as their effect on the system response. In the current work, we present a 
\ac{DSRAR} framework
for dealing with all of the aforementioned challenges. \black{While we focused on numerical experiments in the present study, the developed 
framework can be also applied to \ac{UQ} in experimental studies.}

In practice, a straightforward and robust approach is the \ac{MC} method, which involves collecting 
a large number of samples of the random inputs from their distribution, evaluating 
the \ac{QoI} at each sample point, and then obtaining the statistical properties (mean, variance, \black{sensitivity indices}, 
probability density function, probability of a certain event etc.) of the \ac{QoI}. Unfortunately, to 
get an accurate estimate, the \ac{MC} method requires a large number of simulations due to
its slow convergence rate \cite{Fishman96, Kucherenko_2015}. Furthermore, for large or complex systems, even a single 
instance of these simulations may require very large computational resources. Under such circumstances, 
the computational cost of \ac{MC} method can become extremely large. \textcolor{black}{Several approaches have been developed to 
alleviate such difficulties. 
For instance, sampling approaches such as multilevel-\ac{MC} \cite{Giles_2015,Heinrich_MLMC_2001,Pisaroni_Nobile_CMAME_2017}
and multifidelity-\ac{MC} \cite{Koutsourelakis_SISC_2009, Peherstorfer_Willcox_SISC_2016} have 
been designed to optimize the computational load when samples of 
the \ac{QoI} are available at hierarchical levels of accuracy;} sampling approaches like quasi-\ac{MC} 
\cite{Fox99,Niedeereiter92,NiederreiterHLZ98} and Latin Hypercube sampling 
\cite{Mckay_Beckman_Technometrics_1979,Stein_tech87,Loh_AS96}, have been designed to accelerate convergence. 
However, when the underlying distribution of the inputs is arbitrary and not explicitly given, 
\textcolor{black}{the latter two sampling strategies may lose their advantage if it 
is not straightforward to generate quasi-random sequences following the underlying distribution}. 

An alternative approach approximates the \ac{QoI} via constructing the surrogate model 
of the random inputs and then calculates the statistics of the \ac{QoI} 
analytically or numerically.
\black{
Among such approaches, the most popular are the Gaussian Process \cite{Sack_Welch_Stat_1989, 
Kennedy_Hagan_JRSS_2001, GP_book_Rasmussen_2006}, and the polynomial chaos 
expansion originally introduced by Wiener \cite{Wiener38}, applied to \ac{UQ} by Ghanem 
\cite{Ghanem_1991Spectralappro} and extended to the \ac{gPC} expansion by Xiu \cite{Xiu_2002Wiener}.
The \ac{GP} is a stochastic process which approximates the values of the \ac{QoI} at every finite sets of 
sample point as multivariate Gaussian random vectors. The flexibility of the mean and covariance functions 
enables \ac{GP} to characterize a wide range of function behavior with broad applications on \ac{UQ}
\cite{GP_book_Rasmussen_2006, Qian_Wu_2008, Williams_2006, Oakley_OHagan_Biometrika_2002, Lockwood_Anitescu_2012}.}
The \ac{gPC} expansion approximates the \ac{QoI} by a set of simple basis functions. 
It is known to be a \emph{mathematically optimal} approximation of the \ac{QoI} 
when the basis functions are chosen to be orthogonal with respect to the probability measure of the random inputs.
This approach has been demonstrated for diverse applications 
in \ac{UQ} \cite{XiuK_JCP03,GhanemMPW_CMAME05,Knio_CFD06,Sudret2008,LiXiu_JCP09, MarzoukXiu2009, Li2010,Li2018} 
due to its spectral convergence under certain situations. 
\black{In this study, we focus on the approach developed 
based on \ac{gPC} and we refer to previous publications \cite{Schobi_2015,Gratiet_2016,Owen_Challenor_SIAM_2017,Roy_Mocayd_2018} 
(and the references therein) for comparative studies of the two approaches. }
%Both approaches have broad applications on \ac{UQ} \cite{XiuK_JCP03,GhanemMPW_CMAME05,Knio_CFD06,
%Sudret2008,LiXiu_JCP09, MarzoukXiu2009, Li2010,Li2018}  
%\ac{gPC} expansions have diverse applications and far-reaching influence on \ac{UQ} \cite{XiuK_JCP03,GhanemMPW_CMAME05,Knio_CFD06,Sudret2008,LiXiu_JCP09, MarzoukXiu2009, Li2010,Li2018} due to their spectral convergence under certain situations.
%However, \ac{gPC} expansions can only handle independent identically distributed (i.i.d.) random variables of standard types (uniform, Gaussian, gamma, beta, etc.).
%Otherwise, a pre-processing step to transform the original random variables into i.i.d.\ random variables of standard type is required.
%In general, these transformations are highly nonlinear which result in the final function approximation of the \ac{QoI} to be a high-degree polynomial in order to maintain the accuracy.
%To avoid the slow convergence 
%of the transformed random variables, Wan\cite{WanK_SISC06} introduced the gPC of arbitrary probability 
%measures and extended the technique to discrete measures in \cite{Zheng_2015MEPCM}.  

%
In principle, if the orthogonal polynomial type and the corresponding random 
variables are determined, both intrusive and non intrusive methods can be used 
to evaluate the coefficients of the expansion. For example, stochastic collocation, based 
on tensor products of one-dimensional quadrature rules, is often employed when 
dimensionality is small \cite{MathelinH_NASA03,XiuH_SISC05,Babuska2007}, \black{with the 
number of basis functions given by $(p+d)!/p!d!$, where $p$ is polynomial order and $d$ is the dimension}.
However, as the dimension increases, the number of quadrature points needed for the tensor 
product rule increases exponentially. \black{To mitigate this issue, sparse grid 
and adaptive 
collocation methods have been proposed to deal with moderate dimensionality 
\cite{XiuH_SISC05,nobile2008sparse,Ma2009,Foo2010,CONSTANTINE2012,jakemanG2013,Li2016}}. 
%However, one of the big challenges of uncertainty quantification is the curse of dimensionality--as the dimensionality becomes 
%larger, both the number of polynomial basis as well as the number of 
%quadrature points increase dramatically. To calculate the coefficients of the expansion of the QoI, simulations 
%on each quadrature points are desired. 
%
When the dimension of the random inputs is large, none of the above collocation 
methods is feasible. In the case of a limited number of available simulations and large dimensionality, \ac{CS} approaches have been used to construct sparse polynomial approximations of the \ac{QoI} \cite{Doostan_2011nonadapted, Yan_2012Sc, Rauhut_2012sparseLegen,mathelin_gallivan_2012,Yang_2013reweightedL1,Hampton_2015Cs,Peng_2016gradientL1,Yan_2017,Liu_2016QMCL1, Lei_Karniadakis_JCP_2017,ALEMAZKOOR2017,DIAZ2018,RAI2018}. 
\black{Finally, we note that \ac{gPC} (including extensions such as arbitrary polynomial
chaos \cite{OLADYSHKIN2012}), in its current form, can only handle
 random vector with 
independent identically distributed (i.i.d.) components in standard types (uniform, Gaussian, gamma, beta, etc.).
For other distributions, a pre-processing step is required to transform the original random variables into i.i.d.\ random variables of standard types.
In general, these transformations are highly nonlinear which result in the final \ac{QoI} function approximation 
to be a high-degree polynomial in order to maintain accuracy.}

%In this paper we focus on a {\it non-intrusive} way to effect the approximation. 
The methods discussed above rely on the \emph{explicit} knowledge of the 
underlying probability measures and/or the assumption of mutual independence between
the components of the random inputs. However, such assumptions on the random inputs can be quite 
restrictive for realistic applications. One such example is the
\ac{UQ} for molecular system properties \acp{QoI} due to conformational 
fluctuations \cite{Huang_Protein_book_2005}. For such systems, the random inputs are the various conformational states 
(i.e., the instantaneous structure) of the molecule. The underlying distribution is 
determined by the free energy function of the system, which is essentially the multi-dimensional marginal density
distribution with respect to the (Boltzmann) distribution of the full Hamiltonian system.
Unfortunately, numerical evaluation of the free energy function is a well-known challenging
problem. Although various sampling strategies have been 
developed \cite{Kumar_Kollman_JCC_1992,Laio_Parrinello_PNAS_2002,Mar_Van_JCP_2008}, the
explicit free energy function is usually unknown for dimensions greater than $4$. 
In practice, the underlying density is only known implicitly through a large collection 
of the molecule conformational states obtained from experiments or simulated trajectories. 
Another commonly encountered example arises in our recent work \cite{Lei_Yang_MMS_2015} on constructing sparse 
representations of a  \ac{QoI} based on \ac{CS}. 
Inspired by the active subspace method \cite{ConstantineDW14}, we 
proposed a method to enhance the sparsity of polynomial expansion in terms of a 
new random vector via unitary rotation of the original random vector. For i.i.d.\ Gaussian 
random inputs, the new random vector retains the same distribution. However, for 
non-Gaussian random inputs, which are more realistic for applications, the new random vector
does not retain the mutual independence even if the original random vector elements are i.i.d.
%
%The random inputs are described as a large collection of independent samples. 
%

For problems with non-Gaussian random inputs,
the traditional approach is to cast the available statistics 
into a family of standard distributions and then to apply the \ac{gPC} techniques discussed above.
Gaussian mixture models, due to their flexibility, are broadly employed to approximate 
the distribution of the data. 
With the distribution approximated, a \ac{gPC} expansion of the \ac{QoI} can be constructed for each Gaussian component.
The statistical properties of the \ac{QoI} are derived by combining the statistical properties of all components \cite{LIweixuan2015,Vittaldev2016}.
However, there are two drawbacks of the Gaussian mixture approach: (i) it lacks one-to-one correspondence between one instance of random inputs and the approximated function evaluation, (ii) it is difficult to determine an appropriate and accurate probability density approximation when the dimension is larger than one.
Copulas have been employed to treat dependent probabilistic models for surrogate construction in \cite{FEINBERG2015}.
Zabaras \cite{Zabaras_2014} has established a graph-based approach to factorize the joint distribution into a set of conditional distributions based on the dependence structure of the variables.
Alternatively, several studies have been devoted to constructing orthogonal
polynomial bases using the moments of the random variables. 
Orthogonal polynomial chaos for random vectors with independent components of
arbitrary measure
was proposed in \cite{OLADYSHKIN2012, WanK_SISC06, Witteveen_Bijl_2006, Zheng_2015MEPCM,YIN2018}. 
Ahlfeld investigated the quadrature rule of this \ac{aPC} and proposed 
a sparse quadrature rule for the integration which can facilitate the evaluation of the 
expansion coefficients \cite{AHLFELD2016}. However, those quadrature rules of arbitrary 
polynomial chaos again assume \black{ \emph{the components of the random inputs are mutually independent.}} 

In this paper, we develop a \black{general \ac{UQ} framework for constructing 
surrogate models via
\ac{DSRAR} \emph{irrespective of possible mutual dependencies 
between the random input components}.}
\black{This approach is different from the aforementioned studies based on 
polynomial chaos expansions and, therefore, can be particularly} 
useful for realistic systems where the input distributions
can be non-standard or unknown analytically. 
\black{
The key idea is a data-driven approach for basis construction, consisting
of multivariate orthonormal
\ac{amdP},}  
coupled with the previously developed rotation-based sparsity enhancement approach \cite{Lei_Yang_MMS_2015}. This can be 
viewed a special case of the present method when the random inputs are from a Gaussian distribution. 
When the size of the training set is limited, the method can recover the 
expansion coefficients by \ac{CS}, \black{under the assumption that there 
exists a sparse representation of surrogate model.}
As we will show, directly employing a regular polynomial basis and/or the sparsity enhancement rotation 
on the random input may result in large recovery error due to the violation of orthogonality for non-standard density distributions.
The procedure of data-driven basis construction described in the present study retains proper orthogonality 
with respect to the associated random inputs and therefore ensures more accurate recovery.
In this sense, the present method takes advantage of both the orthonormal basis expansion and the enhanced
sparsity of the expansion coefficients. 
{The method deals with two situations widely encountered in real-world applications: (\Rmnum{1}) probability measures 
  that are implicitly represented by \textcolor{black}{a large collection of samples} and 
  (\Rmnum{2}) non-Gaussian probability measures with explicit (analytical) forms.
For the first situation, we construct
orthonormal polynomial bases with respect to discrete measures on the sample set. 
Besides the exact orthonormal basis, we also propose a heuristic method to 
construct a \emph{near-orthonormal} basis, which yields a smaller basis bound 
than the exact orthonormal basis and results in more accurate recovery of the 
sparse representation. For the second situation, we construct the orthonormal basis when
the quadrature rules for polynomial integration are known. This construction
is especially well suited to random variables obtained from sparsity enhancement of non-Gaussian distributions. 
}

The paper is organized as follows.
In Section~\ref{sec:prelim}, we present the problem setup and briefly review preliminary background on multivariate orthogonal polynomials and compressed sensing.
In Section~\ref{sec:num_method}, we present the \ac{DSRAR} framework by first introducing 
the methods to construct data-driven orthonormal \ac{amdP} basis. When the underlying density is implicitly represented by a large 
collection of random input samples, we propose a heuristic approach to construct a near-orthonormal basis  along with some heuristics on the advantage over an exactly orthonormal basis.
Then we introduce the rotation-based sparsity enhancement method and provide algorithmic details on how to combine the data-driven basis construction and sparsity enhancement rotation.
In Section~\ref{sec:numerical}, we demonstrate the developed framework 
\black{
in a realistic molecular system fluctuating in a high-dimensional conformational space ($O(10)$) 
as well as \acp{PDE} with arbitrary randomness where the 
underlying distributions are either explicitly known or implicitly represented by a large collection of samples.} 
Concluding remarks and directions for future work are provided in Section~\ref{sec:summary}.

\section{Background} \label{sec:prelim}

\subsection{Approximation with orthogonal polynomials} \label{sec:aPE}

We begin with a few facts about multivariate orthogonal polynomials \cite{dunkl_xu_2014}. 
Let $\Pi^d$ be the set of polynomials in $d$ variables on $\mathbb{R}^d$.
Polynomials in $\Pi^d$ are naturally indexed by the multi-indices set $\mathbb{N}_0^d$.
For $\bm{\alpha} = (\alpha_1,\dots,\alpha_d)\in \mathbb{N}_0^d$ and $\bm{z} = (z_1,\dots,z_d)$, a monomial $z^\alpha$ is defined by $\bm{z}^{\bm{\alpha}} = z_1^{\alpha_1}\cdots z_d^{\alpha_d}$ and the degree of $\bm{z}^{\bm{\alpha}}$ is defined by $|\bm{\alpha}| = a_1+\dots+a_d$. From now on, without confusion, $|\cdot|$ operating on a \black{multi-index} $\ba$ denotes the $\ell_1$ norm of $\ba$ while $|\cdot|$ operating on a set $T$ denotes the cardinality of $T$.
The degree of a polynomial is defined by the largest degree of its monomial terms.
Then the space of polynomials of degree at most $p$ is defined by
\begin{equation}\label{def:P_n^d}
    \Pi_p^d :=  \textrm{span}\{\bm{z}^{\bm{\alpha}}:|\bm{\alpha}|\leq p, \bm{\alpha} \in \mathbb{N}_0^d \}
    \textrm{ and }
    \dim \Pi_p^d = \left(\begin{array}{c} p+d\\ p\end{array}\right).
\end{equation} 
If we equip $\mathbb{R}^d$ with a probability measure $\rho$, then we can define an inner product on $\Pi^d$,
\begin{equation} \label{def:inner_prod}
    \langle f,g \rangle_{\rho} = \int_{\mathbb{R}^d} f g \dif \rho \qquad f,g\in \Pi^d.
\end{equation}
$f$ and $g$ are said to be orthogonal with respect to $\rho$ if $\langle f,g \rangle_{\rho} = 0$.
Given such an inner product, and an order of the set $\mathbb{N}_0^d$, we can apply the Gram-Schmidt process on the ordered set $\{\bm{z}^{\bm{\alpha}}:\bm{\alpha}\in \mathbb{N}_0^d\}$ to generate a sequence of orthogonal polynomials.
We will revisit this construction in Section \ref{sec:data_driven_basis}.
When $d>1$, there is no natural order among monomials.
As a result, multivariate orthogonal polynomials are, in general, not uniquely determined.
In this paper, we choose the \emph{graded lexicographic order} when applying the Gram-Schmidt process, that is, $\bm{z}^{\bm{\alpha}}\succ \bm{z}^{\bm{\beta}}$ if $|\bm{\alpha}|>|\bm{\beta}|$ or if $|\bm{\alpha}| = |\bm{\beta}|$ and the first nonzero entry in the difference $\bm{\alpha}-\bm{\beta}$ is positive. 

When a simulation model is expensive to run, building an approximation of the response of the model output with respect to the variations in the model input can often be an efficient approach to quantify uncertainty propagation.
The polynomial approximation of a function (model) $f(\bm{z}): \mathbb{R}^d\rightarrow\mathbb{R}, d\geq1$ 
where $\bm{z}= (z_1, \dots, z_d)$ $:\Omega\rightarrow \mathbb{R}^d$ is a $d$-dimensional random variable 
with associated probability measure $\rho(\bm{z})$, which is widely used due to its fast convergence when $f(\bm{z})$ is analytic.
In this paper, we will approximate $f$ using an orthogonal polynomial basis.
It is a generalization of the gPC expansion which usually deals with i.i.d. random variables. 

Let $\Psi = \{\psi_{\bm{\alpha}}(\bm{z}):\bm{\alpha}\in \mathbb{N}_0^d\}$ be a set of orthonormal polynomial basis of $\Pi^d$ associated with the measure $\rho(\bm{z})$, that is,
\begin{align}\label{eq:orthogonal}
    \int\psi_{\bm{\alpha}}(\bm{z})\psi_{\bm{\beta}} (\bm{z})\dif \rho(\bm{z})=\delta_{\bm{\alpha}\bm{\beta}}, \quad \bm{\alpha},\bm{\beta}\in \mathbb{N}_0^d,
\end{align}
where $\delta_{\bm{\alpha\beta}} := \prod_{i=1}^d \delta_{\alpha_i,\beta_i}$ to be the multi-index Kronecker delta.
Then the $p$th-degree arbitrary orthogonal polynomial expansion $f_p(\bm{z})$ of function $f(\bm{z})$ associated with $\psi$ is defined as,
\begin{equation}\label{eq:f_ap_expan}
    f(\bm{z}) \approx f_p(\bm{z}) := \sum_{\bm{\alpha} \in \Lambda^d_{p}} c_{\bm{\alpha}}\psi_{\bm{\alpha}}(\bm{z}),\quad \Lambda^d_{p}= \left\{\bm{\alpha}\in \mathbb{N}_0^d : |\bm{\alpha}| \leq p \right\},
\end{equation}
\black{where $c_{\bm\alpha}$ is the coefficient 
  to be evaluated}. Using an ordering of the orthonormal 
polynomial basis, we can change \eqref{eq:f_ap_expan} into the following single index version
\begin{align} \label{eq:ape_single}
    f_p(\bm{z})=\sum_{\bm{\alpha}\in \Lambda^d_{p}}c_{\bm{\alpha}}\psi_{\bm{\alpha}}(\bm{z})=\sum_{n=1}^{N}{c}_n \psi_n(\bm{z}),
\end{align}
where \black{$N$ is the total number of basis and is given by}
\begin{align*}
    N = \dim \Pi_p^d =  |\Lambda^d_{p}| = \left( \begin{array}{c} d+p \\ p\end{array}\right).
\end{align*}

\subsection{Compressed sensing} \label{sec:cs}

Compressed sensing is a well-studied and popular approach to find sparse solutions to linear equations \cite{Candes_2005error, Candes_2008Rip, Davies_2010RICLp,Donoho_2006srs}.
\black{In this subsection, we briefly review the theory of \ac{CS} and discuss the conditions
which allow accurate recovery of solutions to underdetermined linear system.}

Under certain assumptions, the solution---or its approximation---can be found 
by the well-studied $\ell_1$ minimization, i.e., finding the minimizer
\begin{equation}\label{eq:ape_L1}
    \min \|\bm{c}\|_1 \quad \text{subject to }\bm A\bm{c}=\bm{b},
\end{equation}
where $\bm{A}\in\mathbb{R}^{M\times N}$, $\bm b \in \mathbb{R}^M$ 
and $\|\bm{c}\|_1 = \sum_{i=1}^N|c_i|$ is the $\ell_1$ norm of the vector $\bm{c}$.

When the data $\bm b$ is contaminated by noise, the constraint in \eqref{eq:ape_L1} is relaxed to obtain the basis pursuit denoising problem,
\begin{equation}\label{eq:ape_L1_denoise}
    \min \|\bm{c}\|_1 \quad \text{subject to }\|\bm A\bm{c}-\bm{b}\|_2\leq \sigma,
\end{equation}
where $\sigma$ is an estimate of the $\ell_2$ norm of the noise.
The optimization problems \eqref{eq:ape_L1} and \eqref{eq:ape_L1_denoise} can be solved with efficient algorithms from convex optimization \cite{Berg_2007spgl}.

\black{Next we discuss the conditions for the sparse recovery of $\bm c$.}
\begin{defn} \label{def:s-sparse-vec}
    A vector $\bm{c}$ is said to be $s$-sparse if it has at most $s$ nonzero entries, i.e., $\bm{c}$ is supported on $T \subset \{1,\dots,N\}$ with $|T|\leq s$.
\end{defn}
\begin{defn}[Restricted isometry constant \cite{Candes_2005Decoding, Candes_2006Stablesrec}]\label{def:RIC}
    For each integer $s=1,2,\ldots,\label{black}{N}$ define the isometry constant $\delta_s$ of a matrix $\bm A$ as the smallest number such that
    \begin{equation*}
        (1-\delta_s)\|\bm{c}\|_2^2\leq\|\bm A\bm{c}\|_2^2\leq(1+\delta_s)\|\bm{c}\|_2^2
    \end{equation*}
    holds for any $s$-sparse vector $\bm{c}\in \textcolor{black}{\mathbb{R}^{N}}$. 
\end{defn}
The \acp{RIC} characterizes matrices that are nearly orthonormal. The spare recovery is established by the following theorem. 

\begin{thm}[Sparse Recovery for \ac{RIP}-Matrices]\label{thm:rip_recover}
Let $\bm A\in \mathbb{R}^{M\times N}$. Assume that its isometry constant $\delta_{2s}$ satisfies $\delta_{2s}<0.4931$. 
Let \textcolor{black}{$\bm{c}\in\mathbb{R}^N$}, and assume noisy measurements $\bm{b} = \bm{A}\bm{c} + \bm{\eta}$ are given with $\|\bm{\eta}\|_2\leq \sigma$, then the minimizer $\bm{c}^*$ of 
\begin{equation*}
\min \|\bm{c}\|_1 \quad \textrm{subject to }\|\bm A\bm{c}-\bm{b}\|_2\leq \sigma,
\end{equation*}
satisfies
\begin{equation}\label{eq:error}
\begin{split}
\|\bm{c}-\bm{c}^*\|_2&\leq C_1 \frac{\sigma_s(\bm{c})}{\sqrt{s}}+C_2\sigma,\\
\|\bm{c}-\bm{c}^*\|_1&\leq C_3\sigma_s(\bm{c})+C_4\sqrt{s}\sigma.
\end{split}
\end{equation}
where constants $C_1$, $C_2$, $C_3$ and $C_4$ depend only on $\delta_{2s}$, and $\sigma_s(\bm{c}) = \inf_{{\bm{c}_s:\|\bm{c}_s\|_0\leq s}}\|\bm{c}-\bm{c}_s\|_1$ with $\|\bm{c}_s\|_0$ indicates the number of nonzero entries of $\bm{c}_s$ In particular, if $\bm{c}$ is $s$-sparse, then the reconstruction is exact.
\end{thm}
\begin{proof}
    See Rauhut and Ward \cite{Rauhut_2012sparseLegen}.
\end{proof}

A bounded orthonormal system has the following definition.
{\color{black}\begin{defn}\label{def:bnd_orth_sys}
    $\{\psi_n\}, n = 1, \ldots, N$ is a bounded orthonormal system, if
    \begin{align}\label{def:sys_bnd}
        K:=\sup_{n}\|\psi_n\|_{\infty}=\sup_n\sup_{\bm{z}}|\psi_n(\bm{z})|< \infty,
    \end{align}
\end{defn}
where $K$ is called the basis bound. }

These definitions allow us to establish the recoverability of \eqref{eq:ape_L1} based on the \ac{RIP}.
\begin{thm}[\ac{RIP} for bounded orthonormal systems]\label{thm:boundedBOS}
    Let $\bm A\in\mathbb{R}^{M\times N}$ be the interpolation matrix with entries 
    $\{a_{j,n}=\psi_n(\bm{z}^{(j)})\}_{1\leq n\leq N,1\leq j\leq M}$ (see \eqref{eq:matrixelem}), where $\{\psi_n\}$ 
    is a bounded orthonormal system satisfying \eqref{def:sys_bnd}. Assume that
    \begin{equation*}
        M\geq C\delta^{-2}K^2s\log^3(s)\log(N),
    \end{equation*}
    then with probability at least $1-N^{-\gamma\log^3(s)}$, the RIC $\delta_s$ of $1/\sqrt{M}\bm A$ satisfies $\delta_s\leq\delta$.
    Here, $C,\gamma>0$ are universal constants.
\end{thm}
\begin{proof}
    See Rauhut and Ward \cite{Rauhut_2012sparseLegen}.
\end{proof}

Theorem \ref{thm:rip_recover} and Theorem \ref{thm:boundedBOS} establish the sparse recoverability of the bounded orthonormal systems.

\section{Methods} \label{sec:num_method}
In this section, we introduce the \ac{DSRAR} framework to construct
surrogate model. The goal of this study is to determine, given a small set of $M\ll N$ unstructured realizations 
$\{\bm{z}^{(i)}\}_{i=1}^M$ and the corresponding outputs $b=(f(\bm{z}^{(1)}),...f(\bm{z}^{(M)}))^{T},$ the polynomial 
approximation in \eqref{eq:f_ap_expan} or \eqref{eq:ape_single} when $f(\bm{z})$ has a sparse representation. 
{\color{black}This small set $\{\bm{z}^{(i)}\}_{i=1}^M$ is usually called 
\emph{training set} and $M$ is the \emph{training sample size.}}
There are two quantities we need to compute: (i) an appropriate orthonormal polynomial basis $\psi$ and (ii) an interpolation-type sparse solution $\bm{c} = (c_1,\dots,c_N)^T\in \mathbb{R}^N$ such that $f_p(\bm{z}^{(i)}) = f(\bm{z}^{(i)})$ for $i = 1,\dots,M$ with the smallest possible number nonzero $\bm{c}$.
The basis construction, step (i), will be discussed in detail in Section \ref{sec:data_driven_basis}.
We can reformulate the second part as the following constrained optimization problem,
\begin{equation}\label{eq:ape_L0}
    \min \|\bm{c}\|_0 \quad \text{subject to }\bm A\bm{c}=\bm{b},
\end{equation}
where $\|\bm{c}\|_0$ indicates the number of nonzero entries of $\bm{c}$ and $\bm A \in \mathbb{R}^{M \times N}$ (usually called the measurement matrix) is written as
\begin{eqnarray}\label{eq:matrixelem}
    \bm A=(a_{ij})_{1\leq i\leq M,1\leq j\leq N}, \quad a_{ij}=\psi_{j}(\bm{z}^{(i)}).
\end{eqnarray}
It is well known that this $\ell_0$ minimization problem \eqref{eq:ape_L0} is NP-hard \cite{Natarajan_1995L0}. 
As mentioned in Section~\ref{sec:cs}, \ac{CS} is a well-studied and popular approach to find sparse solutions to \eqref{eq:ape_L0} through $\ell_1$-minimization shown in \eqref{eq:ape_L1} (no noise) or \eqref{eq:ape_L1_denoise} (with noise).
Therefore, \black{the approach introduced below can be viewed} as a method for 
data-driven construction of bases that allow sparse representation and accurate 
recovery for \acp{QoI} in \ac{UQ} applications.

\subsection{Data-driven construction of the \ac{amdP} basis}\label{sec:data_driven_basis}
Let us start with a set of \black{samples of $d$-dimensional} random vector $\bx \in \mathbb{R}^d$, i.e., $S := \left\{\bx^{(k)}\right\}_{k = 1}^{N_s}$ with the underlying probability measure $\rho(\bx).$ $S$ is usually called the \emph{sample set}.
We aim to construct a set of orthonormal polynomial basis functions $\left\{ \psi_{\ba}(\bx)\right\}_{\vert \ba \vert = 0}^p$ with respect to  $\rho(\bx)$ in $\Pi_p^d,$ the space of polynomials up to degree $p$.
Since $\rho(\bx)$ can be non-Gaussian or even unknown, \emph{we do not make the assumption that {\color{black} each component of $\bx$ is mutually} independent, even under a linear transformation such as those based on \ac{PCA}}.
Consequently, the orthogonal polynomial basis $\psi_{\bm\alpha}(\bx)$ cannot 
be directly constructed as a tensor product 
of univariate {\color{black}orthonormal basis functions in each 
component of $\bx$}. \black{Below, we introduce a data-driven approach 
to construct multivariate \ac{amdP} randomness}.

\subsubsection{Orthonormal basis}
When we have a collection of random samples $S$, and the 
underlying probability measure $\rho(\bx)$ can be approximated by the 
discrete measure $\nu_S(\bx)$
\begin{equation}\label{def:disc_meas}
  \rho(\bx) \approx \displaystyle \nu_S(\bx) := \frac{1}{N_s}\sum_{\bx^{(k)} \in S}\delta_{\bx^{(k)}}(\bx),
\end{equation}
where $\delta_{\bx^{(k)}}$ is the Dirac measure, that is $\delta_{\bx^{(k)}}(\bx)$ is equal to 1 when $\bx = \bx^{(k)}$ and 0 otherwise.
Given the inner product defined as in \eqref{def:inner_prod} with $\rho$ replaced by the discrete measure $\nu_{S},$ we 
can construct a set of orthonormal multivariate polynomial basis functions $\left\{\psi_{\ba}(\bx)\right\}_{\vert \ba \vert = 0}^{p}$ 
via the Gram-Schmidt orthogonalization process on an ordered monomial basis $\{\hat{\psi}_{\bm{\alpha}}(\bx)\}_{|\ba|=0}^p$. 
Here, we use the aforementioned {graded lexicographic ordering} of the multi-index.

Similar to Dunkl and Xu \cite{dunkl_xu_2014}, $\psi_{\ba}$ can be constructed using the recursive formulation 
\begin{equation}
  \psi_{\bm{\alpha}}(\bx) = f_{\bm{\alpha}}^{\bm{\alpha}} \hat{\psi}_{\ba}(\bx) - \sum_{\bm{\beta} \prec \bm{\alpha}} f_{\bm{\beta}}^{\bm{\alpha}} \psi_{\bm{\beta}}(\bx),
  \label{eq:exact_orth}
\end{equation}
where $\hat{\psi}_{\ba}(\bx) := \prod_{k=1}^{d} \xi^{\alpha_k}_k$ represents the multivariate monomial basis function.
The expression $\bm{\beta} \prec \ba$ means that the multi-index $\bm{\beta}$ comes before $\ba$ under the chosen ordering.   
The coefficients $f_{\bm{\beta}}^{\bm{\alpha}}$ are determined by imposing an orthonormal condition with respect to the discrete measure $\nu_{S}$, i.e.,
\begin{equation}
  \begin{split}
    \int \psi_{\bm{\alpha}}(\bm{\xi}) \psi_{\bm{\beta}}(\bm{\xi}) \dif \rho(\bx)  &\approx \int \psi_{\bm{\alpha}}(\bm{\xi}) \psi_{\bm{\beta}}(\bm{\xi}) \dif \nu_{S}(\bx) \\
    &= \frac{1}{N_s}\sum_{k=1}^{N_s} \psi_{\bm{\alpha}}(\bm{\xi}^{(k)})\psi_{\bm{\beta}}(\bm{\xi}^{(k)}) \\
    &  \equiv \delta_{\bm{\alpha,\beta}}, \quad\quad \bm{\beta} \preceq \bm{\alpha}.
  \end{split} \label{eq:exact_orth_coeff}
\end{equation} 
Equations \eqref{eq:exact_orth} and \eqref{eq:exact_orth_coeff} generate a set of orthonormal basis 
functions on the discrete measure 
$\nu_{S}$ \black{irrespective of the mutual dependence between the components of $\bx$.}
\black{We employ $\left\{\psi_{\ba}(\bx)\right\}_{\vert \ba \vert = 0}^{p}$ as 
  the \ac{amdP} basis on $\rho(\bx)$.   
} 
\textcolor{black}{Numerically, the modified Gram-Schmidt orthogonalization can be 
used as an alternative approach when the number of basis is too large and there exists 
instability in the standard Gram-Schmidt orthogonalization.}

When $\rho(\bx)$ is known explicitly, orthonormal basis functions can also be constructed 
by taking the general formulation in Equation \eqref{eq:exact_orth} and 
imposing the inner product in Equation \eqref{def:inner_prod} with respect to $\rho$.
Here we will also consider a special case when $\bx$ is a random vector that is 
linearly transformed from {\color{black} a random vector $\bm z$ with i.i.d.\ components} via $\bx = \mathbf{Q}\bm z.$ 
This case is motivated by the sparsity enhancement approach discussed in 
Sec.\ \ref{sec:sparsity_enhancement}. In particular, we assume that the quadrature rule 
of polynomial integration
with respect to the probability measure of $\bm z$ is explicitly known. 

Given these assumptions, $f_{\bm{\beta}}^{\bm{\alpha}}$ can be determined by \textcolor{black}{Equation \eqref{eq:exact_orth} and the orthonormal condition}
\begin{equation}
  \begin{split}
    \int \psi_{\bm{\alpha}}(\bm{\xi}) \psi_{\bm{\beta}}(\bm{\xi}) \dif \rho(\bx) &= \sum_{k=1}^{N_Q} \psi_{\bm{\alpha}}\left(\mb Q \bm{z}_Q^{(k)}\right) \psi_{\bm{\beta}}\left(\mb Q \bm{z}_Q^{(k)}\right) w_k \\
    &  \equiv \delta_{\bm{\alpha,\beta}}, \quad\quad \bm{\beta} \preceq \bm{\alpha}.
  \end{split} \label{eq:exact_orth_coeff_quad}
\end{equation}
where $\left\{\bm z_Q^{(k)}\right\}_{k=1}^{N_Q}$ and $\left\{w_k\right\}_{k=1}^{N_Q}$ represent the quadrature points and weights constructed to yield an exact integration with probability measure of $\bm z$ for polynomials of degree $\vert \ba\vert + \vert \bm\beta\vert$ or less. 
%-------------------------------------------------------------------------------
\begin{algorithm}[htp]
  \hrule
  \caption{Construct the orthonormal \black{\ac{amdP}} basis $\left\{ \psi_{\ba}(\bx)\right\}_{\vert \ba \vert = 0}^p$ on discrete sample set $S$.}
  \vspace{5pt} \hrule \vspace{5pt}
\begin{algorithmic}[1]
    \State Given sample set  $S = \left\{\bx^{(k)}\right\}_{k = 1}^{N_s}$.
    \State Given a fixed multi-index order $\left\{\ba^{(l)}\right\}_{l=1}^N$.
    \FOR{$l=1$ to $N$}
\State Let $\ba = \ba^{(l)}$, construct $\displaystyle \psi_{\bm{\alpha}}(\bx) = f_{\bm{\alpha}}^{\bm{\alpha}} \hat{\psi}_{\ba}(\bx) - \mathop{\sum}_{\bm{\beta} \prec \bm{\alpha}} f_{\bm{\beta}}^{\bm{\alpha}} \psi_{\bm{\beta}}(\bx)$ subject to Equation \eqref{eq:exact_orth_coeff}.
\ENDFOR
\hrule
\end{algorithmic}
  %\hrule
\label{alg:orth_discrete}
\end{algorithm}
%-------------------------------------------------------------------------------

%-------------------------------------------------------------------------------
\begin{algorithm}[htp]
  \hrule
  \caption{Construct the orthonormal \black{\ac{amdP}} basis $\left\{ \psi_{\ba}(\bx)\right\}_{\vert \ba \vert = 0}^p$ with probability measure $\rho(\ba)$.}
  \vspace{5pt} \hrule \vspace{5pt}
  \begin{algorithmic}[1]
    \State Given a multi-index order $\left\{\ba^{(l)}\right\}_{l=1}^N$.
    \FOR{$l=1$ to $N$} 
 
    \State Let $\ba = \ba^{(l)}$, construct $\displaystyle \psi_{\bm{\alpha}}(\bx) = f_{\bm{\alpha}}^{\bm{\alpha}} \hat{\psi}_{\ba}(\bx) - \mathop{\sum}_{\bm{\beta} \prec \bm{\alpha}} f_{\bm{\beta}}^{\bm{\alpha}} \psi_{\bm{\beta}}(\bx)$ by evaluating the basis inner product using existing quadrature rule or Equation \eqref{eq:exact_orth_coeff_quad} if $\bx$ can be linearly transformed from {\color{black} a  random vector with i.i.d.\ components} $\bm z$  with an explicitly known quadrature rule.
    \ENDFOR
    \end{algorithmic} 
   \hrule
  \label{alg:orth_density}
\end{algorithm}
%-------------------------------------------------------------------------------
Algorithms \ref{alg:orth_discrete} and \ref{alg:orth_density} summarize the procedure of orthonormal basis construction when $\rho(\bx)$ is implicitly represented by a sample set $S$ and known explicitly, respectively.
There is no unique system of orthogonal polynomial basis functions for both scenarios if $d > 1$; different orderings of $\ba$ lead to different orthogonal basis \cite{dunkl_xu_2014}.
On the other hand, the constructed orthonormal basis is unique up to unitary transformations as we prove in Theorem \ref{thm:orth_basis}.

\begin{thm}
  Let $\left\{\psi_{\ba}(\bx)\right\}_{\vert \ba \vert = 0}^{p}$ be a set of orthonormal polynomial basis  with respect to the measure $\rho(\bx)$, $\bx \in \mathbb{R}^d$.
  Denote by $\bm{\varPsi}(\bx)$ the polynomial basis vector
  \begin{equation}
    \bm{\varPsi}(\bx) := (\psi_{\ba^{(1)}}, \cdots, \psi_{\ba^{(N)}})^T,
    \label{eq:poly_basis_vector}
  \end{equation}
  where $\ba^{(1)}, \cdots, \ba^{(N)}$ is the arrangement of multi-index $\ba$ according to a fixed multi-index order.
  Let $\bm\chi = \mathbf{Q}\bx$, where $\mathbf{Q}\in \mathbb{R}^{d\times d}$ is invertible. Let  $\left\{\phi_{\bm\beta}(\bm\chi)\right\}_{\vert \bm\beta \vert = 0}^{p}$ be a set of orthonormal polynomial basis functions with respect to a measure $\rho'(\bm\chi)$ constructed with order $\bm\beta^{(1)},\cdots,\bm\beta^{(N)}$, where $\rho'(\bm \chi)$ is induced from $\rho(\bx)$.
  Then there exists a unitary matrix $\mathbf{U}$ such that $\bm{\varPhi}(\bm\chi) = \mathbf{U}  \bm{\varPsi}(\bx)$, where $\bm{\varPhi}(\bm\chi):=(\phi_{\bm\beta^{(1)}}, \cdots, \phi_{\bm\beta^{(N)}})^T$ denotes the corresponding polynomial basis vector.
  \label{thm:orth_basis}
\end{thm}
\begin{proof}
  Let $\hat{\bm{\varPsi}}(\bx)$ be the monomial basis vector.
  Note that $\left\{\psi_{\ba}(\bx)\right\}_{\vert \ba \vert = 0}^{p}$ and $\left\{\phi_{\bm\beta}(\bm\bx)\right\}_{\vert \bm\beta \vert = 0}^{p}$  are two sets of basis in $\Pi_{p}^d$.
  There exists transfer matrices $\mathcal{M_{\psi}}$ and $\mathcal{M_{\phi}}\in \mathbb{R}^{N\times N}$ such that
  \begin{equation}
    \bm{\varPsi}(\bm\xi)= \mathcal{M}_{\psi} \hat{\bm{\varPsi}}(\bx), \quad \bm{\varPhi}(\bm\xi)= \mathcal{M}_{\phi} \hat{\bm{\varPsi}}(\bx). \nonumber
  \end{equation}
  With $\bm\chi = \mathbf{Q}\bx$, $\bm\varPhi(\mathbf{Q}\bx)$ is also a basis in $\Pi_p^d$.
  Then there exists an invertible matrix $\mathcal{T}\in\mathbb{R}^{N\times N}$ such that
  \begin{equation}
    \bm{\varPhi}(\bm\chi) = \bm{\varPhi}(\mathbf{Q}\bx)= \mathcal{T} \hat{\bm{\varPsi}}(\bx) \nonumber,
  \end{equation}
  which gives $\bm{\varPhi}(\bm\chi) = \mathcal{U} \bm{\varPsi}(\bx)$, where $\mathcal{U} = \mathcal{T} \mathcal{M}_{\psi}^{-1}$.
  Recall $\left\{\psi_{\ba}(\bx)\right\}_{\vert \ba \vert = 0}^{p}$ and $\left\{\phi_{\bm\beta}(\bm\chi)\right\}_{\vert \bm\beta \vert = 0}^{p}$ are orthonormal basis with respect to $\rho(\bx)$ and $\rho'(\bm\chi)$, we have
  \begin{equation}
    \mathbf{I} = \int \bm{\varPhi}(\bm\chi) \bm{\varPhi}(\bm\chi)^T \dif \rho'(\bm\chi) = \int \mathcal{U}\bm{\varPsi}(\bm\xi) \bm{\varPsi}(\bm\xi)^T  \mathcal{U}^T  \dif \rho(\bm\xi) = \mathcal{U} \mathcal{U}^T
  \end{equation}
\end{proof}
We do not need further assumptions on $\rho(\bx)$ because Theorem~\ref{thm:orth_basis} holds both when $\rho(\bx)$  is a measure on the continuous random vector $\bx$ (with probability density function $\omega(\bx)$) or a discrete measure $\nu_S(\bx)$ on a sample set $S$.
Furthermore, it is straightforward to show the following Corollary.
\begin{cor}
  Let $S_1 := \left\{\bx^{(k)}\right\}_{k = 1}^{M}$ and $S_2 := \left\{\bm\chi^{(k)}\right\}_{k = 1}^{M}$ be two sets of random sampling points where $\bm\chi^{(k)} = \mathbf{Q}\bm\xi^{(k)}$ with invertible $\mathbf{Q}$.
  Let $\mathbf{G}_{\bx}$ and $\mathbf{G}_{\bm\chi}$ be the Gram matrix constructed by $\bm{\varPsi}(\bx)$ and $\bm{\varPhi}(\bm\chi)$ defined in Theorem~\ref{thm:orth_basis}, i.e.,
  %$\mathbf{G}_{\bx}$ and $\mathbf{G}_{\bm\chi}$
  %\begin{equation}
    $\displaystyle \mathbf{G}_{\bx} := \sum_{k=1}^{M} \bm{\varPsi}(\bx^{(k)}) \bm{\varPsi}(\bx^{(k)})^T / M$ and $\displaystyle \mathbf{G}_{\bm\chi} := \sum_{k=1}^{M} \bm{\varPhi}(\bm\chi^{(k)}) \bm{\varPhi}(\bm\chi^{(k)})^T / M$.
  %\end{equation}
  Then $\mathbf{G}_{\cdot}$ has invariant $l_2$ norm, that is, $\Vert \mathbf{G}_{\bx}\Vert_2 = \Vert \mathbf{G}_{\bm\chi}\Vert_2$. \label{col:orth_basis}
  Moreover, $\Vert \mathbf{G}_{\bx}-\mathbf{I}\Vert_2 = \Vert \mathbf{G}_{\bm\chi}-\mathbf{I}\Vert_2$.
\end{cor}

%Thm. \ref{thm:orth_basis} and Col. \ref{col:orth_basis} indicate that 
In general, the $l_2$ norm of $\Vert \mathbf{G}_{\bx}-\mathbf{I}\Vert_2$ is independent of specific monomial order of $\ba$ and invariant under linear transformations of the random vector.
The basis functions $\left\{\psi_{\ba}(\bx)\right\}_{\vert \ba \vert = 0}^{p}$ constructed by Equations \eqref{eq:exact_orth} and \eqref{eq:exact_orth_coeff} provide an appropriate candidate for representing the surrogate model $f(\bx)$ via \ac{CS}.

\subsubsection{Near-orthonormal basis}

When $\rho(\bx)$ is implicitly represented by a sample set $S$, we employ the discrete measure $\nu_S$ to construct $\left\{\psi_{\ba}(\bx)\right\}_{\vert \bm\beta \vert = 0}^{p}$. 
However, we note that the training set that queries $f(\cdot)$, denoted by $S_f$, may not be a subset of $S$.
In practice, the sample set $S$ and the training set $S_f$ are usually collected in a sequential manner or directly from different experiments, although individual sampling points of both $S$ and $S_f$ follow the same distribution.
Since $S$ only contains a finite number of samples of $\bx$, basis $\left\{\psi_{\ba}(\bx)\right\}_{\vert \ba \vert = 0}^{p}$  constructed by \eqref{eq:exact_orth} and \eqref{eq:exact_orth_coeff} is not the ``exact orthonormal'' basis with respect to $\rho(\bx)$.
Especially, let $S' = \left\{{\bx'}_k\right\}_{k = 1}^{N_s}$ be another sample set following the same distribution $\rho(\cdot)$ and $\nu_{S'}(\cdot)$ be the discrete measure defined on $S'$.
For the orthonormal \black{\ac{amdP}} basis functions $\left\{\psi_{\ba}(\bx)\right\}_{\vert \ba \vert = 0}^{p}$ 
constructed on $S$, we have $\mathbb{E}\left[ \bm{\varPsi}(\bx) \bm{\varPsi}(\bx)^T \right] \neq \mathbf{I}$ 
under the discrete measure $\nu_{S'}(\bx)$ and vice versa.

The above observation forces us to re-examine the orthonormal condition imposed by \eqref{eq:exact_orth_coeff}.
Since the pre-constructed basis $\psi_{\ba}(\bx)$ does not retain the exact orthonormal condition when later being applied to approximate $f(\bx)$, we may relax the condition when determining the coefficients $f_{\bm\beta}^{\ba}$ in \eqref{eq:exact_orth}.
In the present study, we propose the following heuristic criterion 
\begin{equation}
  \arg\min_{\hat{\mathbf{f}}^{\bm\alpha}} \Vert \hat{\mathbf{f}}^{\bm\alpha}\Vert_2~~\mbox{subject to}~~ \left \vert \int \psi_{\bm{\alpha}}(\bm{\xi}) \psi_{\bm{\beta}}(\bm{\xi}) \dif \nu_{S}(\bx) - \delta_{\bm{\alpha,\beta}}\right \vert < \zeta_{\ba,\bm\beta},~~ \bm{\beta} \le \bm{\alpha},
  \label{eq:near_orth_condition}
\end{equation}
where $\hat{\mathbf{f}}^{\bm\alpha}$ is the coefficient vector of $\psi_{\bm\alpha}$ when represented using monomial basis functions, i.e., $\displaystyle \psi_{\bm\alpha}(\bx) = \sum_{\bm\beta \le \bm\alpha} \hat{f}_{\bm\beta}^{\bm\alpha} \hat{\psi}_{\bm\beta}(\bx)$.  $\hat{\mathbf{f}}^{\bm\alpha}$ is related to $\mathbf{f}^{\ba}$ through the linear transformation
\begin{equation}
  \hat{\mathbf{f}}^{\bm\alpha} =
  \begin{pmatrix}
    \mathbf{F} & 0 \\
    0 & 1
  \end{pmatrix}
  \mathbf{f}^{\ba},
  \label{eq:coeff_transform}  
\end{equation}
where $\mathbf{F}$ is an upper triangle matrix determined by pre-computed $\hat{\mathbf{f}}^{\bm\beta}, \bm\beta \prec \ba$, i.e.,
\begin{equation}
  \left[\mathbf{F}\right]_{I_{\bm\beta'} I_{\bm\beta}} = 
  \begin{cases}
    \hat{f}_{\bm\beta'}^{\bm\beta} \quad &\bm\beta' \preceq \bm\beta \\
    0 \quad  &\bm\beta' \succ \bm\beta,
  \end{cases}
\end{equation}
where $I_{\bm\beta}$ represents the mapping from multi-index to single index.

The parameter $\zeta_{\ba,\bm\beta}$ quantifies the relaxation of the orthonormal condition. We split the sample set $S$ equally into two parts $S:=S_1\cup S_2$.
Denote $\left\{\psi^{\left(1\right)}_{\ba}(\bx)\right\}_{\vert \ba \vert = 0}^{p}$ and $\left\{\psi^{\left(2\right)}_{\ba}(\bx)\right\}_{\vert \ba \vert = 0}^{p}$ the orthonormal bases constructed by Equations \eqref{eq:exact_orth} and \eqref{eq:exact_orth_coeff} on the discrete measures $\nu_{S_1}(\bx)$ and $\nu_{S_2}(\bx)$, respectively. 
Inspired by cross-validation, we have chosen 
$\displaystyle \zeta_{\bm\alpha,\bm\beta} = \frac{\vert \zeta_1\vert + \vert \zeta_2\vert }{2\sqrt{2}}$%, where $\zeta_1$ and $\zeta_2$ are determined by
\begin{equation}
  \zeta_1 = \int \psi^{(1)}_{\bm{\alpha}}(\bm{\xi}) \psi^{(1)}_{\bm{\beta}}(\bm{\xi}) \dif\nu_{S_2}(\bx),
  \quad
  \zeta_2 = \int \psi^{(2)}_{\bm{\alpha}}(\bm{\xi}) \psi^{(2)}_{\bm{\beta}}(\bm{\xi}) \dif \nu_{S_1}(\bx).
  \label{eq:zeta_cross_validation}
\end{equation}
%\textcolor{black}{Here the $\sqrt{2}$ in the denominator is caused by the relation that $|S| = 2|S_2|$. Consider a special case, that we construct exact orthogonal basis on set $S1$ and employ this basis to evaluate $\zeta$ on $S_1$ $S_2$ and the entire set $S$, then by \eqref{eq:zeta_cross_validation} $\zeta_{S_1} = 0$ and $\zeta_{S_2} = \zeta_1$ and $\zeta_{S} = \frac{\zeta_1}{2}$ 
%-------------------------------------------------------------------------------

\begin{algorithm}
  \hrule
  \caption{Construct the near-orthonormal \black{\ac{amdP}} basis $\left\{ \psi_{\ba}(\bx)\right\}_{\vert \ba \vert = 0}^p$ on discrete sample set $S$.}
  \vspace{5pt} \hrule \vspace{5pt}
  \begin{algorithmic}[1]
    \State Collect samples of $\bx$ from sample set  $S = \left\{\bx^{(k)}\right\}_{k = 1}^{N_s}$, split $S$ equally into two disjoint subsets, i.e., $S = S_1\cup S_2$, 
    $S_1 \cap S_2 = \O $. 
    \State Given fixed monomial index order $\left\{\ba^{(l)}\right\}_{l=1}^N$, construct the orthonormal \black{\ac{amdP}} basis $\left\{\psi^{\left(1\right)}_{\ba}(\bx)\right\}_{\vert \ba \vert = 0}^{p}$ and $\left\{\psi^{\left(2\right)}_{\ba}(\bx)\right\}_{\vert \ba \vert = 0}^{p}$ on set $S_1$ and $S_2$ by \textbf{Algorithm 1}.
    \FOR{$l=1$ to $N$}
      \State Let $\ba = \ba^{(l)}$, construct $\displaystyle \psi_{\bm{\alpha}}(\bx) = f_{\bm{\alpha}}^{\bm{\alpha}} \hat{\psi}_{\ba}(\bx) - \mathop{\sum}_{\bm{\beta} \prec \bm{\alpha}} f_{\bm{\beta}}^{\bm{\alpha}} \psi_{\bm{\beta}}(\bx)$ on by Equations \eqref{eq:near_orth_condition}, \eqref{eq:coeff_transform}, and \eqref{eq:zeta_cross_validation}.
    \ENDFOR   
  \end{algorithmic}
  \hrule
  \label{alg:near_orth_discrete}
\end{algorithm}
%-------------------------------------------------------------------------------

Algorithm \ref{alg:near_orth_discrete} describes construction for a set of near-orthonormal \black{\ac{amdP}} 
basis functions on the sample set $S$. 
When applied to the sample set $S'$ to approximate $f(\bx)$, the basis shows comparable orthonormal conditions with the basis constructed by \eqref{eq:exact_orth_coeff}. 
Such results can be partially understood by the theoretical bound from Theorem~\ref{thm:boundedBOS} on the number of samples $M$ for exact recovery in orthonormal polynomial systems, $\displaystyle M \ge C_1 K^2 s \log^3(s) \log(N)$, where $\displaystyle s = \Vert \bm{c}\Vert_0$ and $\displaystyle K = \sup_{\ba}\Vert \psi_{\ba}\Vert_{\infty}$.
Theoretical analysis of the recovery error under different basis functions is out of the scope of the present work and is left for future investigation.
However, we note that the accuracy of the surrogate model $f(\bx)$ can be further improved by enhancing the sparsity of $\bm c.$
This can be achieved through the ideas presented in our previous work \cite{Lei_Yang_MMS_2015} which will be extended to general distributions below.

\begin{rem}
  We emphasize that \eqref{eq:near_orth_condition} provides a heuristic approach to construct the near-orthonormal 
\black{\ac{amdP}} basis functions $\psi_{\ba}(\bx)$ with {\color{black}a} smaller basis bound.
  In practice, \eqref{eq:near_orth_condition} can be further relaxed to
  \begin{equation}
    \begin{split}
      \arg\min_{\hat{\mathbf{f}}^{\bm\alpha}} \Vert \hat{\mathbf{f}}^{\bm\alpha}\Vert_2~~\mbox{subject to}~~ \sum_{\vert\bm\beta\vert = r, \bm\beta < \ba} \left \vert \int \psi_{\bm{\alpha}}(\bm{\xi}) \psi_{\bm{\beta}}(\bm{\xi}) \dif \nu_{S}(\bx) \right \vert^2 < \sum_{\vert\bm\beta\vert = r, \bm\beta < \ba} \zeta_{\ba,\bm\beta}^2, \\
      \left \vert \int \psi_{\bm{\alpha}}(\bm{\xi}) \psi_{\ba}(\bm{\xi}) \dif \nu_{S}(\bx) - 1 \right \vert < \zeta_{\ba,\ba},~~ r = 0,\cdots, \vert\ba\vert,
    \end{split} \label{eq:near_orth_condition_approx}
  \end{equation}
which shows similar numerical performance.
There is no theoretical guarantee yet that Equations \eqref{eq:near_orth_condition} and \eqref{eq:near_orth_condition_approx} 
yield a smaller basis bound than \eqref{eq:exact_orth_coeff} on ${S_f}$, $S$ or the entire domain of $\bx$.
We numerically compare some properties of different bases in Section~\ref{sec:basis_comparison}, which illustrate the 
performance of the near-orthonormal \black{\ac{amdP}} basis constructed above.
There may exist other numerical approaches to optimize $\psi_{\ba}(\bx)$ that can lead to an even smaller basis bound. 
\textcolor{black}{We also note that the threshold $\zeta_{\bm{\alpha},\bm{\beta}}$ is determined by directly splitting $S$ into
two disjoint sets. In practice, it is possible to design more sophisticated strategies to optimize the choice 
of $ \zeta_{\bm{\alpha},\bm{\beta}}$ and the basis construction procedure.} We leave such studies for future work.
\end{rem}

\subsection{Sparsity enhancement} \label{sec:sparsity_enhancement}

For the linear system in \eqref{eq:ape_L1}, the numerical accuracy of the recovered $\bm\tilde{c}$ via $l_1$-minimization depends on the sparsity of $\bm c$.
This dependence motivates us to develop a numerical approach to further enhance the sparsity of $\bm c$ through the variability analysis of $f(\bx)$ \cite{Lei_Yang_MMS_2015}.
If we know $f(\bx)$ explicitly, the (sorted) directions of variance in $f(\bx)$ under the distribution of $\bx$ can be found 
based on the active subspace method \cite{Russi_PHD_2001,ConstantineDW14}.
In particular, we define the gradient matrix $\mathbf{G}$ by
\begin{equation}
	\mathbf{G} = \mathbb{E} \left[\nabla f(\bx) {\nabla f(\bx)}^T\right]
\end{equation}
where $\nabla f(\bx)$ is the gradient vector defined by $\nabla f(\bx) = \left(\frac{\partial f}{\partial \xi_1}, \frac{\partial f}{\partial \xi_2}, \cdots \frac{\partial f}{\partial \xi_d}\right)^{T}$.
Eigendecomposition of $\mathbf{G}$,
\begin{subequations}
  \begin{equation}
    \mathbf{G} = \mathbf{Q} \mathbf{K} \mathbf{Q}^{T}, ~~~~~ \mathbf{Q} = \left [{\mathbf q}_1 ~\mathbf{q}_2 \cdots ~\mathbf{q}_{d} \right ],
  \end{equation}
  \begin{equation}
    \mathbf{K} = {\rm diag}(k_1, \cdots , k_{d}), ~~~k_1 \ge \cdots \ge k_d \ge 0,
  \end{equation}
\end{subequations}
yields the sorted variability directions $\mathbf{q}_1, \mathbf{q}_2, \cdots, \mathbf{q}_d$.
Accordingly, we may define a new random vector $\bm \chi$ following the sorted variability directions via linear transformation
\begin{equation}
	{\bm \chi} = \mathbf{Q}^T \bx.
	\label{eq:chi_def}
\end{equation}
{\color{black} $f(\bx) = f((\mathbf{Q}^T)^{-1}\bm\chi) = f(\mathbf{Q}\bm\chi)$  can be approximated by expansion in an orthonormal polynomial basis $\bm\chi$ with a coefficient vector $\bm c$ which is sparser than the $f(\bx)$ being expanded by orthonormal basis of $\bx$. For the remainder of this paper, we use $\mathbf{Q}$ to denote the rotation matrix to transform $\bx$ to $\bm\chi$ and $g(\bm\chi)$ to represent $f(\mathbf{Q}\bm\chi)$.}

In practice, $f(\bx)$ is usually not explicitly known.
We may numerically approximate $\mb G$ by 
\begin{equation}
  \mathbf{G} \approx \mathbb{E} \left[ \nabla \tilde{f} (\bx) {\nabla \tilde{f} (\bx)}^T\right],
  \label{eq:G_approx}  
\end{equation}
where $\tilde{f}(\bx)$ represents the approximation of $f(\bx)$ by the orthonormal polynomial basis functions $\psi_{\ba}(\bx)$ as proposed in \cite{Lei_Yang_MMS_2015} or obtained via solving \eqref{eq:ape_L1} with the data-driven basis approach (i.e., basis functions constructed with respect to an arbitrary measure) described in Section~\ref{sec:data_driven_basis}.
In particular, if $\bx$ {\color{black}is a random vector with i.i.d.\ Gaussian components, $\bm\chi$ is also 
a random vector with i.i.d.\ Gaussian components.
Thus, $\tilde{f}(\bx)$ and $\tilde{g}(\bm\chi):=\tilde{f}(\mathbf{Q} \bm\chi)$ can be 
represented by the orthonormal basis functions of the same form, e.g., tensor products of univariate 
Hermite polynomials. Without of lost of generality, from now on, we use $\tilde{g}(\bm\chi)$ to 
represent $\tilde{f}(\mathbf{Q} \bm\chi)$.} 

\emph{However, if $\rho(\bx)$  is not i.i.d.\ Gaussian, $\bm\chi$ and $\bx$ do not generally have the same distribution.
Therefore, an orthonormal polynomial basis $\psi(\cdot)$ with respect to $\bx$ \emph{cannot} be directly applied to $\bm\chi$.}
\black{The general approach presented in Section~\ref{sec:data_driven_basis} enables us to construct
the \ac{amdP} basis with respect to the probability measure of the rotated vector $\bm\chi$.}
The two orthonormal bases associated with $\bx$ and $\bm\chi$ respectively are related to each other via a unitary 
transformation as shown in Theorem~\ref{thm:orth_basis}.
In particular, if $\rho(\bx)$ is implicitly described by a sample set $S = \left\{\bx^{(k)}\right\}_{k = 1}^{N_s}$, $\mb G$ can be easily evaluated by representing $\psi_{\ba}(\bx)$ via the monomial basis, i.e., $\displaystyle \psi_{\bm\alpha}(\bx) = \sum_{\bm\beta \preceq \bm\alpha} \hat{f}_{\bm\beta}^{\bm\alpha} \hat{\psi}_{\bm\beta}(\bx)$ via Equation \eqref{eq:coeff_transform} and then integrating with discrete measure $\nu_S$. 
%, where $\tilde{\mathbf{f}}^{\bm\beta}$ and  $\mathbf{f}^{\bm\beta}$ are available 
%from the construction of $\bm \psi_{\ba}(\bx)$. 
By transforming $S$ and $S_f$ into $\left\{{\bm\chi}^{(k)}\right\}_{k = 1}^{N_s}$ and $\left\{{\bm\chi'}^{(k)}\right\}_{k = 1}^{M}$, the orthonormal and near-orthonormal \black{\ac{amdP}} basis functions with respect to $\bm\chi$ can be constructed by Eqs.~\eqref{eq:exact_orth_coeff} \eqref{eq:near_orth_condition_approx}.
The surrogate model $\tilde{g}(\bm\chi)$ can then be constructed by solving \eqref{eq:ape_L1}.

The entire \ac{DSRAR} procedure is presented in \textbf{Algorithm 4}. 
Compared with $\tilde{f}(\bm\xi)$, $\tilde{g}({\bm\chi})$ shows
smaller numerical error in general. The additional cost of sparsity 
enhancement procedure in Step 4 - 6 is less than $0.6$ CPU (3.7 GHz Quad-Core 
Intel Xeon E5) hour for the numerical examples considered in this study. 
For realistic applications, the overhead of Step 4 - 6 could be 
relatively small if sampling of QoI is expensive or the available training set is limited. 

%-------------------------------------------------------------------------------
\begin{algorithm}
  \hrule
  \caption{\ac{DSRAR}: Surrogate model construction with discrete sample set $S$ and training set $S_f$.}
  \vspace{5pt} \hrule \vspace{5pt}
  \begin{algorithmic}[1]
    \State Collect the sample set within the random space $S = \left\{\bx^{(k)}\right\}_{k = 1}^{N_s}$.
    \State Generate evaluations of $f$ on training set $S_f = \left\{{\bx'}^{(k)}\right\}_{k = 1}^{M}$ with $M$ outputs $f_1, f_2, \cdots, f_M$.
    \State Construct the data-driven \black{\ac{amdP}} basis $\left\{ \psi_{i}(\bx)\right\}_{i = 1}^N$ on discrete measure $\nu_S(\bx)$ as the exact orthonormal basis by Algorithm~\ref{alg:orth_discrete} or the near orthonormal basis by Algorithm~\ref{alg:near_orth_discrete}.
    \State Evaluate the measurement matrix ${\bm A}_{ij} = \psi_{j}({\bx'}^{(i)})$, $1\le i\le M$, $1\le j \le N$; construct surrogate model $\tilde{f}(\bx) = \displaystyle \sum_{\vert \alpha\vert = 0}^{p} c_{\bm \alpha} \psi_{\bm \alpha}(\bx)$ by solving the $l_1$-minimization problem.
    \State Evaluate the gradient matrix $ \mathbf{G} \approx \mathbb{E}\left[\nabla \tilde{f}(\bx) \nabla \tilde{f} (\bx)^T \right ]$ on $\nu_{S}(\bx)$.
    Find the eigendecomposition $\mathbf{G} = \mathbf{Q} \mathbf{K}\mathbf{Q}^T$, define sample set $\left\{\bm{\chi}^{(k)}\right\}_{k=1}^{N_s}$ and training set $\left\{{\bm \chi'}^{(k)}\right\}_{k=1}^{M}$ by $\bm{\chi}^{(k)} = \mathbf{Q}^T \bx^{(k)}$, $\bm{\chi'}^{(k)} = \mathbf{Q}^T {\bx'}^{(k)}$.
    \State Reconstruct the data-driven \black{amdP} basis $\left\{ \phi_{\ba}(\bm\chi)\right\}_{\vert \ba \vert = 0}^p$ by Algorithm~\ref{alg:near_orth_discrete} and surrogate model $\tilde{g}(\bm\chi)$ with enhanced sparsity following Step 3 and Step 4.
    %\State Repeat Step 5 to Step 6 
  \end{algorithmic}
  \hrule
  \label{alg:data_driven}
\end{algorithm}

The \ac{DSRAR} framework described above is also applicable to systems with 
standard density distributions, where $\rho(\bx)$ is known explicitly.
Without loss of generality, we assume that an orthonormal polynomial basis $\left\{ \psi_{\ba}(\bx)\right\}_{\vert \ba \vert = 0}^p$ is known.
Evaluation of $\mb G$ by \eqref{eq:G_approx} on $\rho(\bx)$ is straightforward.
The surrogate model of $f$ can be constructed via $l_1$ minimization with enhanced sparsity through Algorithm~\ref{alg:explicit_density}.

\begin{algorithm}
  \hrule
  \caption{\ac{DSRAR}: Surrogate model construction with training set $S_f$ and probability measure $\rho(\bx)$.}
  \vspace{5pt} \hrule \vspace{5pt}
  \begin{algorithmic}[1]
    \State Evaluate $f$ on training set $S_f = \left\{{\bx'}^{(k)}\right\}_{k = 1}^{M}$ with $M$ outputs $f_1, f_2, \cdots, f_M$.
    \State Evaluate the measurement matrix ${\bm A}_{ij} = \psi_{j}({\bx'}^{(i)})$, $1\le i\le M$, $1\le j \le N$ ; construct surrogate model $\tilde{f}(\bx) = \displaystyle \sum_{\vert \alpha\vert = 0}^{p} c_{\bm \alpha} \psi_{\bm \alpha}(\bx)$ by solving $l_1$ minimization problem.
    \State Evaluate the gradient matrix $ \mathbf{G}  = \mathbb{E}\left[\nabla \tilde{f}(\bx) \nabla \tilde{f} (\bx)^T \right ]$ on $\rho(\bx)$.
    Conduct eigendecomposition $\mathbf{G} = \mathbf{Q} \mathbf{K}\mathbf{Q}^T$ and define training set $\left\{\bm{\chi'}^{(k)}\right\}_{k=1}^{M}$, $\bm{\chi'}^{(k)} = \mathbf{Q}^T {\bx'}^{(k)}$.
    \State Re-construct the orthonormal \black{\ac{amdP}} basis $\left\{ \phi_{\ba}(\bm\chi)\right\}_{\vert \ba \vert = 0}^p$ with respect to $\rho'(\bm\chi)$ by \textbf{Algorithm 2}.
    Construct the surrogate model $\tilde{g}(\bm\chi)$ with enhanced sparsity following Step 3.
    %Repeat Step 3 to Step 4 if necessary.
  \end{algorithmic}
  \hrule
  \label{alg:explicit_density}
\end{algorithm}

The procedures for random vector rotation and surrogate construction presented in  Algorithms \ref{alg:data_driven} and \ref{alg:explicit_density} can be conducted in an iterative manner.
We have investigated this issue \cite{Yang_Lei_JCP_2016} by applying a previously developed rotation procedure \cite{Lei_Yang_MMS_2015} successively to systems with underlying Gaussian distributions.
For the systems studied in the present work, the improvement of the numerical accuracy is marginal after the first rotation procedure.
Therefore, the numerical results with only one rotation procedure will be presented in this manuscript.
% In the following section, we first compare numerically the properties of different polynomial bases, including orthonormal bases, near orthonormal basis and other standard polynomial (Legendre, Hermite, Chebyshev, etc.) bases which are popular in applications.
% Then we test the developed methods for surrogate construction in various systems with different density distributions.
% The reconstruction of the orthonormal basis $\left\{ \phi_{\ba}(\bm\chi)\right\}_{\vert \ba \vert = 0}^p$ with respect to probability measure of $\bm\chi$ (denoted by $\rho'(\bm\chi)$) in \textbf{Step 4} of \textbf{Algorithm 5} is crucial for the accurate of the surrogate model of $f(\bm\chi)$.
% As shown in Sec. \ref{sec:numerical}, directly utilizing the original basis set $\psi(\cdot)$ on the new random vector $\bm\chi$, violates the orthonormal condition and may lead to increased error of the surrogate model.

\section{Results} \label{sec:numerical}

This section presents the numerical results of the present \ac{DSRAR} framework for surrogate model construction with 
arbitrary underlying distributions. For numerical examples where the probability measure $\rho(\bx)$ (with density function 
$\omega(\bx)$) is not known explicitly and is represented by a discrete data 
set $S= \left\{\bx^{(k)}\right\}_{k = 1}^{N_s}$, we split $S$ equally into two subsets $S=S_1\cup S_2$.
\textcolor{black}{We use $S_1$ to construct the data-driven \black{\ac{amdP}} basis and split $S_2$ into two disjoint subset 
$S_2 = S_{2,1}\cup S_{2,2}$, where $S_{2,1}$ is the training set for surrogate model construction and $S_{2,2}$ is the test
set to evaluate the accuracy of the constructed surrogate model. The size of the training set is $O(10^2) - O(10^3)$ 
and size of the test set is $O(10^5)$.}

\subsection{Accurate recovery of linear systems with data-driven bases} \label{sec:basis_comparison}

In this test, we collected a sample set $S = \left\{\bx^{(k)}\right\}_{k = 1}^{N_s}$ with $N_s = 2\times10^5$.
The random vector $\bx$ followed the Gaussian mixture distribution 
\begin{equation}
  \omega (\bx) = \sum_{i=1}^{N_m} a_i \mathcal{N}(\bm{\mu_i,\Sigma_i})
  \label{eq:GM_test_set}
\end{equation}
where $N_m$ is the number of Gaussian modes.
We set $N_m = 3$, $a_i>0$ for $i=1,2,3$ and $\sum_{i=1}^3 a_i = 1$.
For each Gaussian mode, $\bm\mu_i$ is a 25-dimensional i.i.d.\ random vector with uniform distribution $\mathcal{U}[-2.5,2.5]$ on each dimension and then shifted such that $\sum_{i=1}^3 a_i \bm\mu_i = 0$.
The matrices $\bm \Sigma_i$ were chosen such that 
\begin{equation}
  \bm\Sigma_i = (\bm \Upsilon_i {\bm \Upsilon_1}^T + \mathbf{I})/4,
  \label{eq:GM_var}
\end{equation}
where $\bm\Upsilon_i$ is a random matrix with i.i.d.\ entries from $\mathcal{U}[0,1]$ for $i=1,2,3$.

We considered a linear system \[{\bm A} \bm c = {\bm b} + \bm\ve \] and recovered $\bm c$ \textcolor{black}{using $M$ training points} 
by solving the $l_1$ minimization problem defined by \eqref{eq:ape_L1} where 
\begin{equation}
\left[ {\bm A} \right]_{i,j} = \psi_{j}(\bx^{(i)}), \quad b_i = \sum_{k=1}^N c_{k} \psi_{k}(\bx^{(i)}), 
\label{eq:matA}
\end{equation}
with $1\le i \le M$, $1\le j \le N$, and $\bm \ve$ is noise with $\|\bm \ve\|_2\leq 10^{-7}$. 
We set $d = 25$, $p = 2$ and $N =  \left( \begin{array}{c} d+p \\ p\end{array}\right) = 351$. 
The basis functions $\psi_{\ba}(\bx)$ were constructed on the set $S_1$ by the following approaches: 
\begin{enumerate}
  \item the orthonormal \black{\ac{amdP}} basis subject to Equations \eqref{eq:exact_orth} and \eqref{eq:exact_orth_coeff}; 
  \item the near-orthonormal \black{\ac{amdP}} basis subject to Equation \eqref{eq:near_orth_condition};
  \item tensor product of univariate normalized Legendre polynomials (both sampling points and training points are scaled to lie in $[-1, 1]$ on each dimension accordingly). 
\end{enumerate}
Training points from set $S_2$ were used to examine the recovery accuracy of $\bm c$.

\subsubsection{Sparse linear systems}
\label{sec:recover_sparse}
First, we considered the scenario where $\bm c$ is a $s$-sparse vector and employed the following theoretical bound to examine the recovery accuracy via $l_1$-minimization.

\begin{thm}
  Given a matrix $\bm{{\Psi}} \in \mathbb{R}^{M\times N}$ and set $T_{\ba}$ with $s = \vert T_{\ba} \vert$, a $s$-sparse vector $\bm c$ with non-zero entries on $T_{\ba}$ can be exactly recovered via $l_1$-minimization if $\frac{\theta_s}{1-\delta_s} < 0.5$, where $\delta_s$ and $\theta_s$ are defined by
  \begin{equation}
    \begin{split}
      &\delta_s := \inf \left[ \delta :  (1-\delta)\|\bm y\|_2^2 \le \|\bm\Psi_t \bm y\|_2^2 \le (1+\delta)\|\bm y\|_2^2 \right] ,\  \forall t \subseteq \mathbf{T}, \forall \bm y \in \mathbb{R}^{|t|} \\
      &\theta_s := \inf \left[ \theta : \vert \langle \bm\Psi_{t'} {\bm y}', \bm\Psi_t \bm y\rangle \vert \le \theta \Vert {\bm y}'\Vert_2 \Vert \bm y\Vert_2 \right] ,\  \forall t \subseteq \mathbf{T}, t'\nsubseteq \mathbf{T}, |t'| \le s, \forall \bm y \in \mathbb{R}^{|t|}, {\bm y}' \in \mathbb{R}^{|t'|}
    \end{split}
    \label{eq:delta_s_theta_s}
  \end{equation} \label{thm:exact_recovery}
   where $\bm\Psi_{t}$ and $\bm\Psi_{t'}$ denote the sub-matrices of $\bm\Psi$ with column indices in $t$ and $t'$ respectively.
\end{thm}

Theorem~\ref{thm:exact_recovery} (see \ref{app:exact_recovery} for proof) provides a sufficient condition to exactly 
recover $\bm c$ with non-zero entries on index set $T_{\ba}$. For numerical study, we randomly chose an index set $T_{\ba}$  
from $\Lambda^d_{p}$ with $\vert T_{\ba}\vert = 3$, where $\Lambda_p^d$ is defined by \eqref{eq:f_ap_expan}.
For each training set, we constructed the measurement matrix $\bm A$ with different 
bases and computed $\theta_s/\left(1-\delta_s\right)$ by \eqref{eq:delta_s_theta_s}.
Figure~\ref{fig:err_sparse_vector}(a) shows the mean value $\mathbb{E}\left[\theta_s/\left(1-\delta_s\right)\right]$ on $200$ independent
training sets chosen from $S_2$ for each $M$. 
The exact and near-orthonormal bases yield similar results: $\mathbb{E}\left[\theta_s/\left(1-\delta_s\right)\right]$ becomes smaller than $0.5$ as $M$ approaches $210$,  which is also shown in the inset plot of Figure~\ref{fig:err_sparse_vector}(a).
In contrast, $\mathbb{E}\left[\theta_s/\left(1-\delta_s\right)\right]$ obtained from Legendre polynomial basis shows worse performance due to the loss of orthonormality.

%%%
\begin{figure}[tbp]
  \center
  \subfigure[]{\includegraphics*[scale=0.25]{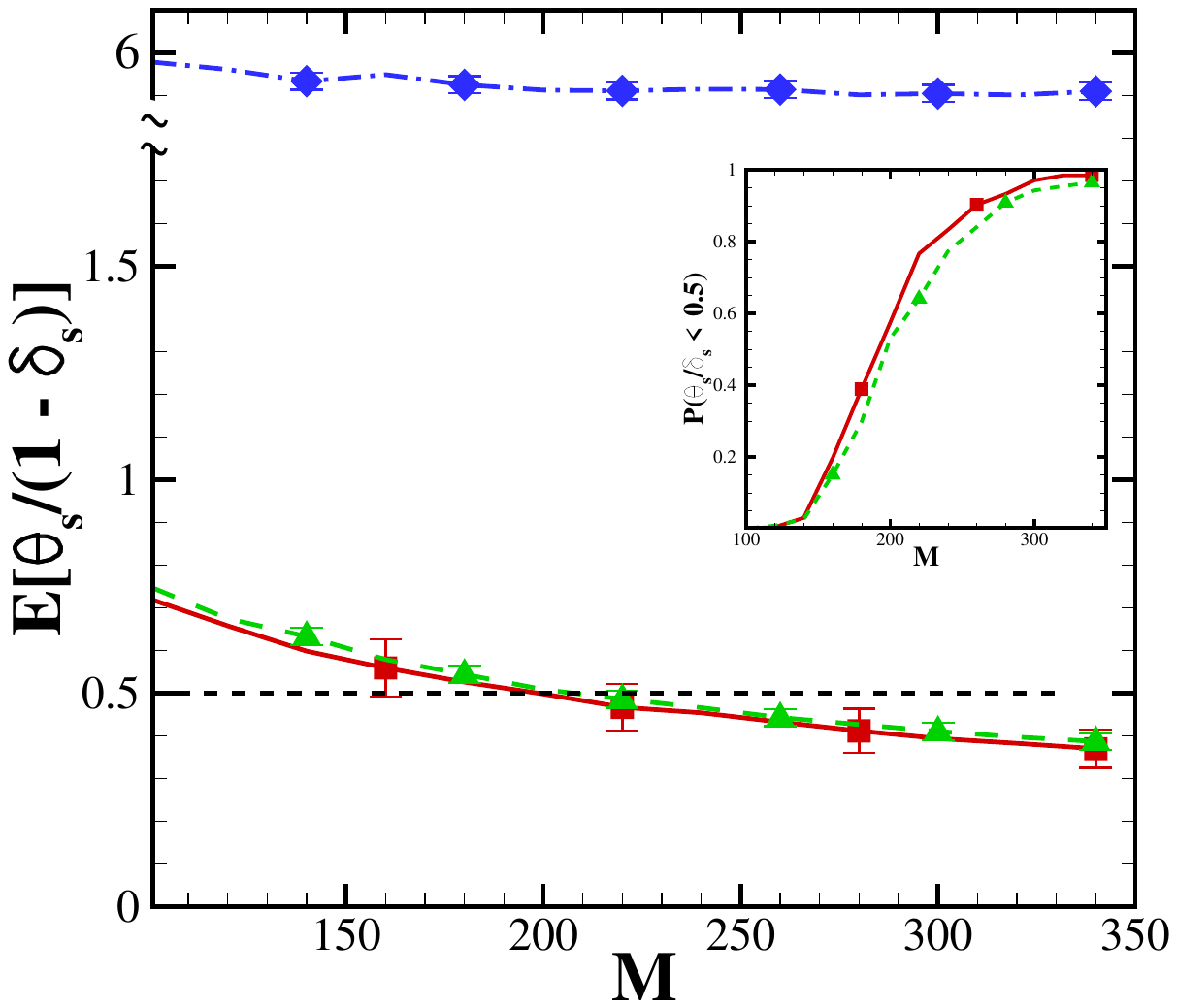}}
  \subfigure[]{\includegraphics*[scale=0.25]{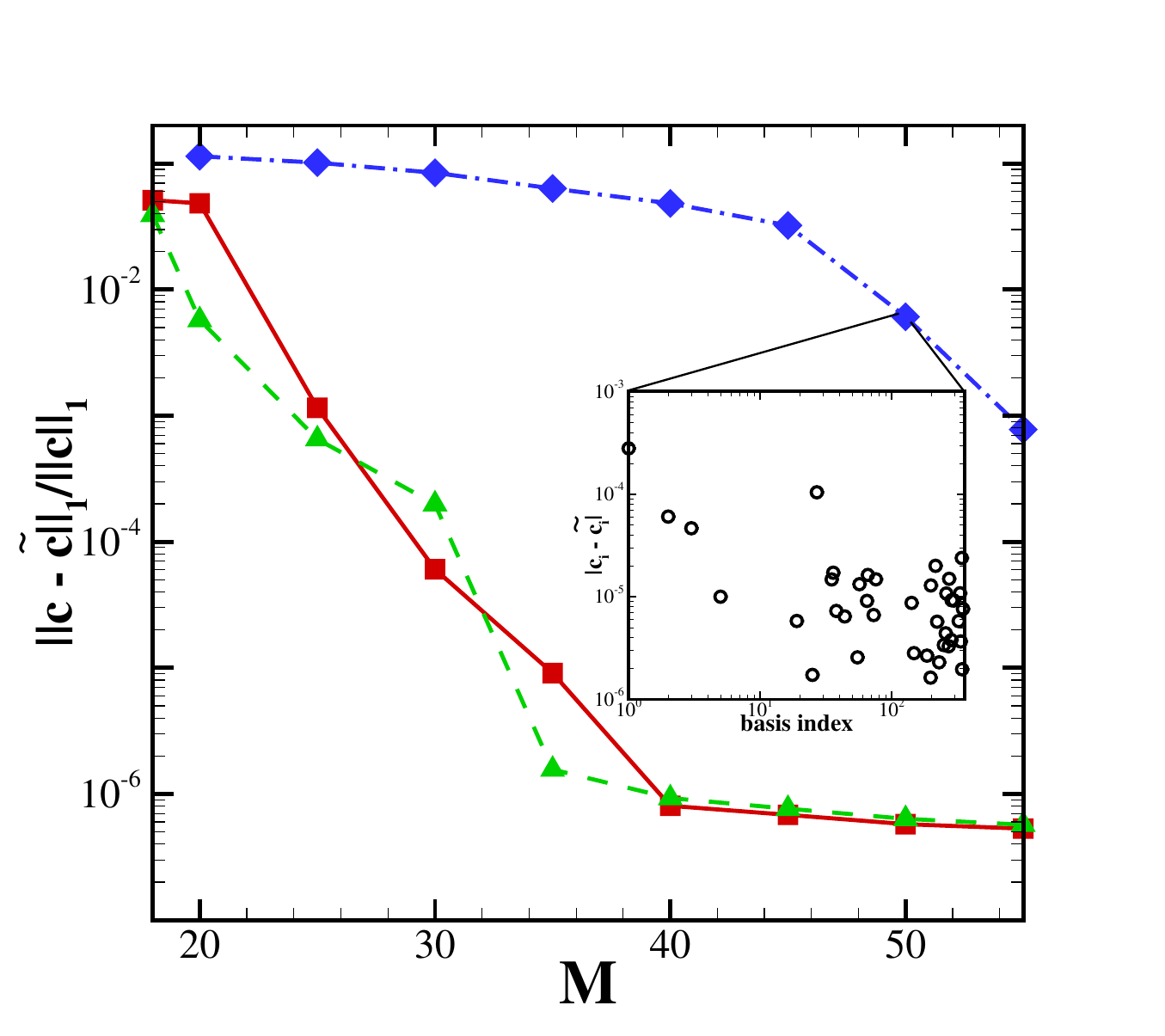}}
  \caption{The measurement matrices constructed by the exact and near-orthonormal bases exhibit similar performance
  in the theoretical (sufficient) bound and numerical results for  recovery of sparse vector. 
   Both bases outperform the Legendre basis.
  ``\textcolor{red}{\protect\rectanglesolidline}'': the exact orthonormal \black{\ac{amdP}} basis;  
  ``\textcolor{green}{\protect\triangledashline}'': the near-orthonormal \black{\ac{amdP}} basis;  
  ``\textcolor{blue}{\protect\diamonddashdotline}'': Legendre basis.
  (a) Mean value of the theoretical bound $\mathbb{E}\left[\theta_s/\left(1-\delta_s\right)\right]$ of exact recovery 
  for measurement matrices $\bm A$ constructed by various bases for the chosen non-zero index $T_{\ba}$ with $s = 3$.
  \black{The error bar represents the standard deviation.}
  The inset plot shows the theoretical prediction of the exact recovery probability.
  (b) Relative $l_1$ error of the recovered sparse vector ($s = 5$) using different training set size $M$.
  The inset plot shows the recovery error $\Vert \bm c - \tilde{\bm c}\Vert_1$ of one training set for the Legendre basis system.} \label{fig:err_sparse_vector}
\end{figure}
%%%

In our numerical experiments, we were able to recover $\bm c$ using fewer samples
than the number {\color{black} $M$---as suggested by the sufficient condition (Theorem 2.5) originally given} by Rauhut
\cite{Rauhut_2012sparseLegen}---since this number is based on the worst case scenario and is not, in general, a sharp bound.
Figure~\ref{fig:err_sparse_vector}(b) shows the numerical results of a test case with $\bm c_{T_{\ba}} = 1$, $\bm c_{T_{\ba}^c} = 0$, $\vert T_{\ba} \vert = 5$. 
%The numerical solution $\tilde{\bm c}$ is computed by 
%setting $\Vert \varepsilon \Vert_2 = 10^{-7}$. 
For each $M$, $200$ CS implementations were conducted to compute the average of the relative error $\Vert \bm c - \tilde{\bm c}\Vert_1/ \Vert \bm c \Vert_1$.
The exact and near-orthonormal \black{\ac{amdP}} bases show similar performance, where $\bm c$ can be accurately recovered (up to $\Vert\bm \varepsilon \Vert_2$) using $M = 45$ training points.
In contrast, the Legendre basis yields larger relative error in $\ell_1$-norm. The relative error of the recovered coefficients from one CS implementation with Legendre basis is shown in the inset plot of Figure~\ref{fig:err_sparse_vector}(b).

\subsubsection{{Non-sparse} linear systems}
\label{sec:recover_dense}
We also tested the recovery performance when the exact representation is not sparse.
The vector $\bm c$ is chosen with a random non-zero index set $T_{\ba}$ with $\vert T_{\ba}\vert = 120$.
Individual components of $\bm c_{T_{\ba}}$ are i.i.d.\ log-normal, such that $\log \bm c_{T_{\ba}} \sim \mathcal{N}(0,2)$.
For each size ($M$) of the training set, $200$ CS implementations were conducted to compute the average of the numerical error $\Vert \bm c - \tilde{\bm c}\Vert_2$, as shown in Figure~\ref{fig:err_no_sparse_vector}(a).
Similar to the previous example, the Legendre basis exhibits the largest approximation error.
The near-orthonormal basis shows smaller error than the exact orthonormal basis.

%%%
\begin{figure}[tbp]
  \center
  \subfigure[]{\includegraphics*[scale=0.25]{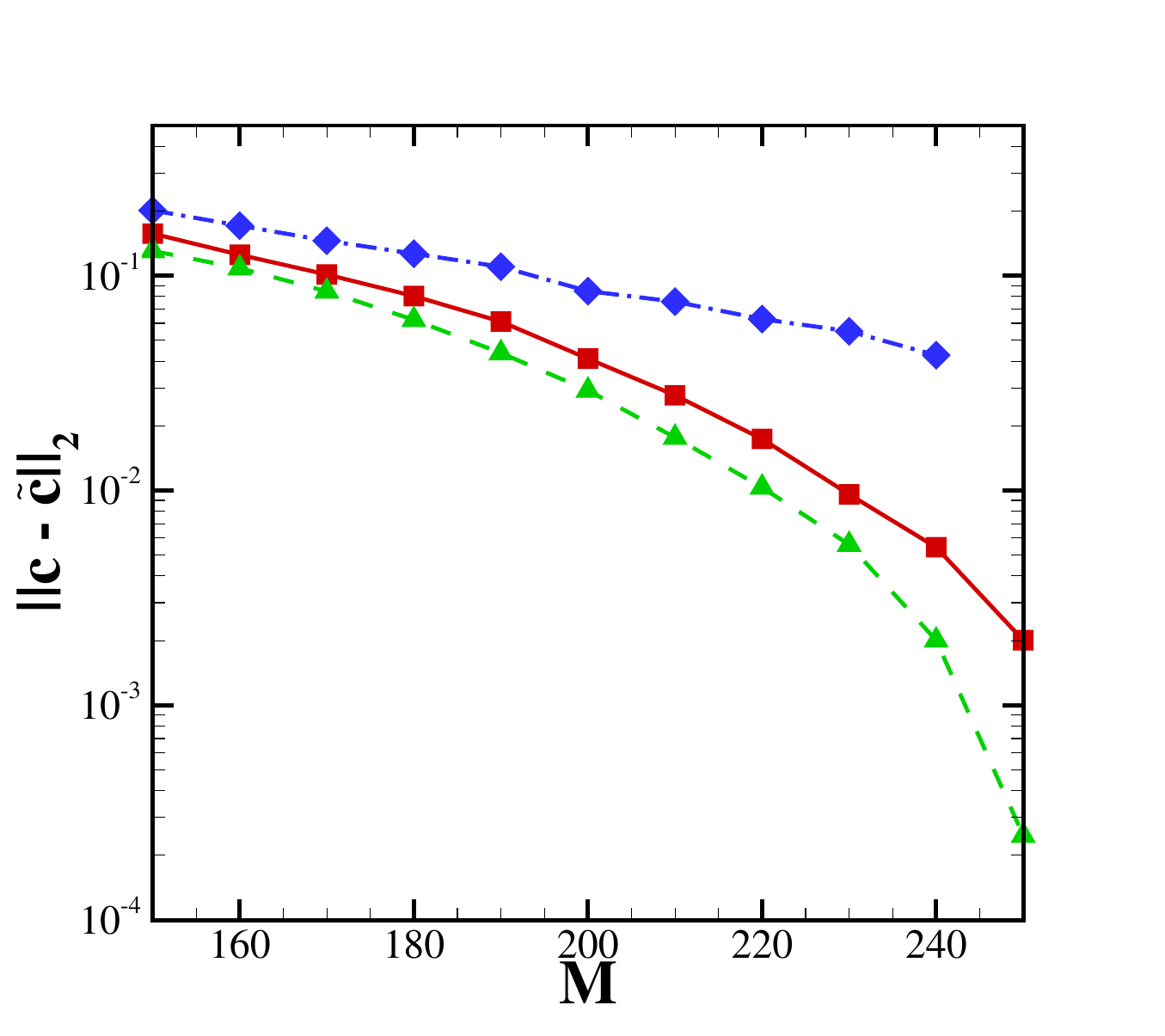}}
  \subfigure[]{\includegraphics*[scale=0.25]{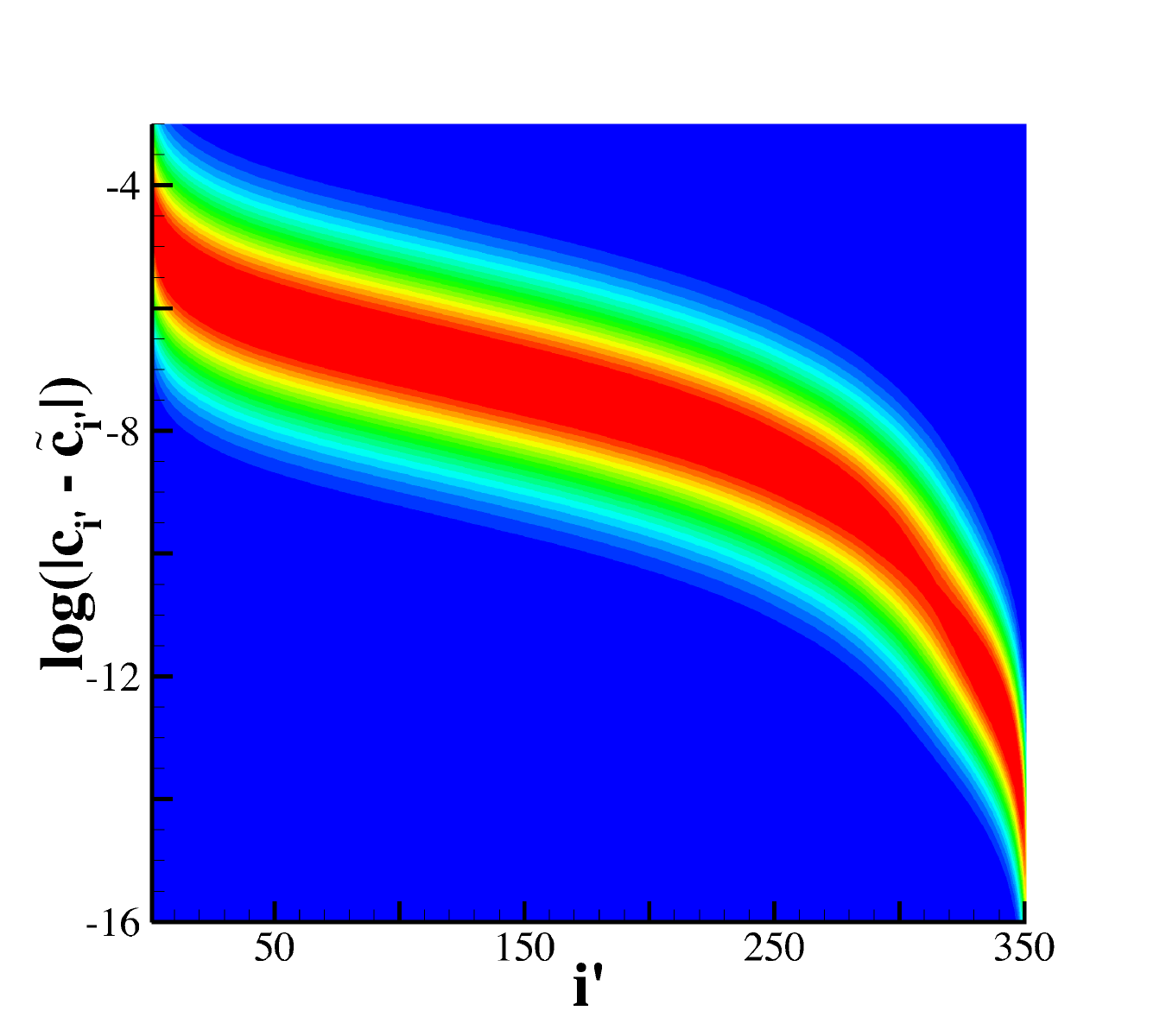}}\\
  \subfigure[]{\includegraphics*[scale=0.25]{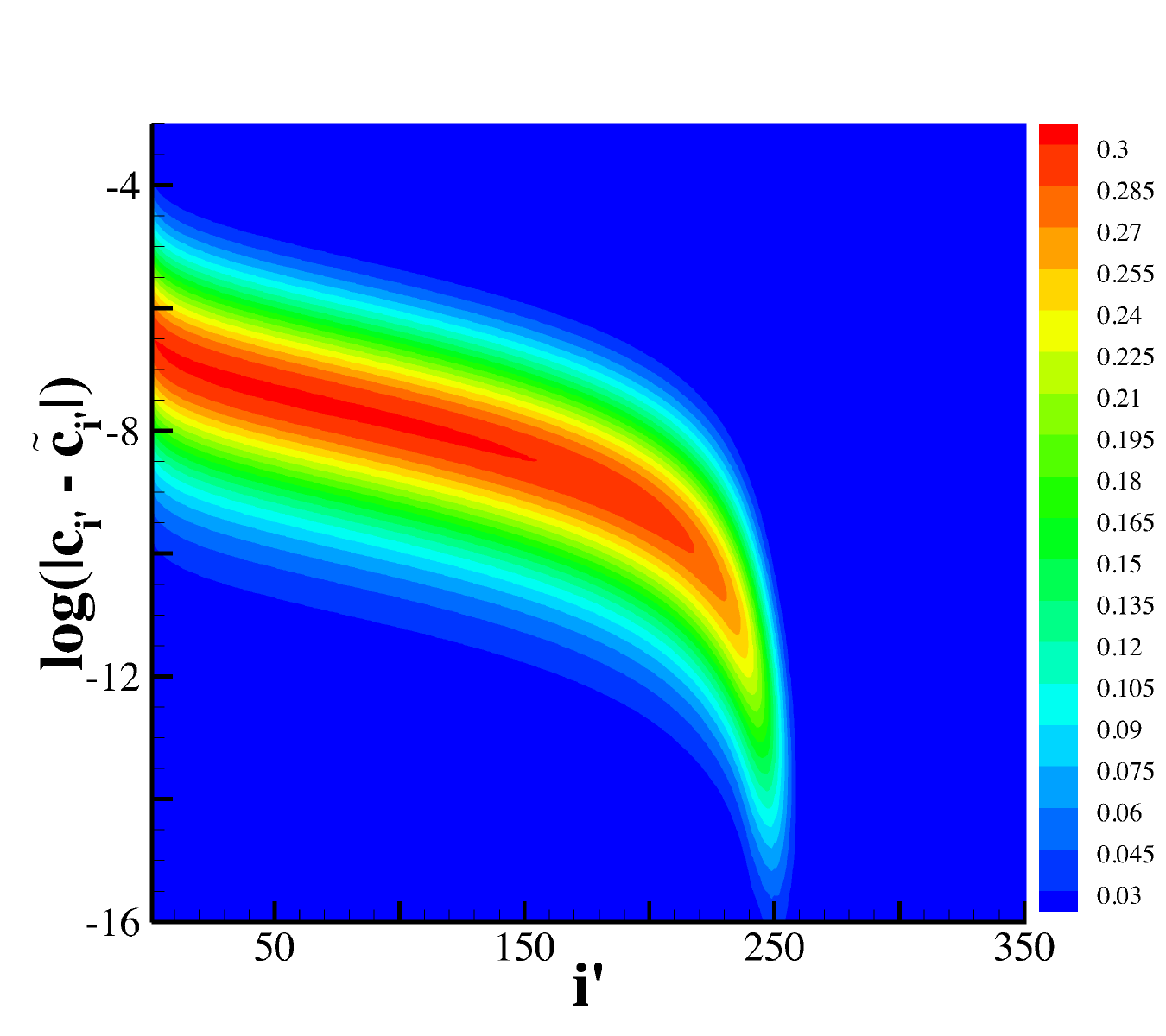}}
  \subfigure[]{\includegraphics*[scale=0.25]{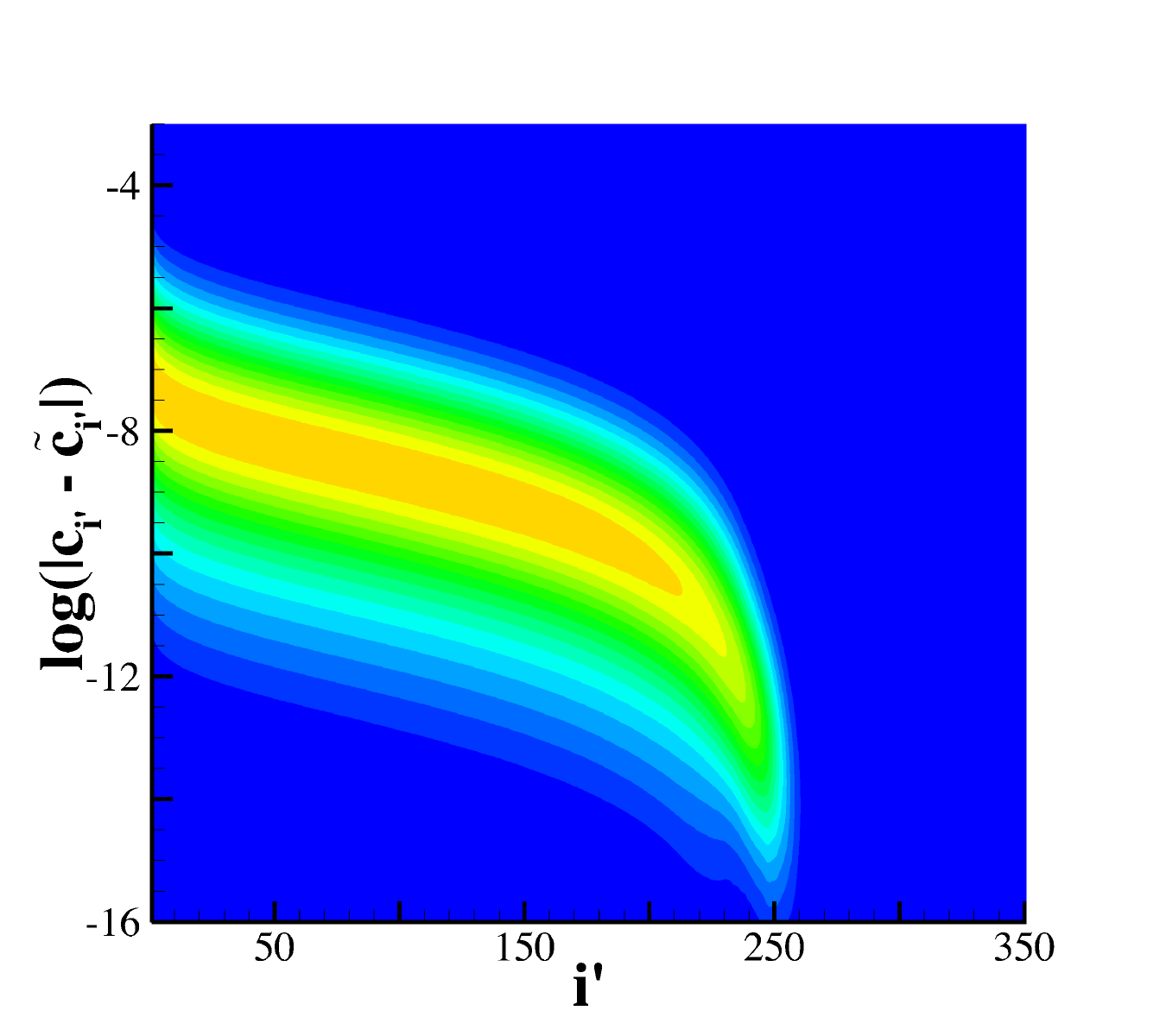}}
  \caption{The measurement matrices constructed by different bases show 
    different numerical performance for the recovery
   of non-sparse vector. The near-orthonormal basis shows the most accurate result.
  (a) $l_1$ error of the recovered vector $\mathbf{c}$ with different bases. 
    ``\textcolor{red}{\protect\rectanglesolidline}'': the exact orthonormal \black{\ac{amdP}} basis;  
  ``\textcolor{green}{\protect\triangledashline}'': the near-orthonormal \black{\ac{amdP}} basis;  
  ``\textcolor{blue}{\protect\diamonddashdotline}'': Legendre basis.
  (b-d) Contours of $\vert \mathbf{c}_{\bm{\alpha}'} -\tilde{\mathbf{c}}_{\bm{\alpha}'}\vert$ (sorted by magnitude) from training
  sets of size $M = 230$ with Legendre (top right), orthonormal (bottom left) and near-orthonormal bases (bottom right).}
\label{fig:err_no_sparse_vector}
\end{figure}
%%%

We also computed the density distribution of individual component $\vert \bm c_{i'} - \tilde{\bm c}_{i'}\vert$, where $i'$ 
refers to single index sorted by the magnitude in descending order.
Figure~\ref{fig:err_no_sparse_vector}(b-d) shows that, compared with the exact orthonormal basis and the Legendre basis, the 
distribution of $\log\left \vert \bm c_{i'} - \tilde{\bm c}_{i'}\right\vert$ obtained from near-orthonormal basis is 
biased toward the smallest magnitudes for error of individual $i'$. This result can be interpreted as that the 
average of $\Vert \bm c - \tilde{\bm c}\Vert_2$ of the near orthogonal basis is smaller than that of the exact 
orthogonal basis and also outperforms the Legendre basis.

\subsection{Systems with explicit knowledge of density function}
\label{sec:explicit_density}
In this subsection, we demonstrate the proposed method in systems with common non-Gaussian 
randomness with analytical density function $\omega(\bx)$. We show that the present method 
based on orthonormal basis construction and rotation of the random variables exploits 
the sparser representation of QoI while retaining proper orthogonality with respect 
to rotated variables. Therefore, it yields more accurate surrogate models than other approaches
based on the direct recovery of $\bm c$ without the sparsity enhancement rotation procedure 
and/or directly applying the rotation procedure without reconstruction of
the orthonormal \black{\ac{amdP}} basis.

\subsubsection{High-dimensional polynomial}
For the first numerical example, we consider a high-dimensional polynomial function 
\begin{equation}
f(\bx) = \sum_{ \vert \ba \vert \le 3} \hat{c}_{\ba} \hat{\psi}_{\ba}(\bx) 
=   \sum_{i = 1}^N \frac{\eta_i}{\vert i \vert^{1.5}} \hat{\psi}_{i}(\bx),
\label{eq:poly_dim_20_p_3}  
\end{equation}
where $\hat{\psi}_{\ba}$ and $\hat{\psi}_{i}$ represent monomial basis functions, $\eta_i$ 
represents uniform random variables $\mathcal{U}\left[0,1\right].$  {We employed
this polynomial function with sparse coefficients as a benchmark problem to examine the recovery
accuracy of the present method.} $\bx$ is a random vector consisting of $20$ i.i.d.\ random variables. 
The density function of the $i\mhyphen$th variable $\xi_i$ is given by 
\begin{equation}
\omega(\xi_i) = e^{-\xi_i},
\label{eq:dens_exp}  
\end{equation}
where the corresponding orthonormal basis are given by the Laguerre polynomials. Accordingly, we construct
a \nth{3}-order polynomial expansion $\tilde{f}(\bx)$ with $N = 1771$ multivariate basis functions,  
which are the tensor product of the univariate Laguerre polynomials. Figure \ref{fig:l2_lauguerre} 
shows the relative $l_2$ error of $\tilde{f}$ computed by level $4$ sparse grid integration.
Similar to the previous example, the \ac{PDF} of $\bm\chi$ does not
retain the form $\omega'(\bm\chi) = \prod_{i=1}^d \exp(-\chi_i)$ after the rotation. Iteratively employing the 
multivariate Laguerre polynomials to represent $\tilde{g}(\bm\chi)$ may result in erroneous prediction (the red dash-dotted curve).
Alternatively, such a problem can be addressed by using the reconstructed orthonormal \black{\ac{amdP}} basis
with respect to $\bm\chi$, which yields a smaller error than $\tilde{f}(\bx)$ (the blue dashed curve).

\begin{figure}[htbp]
\center
\includegraphics*[scale=0.25]{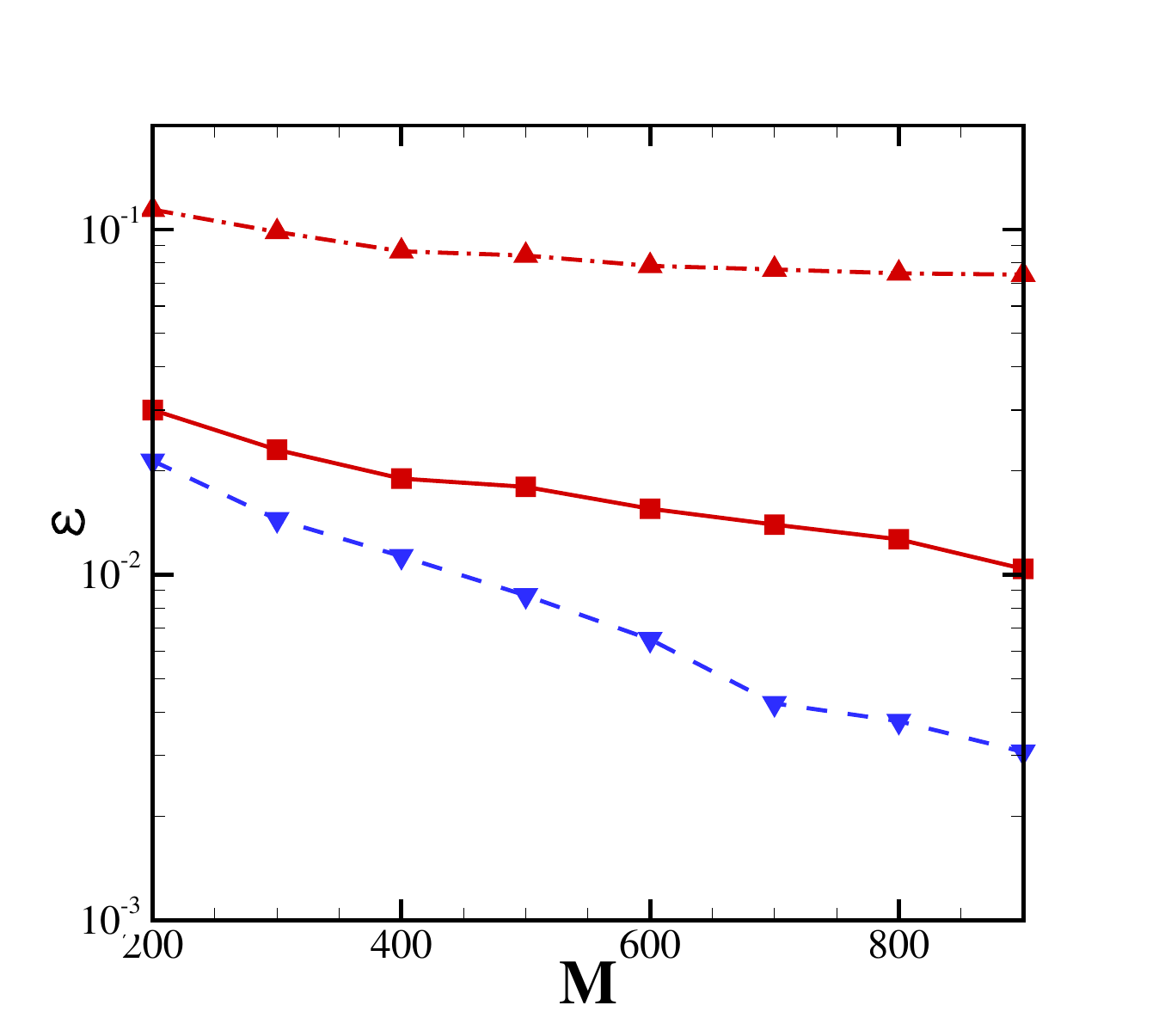}
\caption{Sparsity-enhancing rotation with the reconstructed orthonormal \black{\ac{amdP}} basis yields the most accurate 
  recovery of a high-dimensional polynomial function of random vectors
  following density function given by Equation \eqref{eq:dens_exp}. Directly applying the rotation procedure without
  reconstructing the orthonormal basis yields errorneous prediction. 
  ``\textcolor{red}{\protect\rectanglesolidline}'':  Laguerre polynomial basis with respect to $\bx$;  
  ``\textcolor{red}{\protect\triangledashdotline}'':  Laguerre polynomial basis with respect to rotated vector $\bm\chi$;  
  ``\textcolor{blue}{\protect\downtriangledashline}'':  the reconstructed  \black{\ac{amdP}}  orthonormal basis with respect to rotated vector $\bm\chi$.
}
\label{fig:l2_lauguerre}
\end{figure}

\subsubsection{One-dimensional elliptic \acp{PDE} with high-dimensional random inputs}
We applied the proposed method to model the solution to a one-dimensional (1D) elliptic \ac{PDE} with high dimensional random input
\begin{equation}
  \begin{aligned}
    -\frac{d}{dx} \left( D(x;\bx)\frac{d u(x;\bx)}{dx} \right) = 1, & \quad x \in (0,1) \\
    u(0) = u(1) = 0, &
  \end{aligned} \label{eq:ellip}
\end{equation}
where $a(x;\bx) := \log D(x;\bx)$ is the stochastic input and $a(x;\bx)$ was a stationary process with correlation function 
\begin{equation} \label{eq:exp_kernel}
  K(x,x') = \exp\left(\dfrac{|x-x'|}{l_c}\right),
\end{equation}
where $l_c$ is the correlation length.
We constructed $a(x;\bx)$ by the Karhunen-Lo\`eve (KL) expansion:
\begin{equation} \label{eq:kl}
  a(x;\bx) = a_0(x) + \sigma \sum_{i=1}^{d}\sqrt{\lambda_i}\phi_i(x)\xi_i, 
\end{equation}
where $\{\lambda_i\}_{i=1}^{d}$, and $\{\phi_i(x)\}_{i=1}^{d}$ are the $d$ largest eigenvalues and the corresponding eigenfunctions of $K(x,x')$.
The values of $\lambda_i$ and the analytical expressions for $\phi_i$ were available from the literature \cite{JardakSK02}.
The $\xi_i$ are i.i.d.~ random variables on $\left[-1, 1\right]$. The 
density function of $\xi_i$ is given by
\begin{equation}
\omega(\xi_i) = \frac{1}{\pi \sqrt{1 - \xi_i^2}},
\label{eq:dens_cheb}
\end{equation}
where the corresponding orthonormal basis consists of Chebyshev polynomials of 
the first kind. 
For this example, we set $a_0(x) \equiv 1$, $\sigma = 0.8$, $l_c = 0.14$ and $d = 16$.
We chose the quantity of interest as $u(x;\bx)$ at $x=0.45$ and constructed a \nth{3}-order polynomial expansion with $N=969$ basis functions.
Figure~\ref{fig:l2_chebyshev} shows the relative $l_2$ error of the constructed $\tilde{f}(\bx)$ and $\tilde{g}({\bm\chi})$.
%%%
\begin{figure}[tbp]
  \center
  \includegraphics*[scale=0.25]{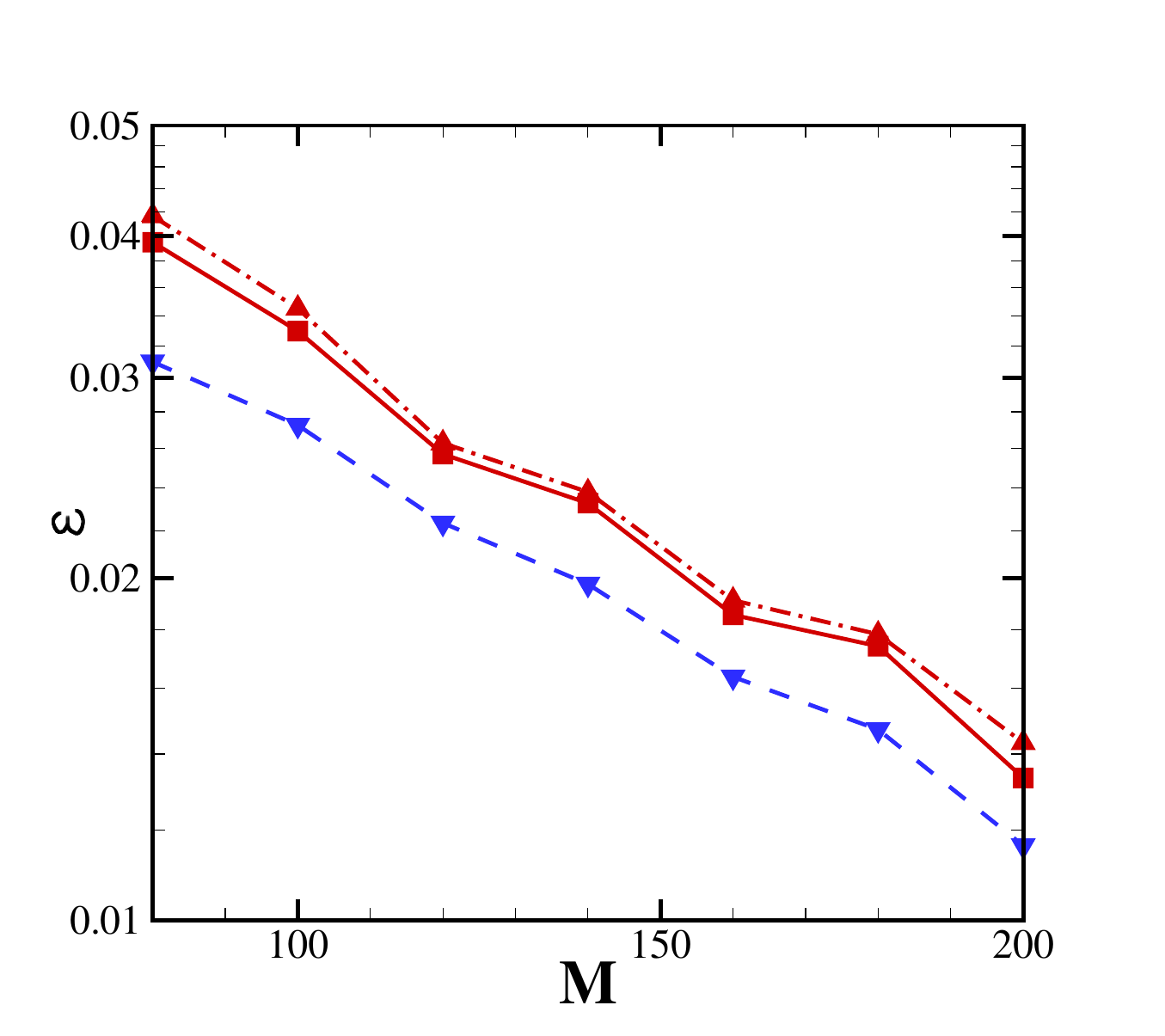}
  \caption{Sparsity-enhancing rotation with reconstructed orthonormal basis yield the most accurate surrogate 
    models for a 1D elliptical \ac{PDE} with random permeability coefficient modeled by Equations~\eqref{eq:kl} and \eqref{eq:dens_cheb}.
  Directly applying the rotation procedure without reconstructing the orthonormal basis yields increased numerical error. 
  ``\textcolor{red}{\protect\rectanglesolidline}'':  Chebyshev polynomial basis with respect to $\bx$;  
  ``\textcolor{red}{\protect\triangledashdotline}'':  Chebyshev  polynomial basis with respect to rotated vector $\bm\chi$;  
  ``\textcolor{blue}{\protect\downtriangledashline}'':  the reconstructed orthonormal \black{\ac{amdP}} basis with respect to rotated vector $\bm\chi$.
  } \label{fig:l2_chebyshev}
\end{figure}
%%%
For the density function $\omega(\bx_i)$ given by \eqref{eq:dens_cheb}, $\tilde{f}(\bx)$ can be represented by a multivariate 
basis constructed by the tensor products of univariate Chebyshev polynomials.
However, in general, the \ac{PDF} of $\bm\chi$ does not retain the form $\omega'(\bm\chi) = \prod_{i=1}^d \frac{1}{\pi \sqrt{1 - \chi_i^2}}$. 
As shown in Figure~\ref{fig:l2_chebyshev}, iteratively employing the multivariate Chebyshev polynomials to represent $\tilde{g}(\bm\chi)$ 
(the red dash-dotted curve)---as done in previous studies \cite{Yang_Wan_rotation_2017}---resulted in a larger error than $\tilde{f}(\bx).$
Representing $\tilde{g}(\bm\chi)$ by the reconstructed orthonormal \black{\ac{amdP}} basis (the blue dashed curve) further decreases 
the numerical error compared to $\tilde{f}(\bx)$ (the solid red curve).

%{In this subsection, we  presented two numerical examples that included explicit knowledge of the probability 
%density. The numerical results showed that the combination of the rotation procedure with orthonormal basis 
%construction yielded the most accurate results. Directly applying the rotation procedure may lead to increased error.}

\subsection{Systems with implicit knowledge of density function}
In this suite of benchmark examples, we investigated  the applicability and efficiency of the developed \ac{DSRAR} framework based on 
data-driven orthonormal bases construction and sparsity enhanced rotation.

\subsubsection{High-dimensional polynomials} \label{sec:high_d_poly}
We studied the ability of the data-driven method to recover a high-dimensional polynomial function 
\begin{equation}
  f(\bx) = \sum_{ \ba \in T_{\ba}} \hat{\psi}_{\ba}(\bx),
\end{equation}
where $\hat{\psi}_{\ba}$ represents the monomial basis function, $T_{\ba}$ represents a set containing $50$ indices randomly chosen from $\Lambda_p^d$ with $d = 25$ and $p = 3$.
The sample set $S$ of random vector $\bx$ for basis construction was generated from the Gaussian mixture model specified in \eqref{eq:GM_test_set} with $|S| = 2\times 10^5$. 

We approximated $f(\bx)$ by a \nth{3}-order polynomial expansion $\tilde{f}(\bx) = \sum_{i=1}^N \tilde{c}_i\psi_i(\bx)$ with $N = 3276$.
Figure~\ref{fig:err_rand_poly_dim_25}(a) shows the relative $l_2$ error of the constructed surrogate model $\tilde{f}$ defined by
\begin{equation}
  \epsilon = \left(\int(f(\bx) - \tilde{f}(\bx))^2 \dif \nu_{S_2}(\bx) \big/ \int f(\bx)^2 \dif \nu_{S_2}(\bx) \right)^{\frac{1}{2}},
\end{equation}
where $20$ implementations were utilized for each training sample size number $M$.
%%%
\begin{figure}[htbp]
  \center
  \subfigure[]{\includegraphics*[scale=0.25]{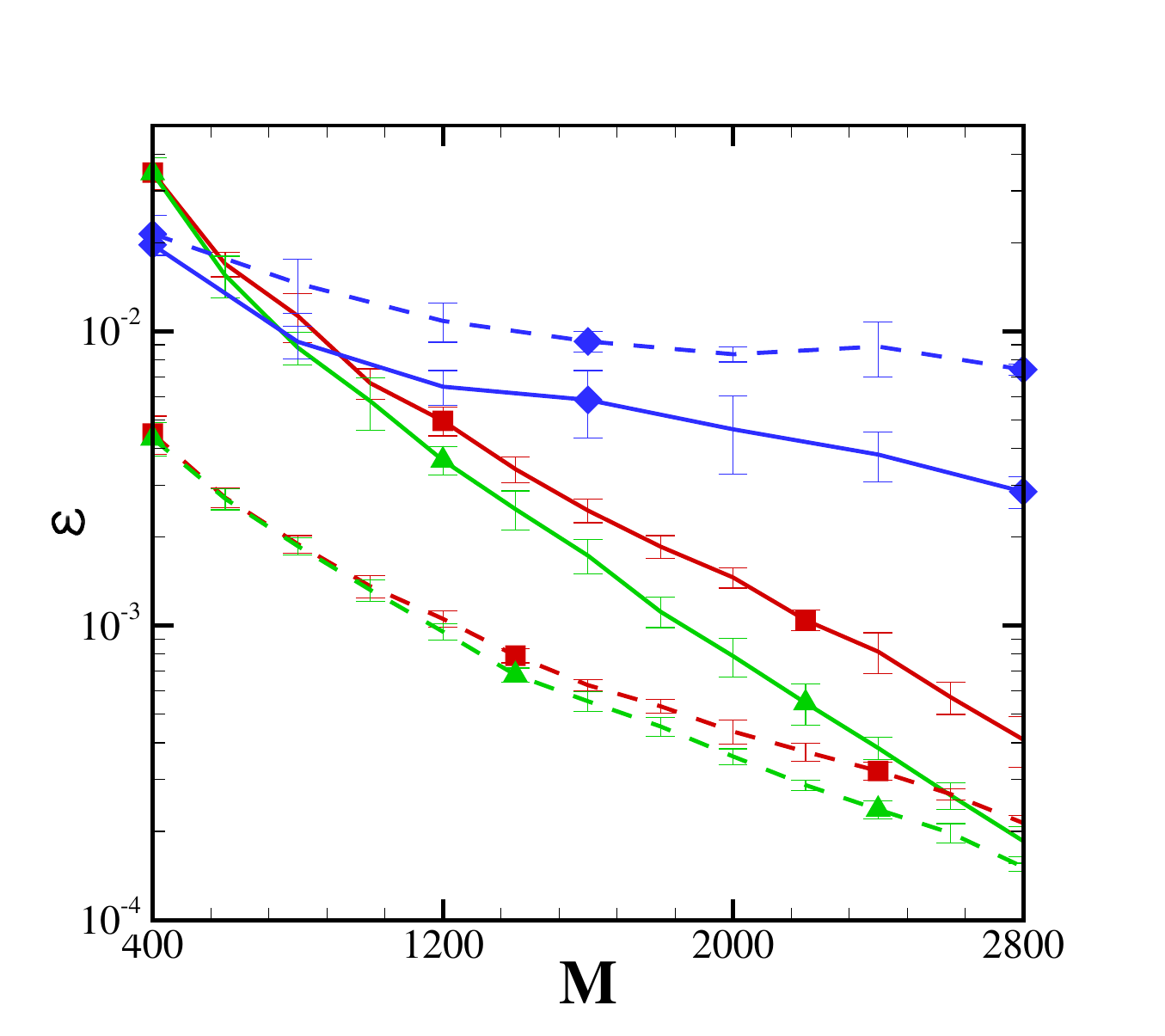}}
  \subfigure[]{\includegraphics*[scale=0.25]{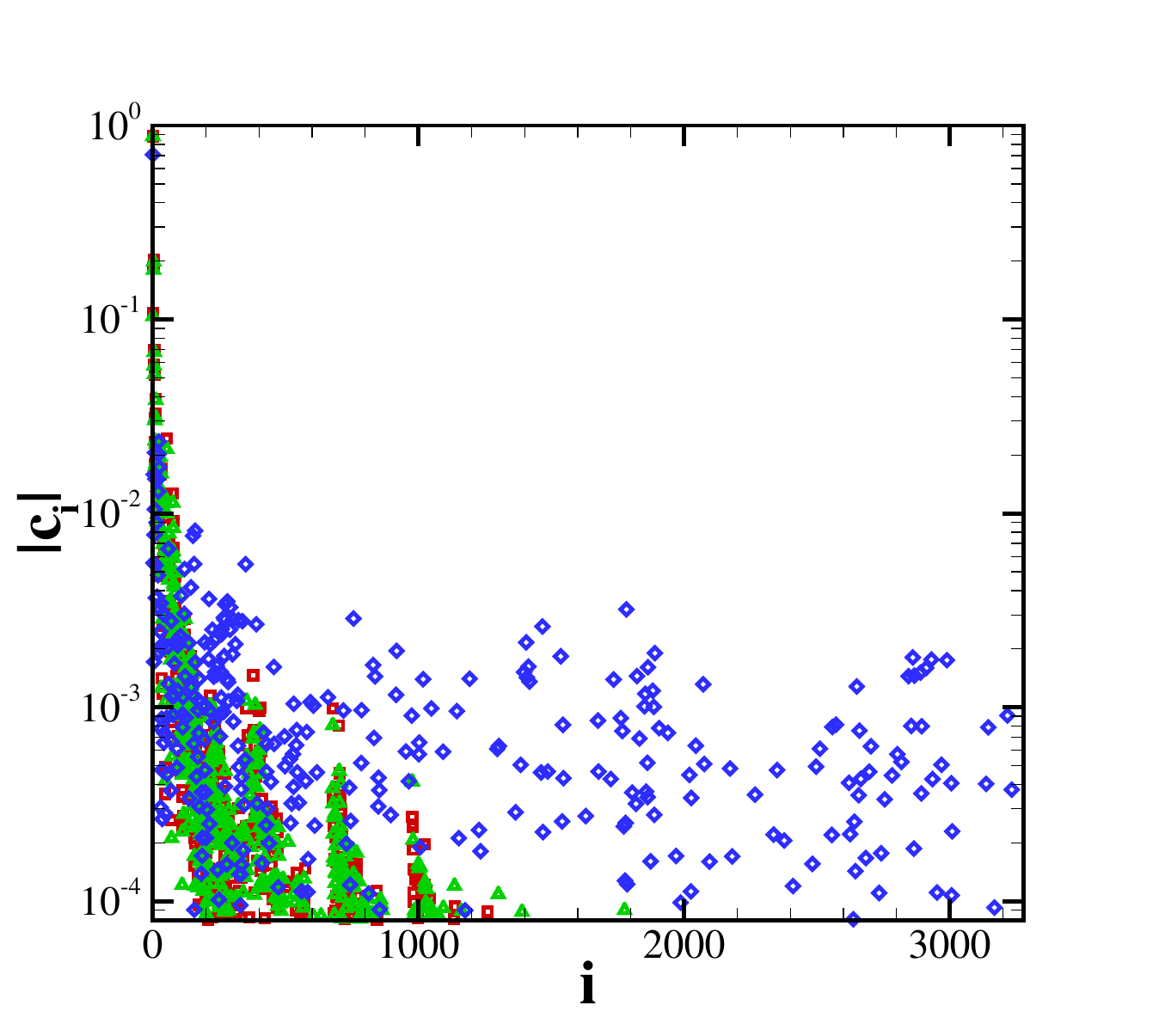}} \\
  \subfigure[]{\includegraphics*[scale=0.25]{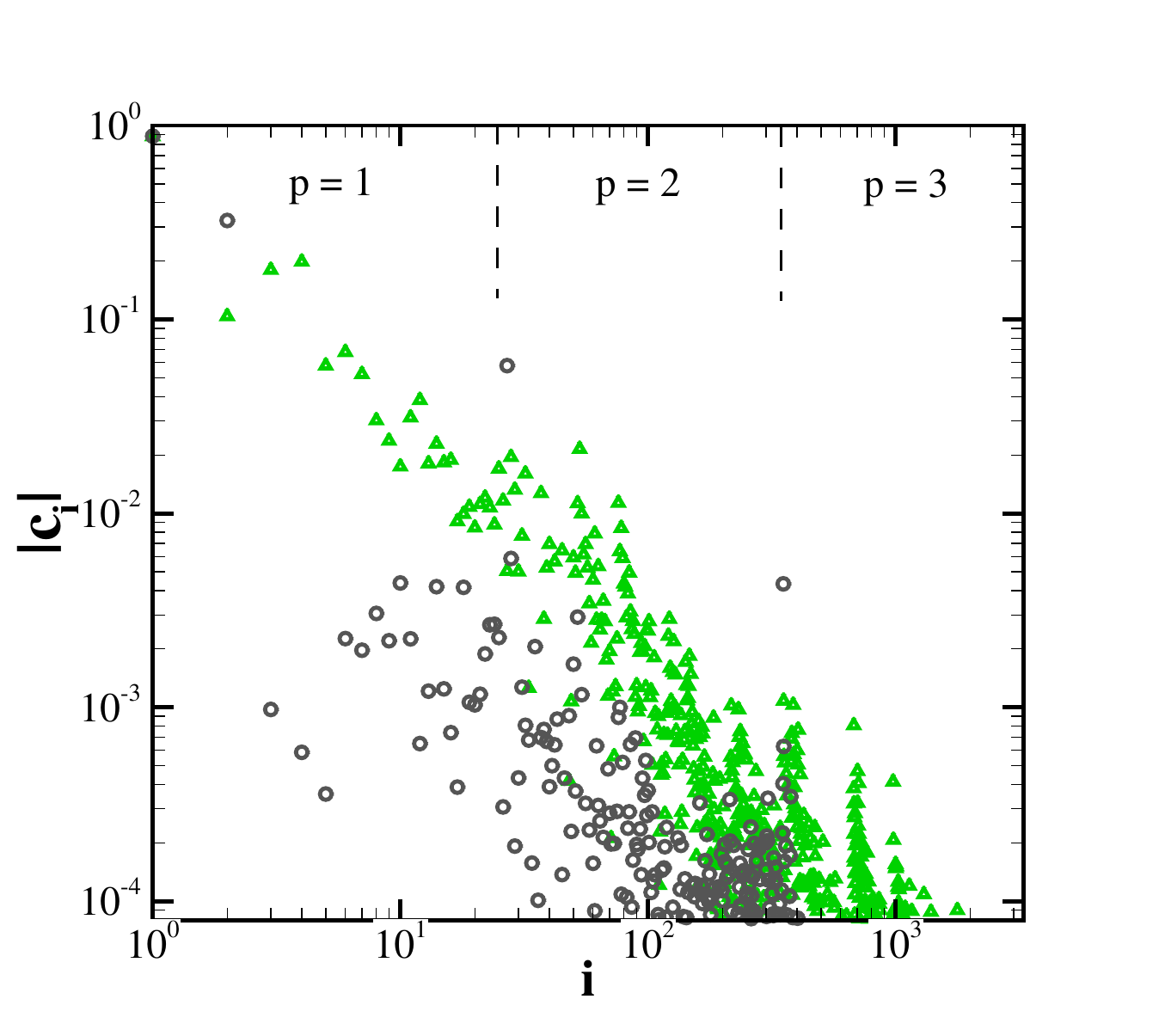}}
  \subfigure[]{\includegraphics*[scale=0.25]{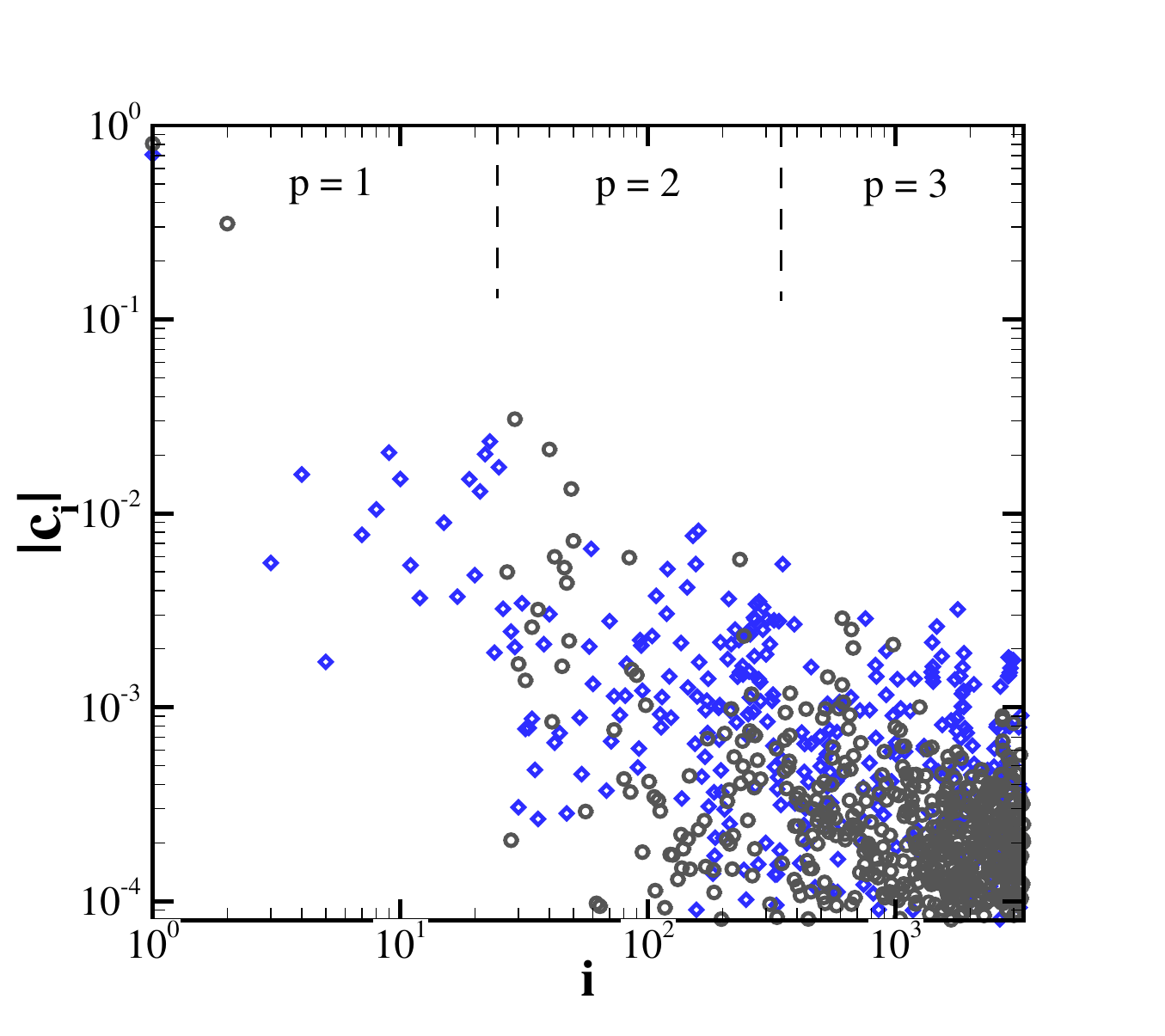}}
  \caption{Numerical results for recovery of a high-dimensional polynomial function. The combination of
  near-orthonormal basis construction with the sparsity enhancement rotation procedure yields 
  the most accurate results. Directly applying the rotation procedure to the Legendre basis 
  may lead to increased error despite increased sparsity in $\bm c$.
  (a) Relative $l_2$ error of the recovered polynomial function with different bases:
  the exact orthonormal \black{\ac{amdP}} basis with respect to $\bx$ (``\textcolor{red}{\protect\rectanglesolidline}'') and 
  $\bm\chi$ (``\textcolor{red}{\protect\rectangledashline}''); 
  the near-orthonormal \black{\ac{amdP}} basis with respect to $\bx$ (``\textcolor{green}{\protect\trianglesolidline}'') and
  $\bm\chi$ (``\textcolor{green}{\protect\triangledashline}'');
  Legendre basis with respect to $\bx$ (``\textcolor{blue}{\protect\diamondsolidline}'') and 
  $\bm\chi$( ``\textcolor{blue}{\protect\diamonddashline}'').
  (b) Coefficients magnitude $\vert c_{i}\vert$ recovered using different bases. 
  ``\textcolor{red}{\protect\rectangleopen}'': the exact orthonormal \black{\ac{amdP}} basis with respect to $\bx$;  
  ``\textcolor{green}{\protect\triangleopen}'': the near-orthonormal \black{\ac{amdP}} basis with respect to $\bx$;  
  ``\textcolor{blue}{\protect\diamondopen}'': Legendre basis with respect to $\bx$.
  (c) Recovered coefficient magnitude $\vert \mathbf{c}_{i}\vert $ using the near orthogonal basis with respect to $\bx$ 
    (``\textcolor{green}{\protect\triangleopen}'') and $\bm\chi$ (``\textcolor{gray}{\protect\circleopen}'').
  The dashed vertical lines indicate the separation between different polynomial orders $p$.
  (d) Recovered coefficient magnitude $\vert \mathbf{c}_{i}\vert $ using the Legendre basis with respect to 
  $\bx$ (``\textcolor{blue}{\protect\diamondopen}'') and $\bm\chi$ (``\textcolor{gray}{\protect\circleopen}'').} 
  \label{fig:err_rand_poly_dim_25}
\end{figure}
%%%
As shown in Figure~\ref{fig:err_rand_poly_dim_25}(a), $\tilde{f}(\bx)$ constructed by the near-orthonormal 
\black{\ac{amdP}} 
basis yielded the smallest error while the tensor product of Legendre basis functions yielded the largest error.
Accordingly, the magnitudes of the recovered coefficients $\vert \tilde{c}_i \vert$ by the exact and near-orthonormal bases decayed more quickly than those recovered using the Legendre basis functions, as shown in Figure~\ref{fig:err_rand_poly_dim_25}(b).
Furthermore, $\tilde{f}(\bx)$ allowed us to define a new random vector $\bm\chi$, which further enhanced the sparsity of $\bm c$, as shown in Figures~\ref{fig:err_rand_poly_dim_25}(c) and (d).
Following Step 5 in Algorithm 4, we defined a new random $\bm\chi$ through rotation. The associated representation
coefficient vector $\bm c$ has enhanced sparsity.

However, for the exact and near-orthonormal basis, the $\tilde{g}(\bm\chi)$ gave smaller errors (the dashed curve) than $\tilde{f}(\bx)$ (the solid curve), as shown in Figure~\ref{fig:err_rand_poly_dim_25}(a).
Thus, enhancing the sparsity of $\bm c$ alone does not guarantee enhanced accuracy of $\tilde{f}$.
In particular, $\tilde{g}(\bm\chi)$ constructed by the Legendre basis yielded larger error than $\tilde{f}(\bx)$ as demonstrated in Figure~\ref{fig:err_rand_poly_dim_25}(a); although, the sparsity of $\bm c$ was greater, as seen in Figure~\ref{fig:err_rand_poly_dim_25}(d).
This behavior indicates that retaining the orthonormal condition can be crucial for the accurate construction of $\tilde{f}$. {The basis bound (see Table \ref{tab:GM_d_25_p_3} in \ref{app:basis_bound}) provides a metric to understand why the near-orthonormal basis performs better than the exact orthonormal basis.}

%Table~\ref{tab:GM_d_25_p_2} shows the basis bounds for samples from Gaussian mixtures $\left\{\bx^{(i)}\right\}, i = 1, \cdots, N_s$ with $N_s = 1\times 10^5$, $d = 25$ and {$p = 2$}.
%%%%
%\begin{table}[tbp]
%  \centering
%  \caption{Basis bound $\tilde{K}$ of constructed basis set for Gaussian mixture system $d = 25$, $p = 2$ and $N_s = 1\times 10^5$.}
%  \begin{tabular}{C{8em}|C{6em} C{6em} C{6em} C{6em} C{6em}}
%    \hline\hline
%    $M_{\sigma}$ & $3$ & $4$ & $5$ & $6$ &$ \displaystyle \mathop{\max}_{\bx\in S} k(\bx)$ \\
%    \hline
%    $\tilde{K}_{\rm orth}$ & 10.359 & 12.048 & 13.895 & 15.513 & 22.208\\ 
%    $\tilde{K}_{\rm near-orth}$ & 9.622 & 11.196 & 12.867 & 14.448 & 18.790\\ 
%    \hline\hline
%  \end{tabular}
%  \label{tab:GM_d_25_p_2}
%\end{table}
%%%
%{For different values of $M_{\sigma}$, the values of $\tilde{K}$ for the near orthogonal basis were consistently smaller than the value for the exact orthogonal basis set.}
%\subsection{\ac*{UQ} in systems with non-Gaussian noise distributions}

\subsubsection{1D elliptic \acp{PDE} with high-dimensional random inputs}
In this example, we revisited the 1D elliptic \ac{PDE} (\ref{eq:ellip})
with random coefficient given by Equation \eqref{eq:kl}. 
Here we set $a_0(x) \equiv 1$, $\sigma = 1$, $l_c=0.12$ and $d=20$ such that $\sum_{i=1}^d\lambda_i > 0.91\sum_{i=1}^{\infty}\lambda_i$. 

Similar to the work by Zabaras et al.\ \cite{Zabaras_2014}, a non-Gaussian multivariate distribution was used for $\bx = \left(\xi_1, \xi_2, \cdots, \xi_d\right)$.
We generated a sample set $\left\{\tilde{\bx}^{(k)}\right\}_{k=1}^{Ns}$, where $N_s = 2\times 10^5$ and $\tilde{\bx}$ came from the Gaussian mixture distribution specified in \eqref{eq:GM_test_set}.
We used \ac{PCA} to transform $\tilde{\bx}$ to $\bx$ such that $\mathbb{E}\left[\bx_i\right] = 0$ and $\mathbb{E}\left[\bx_i \bx_j\right] = \delta_{ij}$.
For each input sample $\bx^{(k)}$, $a$ and $u$ only depended on $x$ and the solution of the deterministic elliptic equation is given by \cite{Yang_2013reweightedL1}
\begin{equation}\label{eq:ellip_sol}
  \begin{split}
    & u(x) = u(0) + \int_0^x \dfrac{a(0)u(0)'-y}{a(y)}\dif y \\
    & a(0)u(0)' = \left(\int_0^1 \dfrac{y}{a(y)}\dif y\right) \Big / \left(\int_0^1 \dfrac{1}{a(y)}\dif y\right).
  \end{split}
\end{equation}

We chose the \ac{QoI} to be $u(x;\bx)$ at $x=0.35$ and constructed a \nth{3}-order polynomial expansion with $N=1771$ basis functions.
Figure~\ref{fig:err_elliptical_dim_20_p_3} shows the relative $l_2$ error of $\tilde{f}(\bx)$ (solid curve) and $\tilde{g}(\bm\chi)$ (dashed curve) constructed by different bases.
%%%
\begin{figure}[tbp]
  \center
  \includegraphics*[scale=0.25]{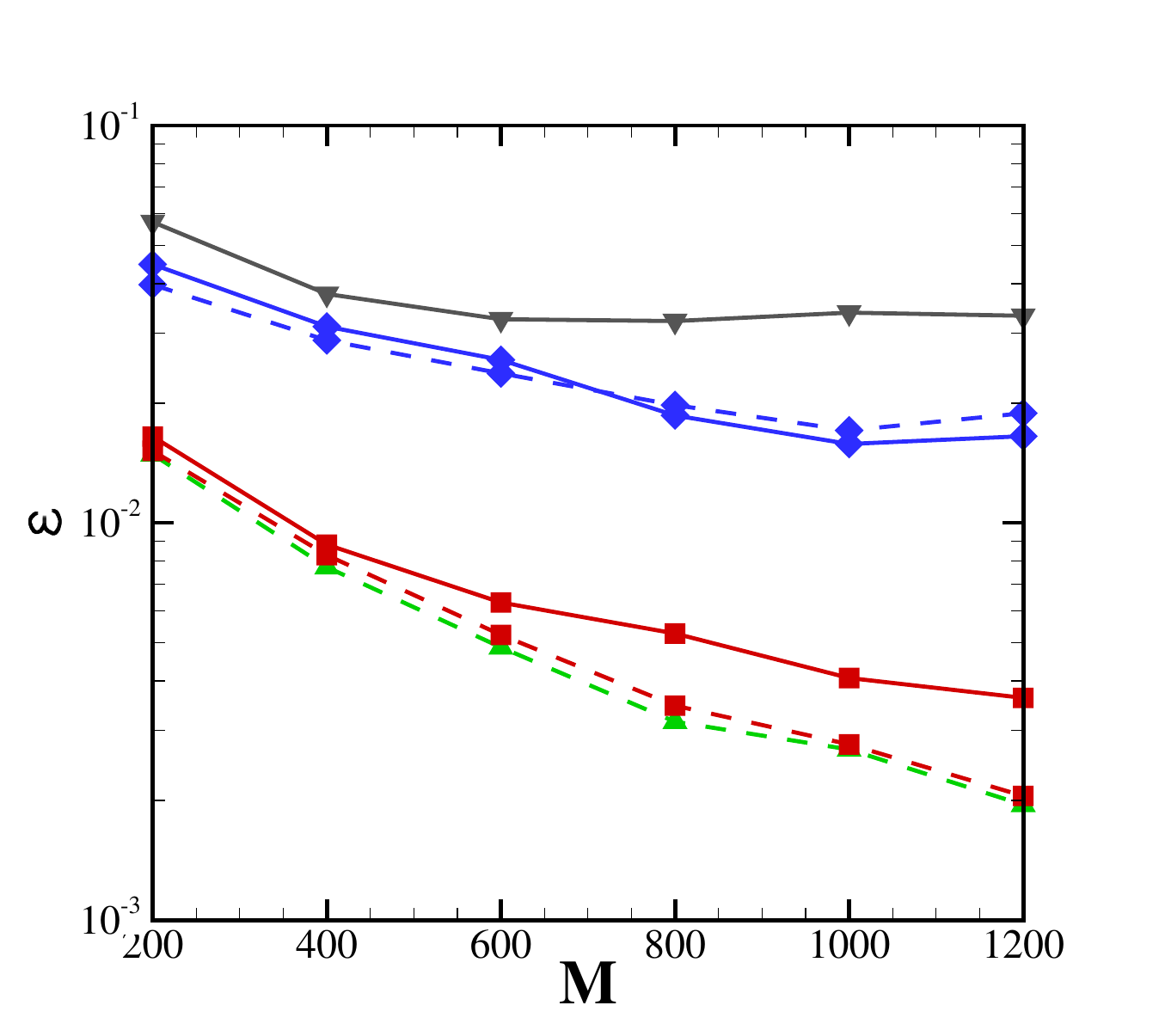}
  \caption{The combination of near-orthonormal basis construction and sparsity enhancement rotation
  yields the most accurate results, as shown through the relative $l_2$ error of the constructed surrogate model for the 1D elliptic \ac{PDE} 
  with random permeability coefficient: the exact orthonormal \black{\ac{amdP}} basis with respect to $\bx$ 
    (``\textcolor{red}{\protect\rectanglesolidline}'') and 
  $\bm\chi$ (``\textcolor{red}{\protect\rectangledashline}''); 
  %near-orthonormal basis with respect to $\bx$ (``\textcolor{green}{\protect\trianglesolidline}'') and
  %$\bm\chi$ (``\textcolor{green}{\protect\triangledashline}'');
  Legendre basis with respect to $\bx$ (``\textcolor{blue}{\protect\diamondsolidline}'') and 
  $\bm\chi$( ``\textcolor{blue}{\protect\diamonddashline}'');
  Hermite basis with respect to $\bx$ (``\textcolor{gray}{\protect\downtrianglesolidline}'');
  the near-orthonormal \black{\ac{amdP}} basis with respect to $\bm\chi$ (``\textcolor{green}{\protect\triangledashline}'').
  } \label{fig:err_elliptical_dim_20_p_3}
\end{figure}
%%%
The data-driven bases (both exact orthonormal basis and near-orthonormal basis) showed more accurate results than the Legendre basis and the Hermite basis.
In particular, the near-orthonormal basis with respect to the rotated variable $\bm\chi$ yielded the most accurate result (the green dashed curve).
In contrast, directly employing the Legendre basis to the rotated  variable $\bm\chi$ without reconstructing the basis function led to increased $l_2$ error, although $\bm c$ shows more sparsity in terms of $\bm\chi$ (the gray dashed curve) than $\bx$ (the gray solid curve). 

\subsection{\ac{UQ} study of a molecule system under Non-Gaussian conformational distributions}\label{sec:mole_example}

We demonstrated the proposed method on a physical system exploring conformational uncertainty in a small molecule system.
Molecular properties, such as solvation energies or \acp{SASA}, are often calculated using single molecular conformations.
However, due to thermal energy, a molecule undergoes conformational fluctuations which can induce significant uncertainty in properties calculated from single structures.
Our previous work \cite{Lei_Yang_MMS_2015} was focused on quantifying this uncertainty using a simple multivariate Gaussian model for conformational fluctuations: the elastic network model \cite{Ati_Bahar_BJ_2001}.
However, it is well known that the conformational fluctuations are often non-Gaussian due to the complicated structure of the underlying energy landscape.
Therefore, in the current study, we construct the data-driven basis \emph{directly} from the samples of molecular trajectories collected from \ac{MD} simulations, thus eliminating the \emph{over-simplified} Gaussian assumption.

We simulated the dynamics of the small molecule benzyl bromide under equilibrium (see \ref{app:sim} for details) and collected a sample set of the instantaneous molecular structure  $\left\{\mb r^{(k)}\right\}_{k=1}^{N_s}$ from \ac{MD} simulation trajectories over $20 \mu$s.
In what follows, $N_s = 2\times 10^5$ and $\mb r$ represent the positions of individual atoms.
As a pre-processing step, we transformed $\left\{\mb r^{(k)}\right\}_{k=1}^{N_s}$ into a set of uncorrelated random vectors $S = \left\{\mb \bx^{(k)}\right\}_{k=1}^{N_s}$ via \ac{PCA}: 
\begin{equation}
  \begin{split}
    &\bm\Sigma = \mathbb{E}\left[\left(\mb r - \bar {\mb r}\right) \left(\mb r - \bar {\mb r}\right)^T\right] \\
    &\bm\Sigma = \mb Q \bm \Gamma \mb Q^T \quad \bx = \bm \Gamma^{-1/2}\mb Q^T\mb r,
  \end{split}
\end{equation}
where the average $\mathbb{E}[\cdot]$ is taken over the entire sample set and $\bx \in \mathbb{R}^{12}$ is the normalized random vector that represents $99.99\%$
of the {observed variance}. 
Figures~\ref{fig:err_mol_dim_12_p_4}(a) and (b) show the joint distributions of $\left(\xi_1,\xi_2\right)$ and $\left(\xi_1,\xi_3\right)$.
%%%
\begin{figure}[tbp]
  \center
  \subfigure[]{\includegraphics*[scale=0.25]{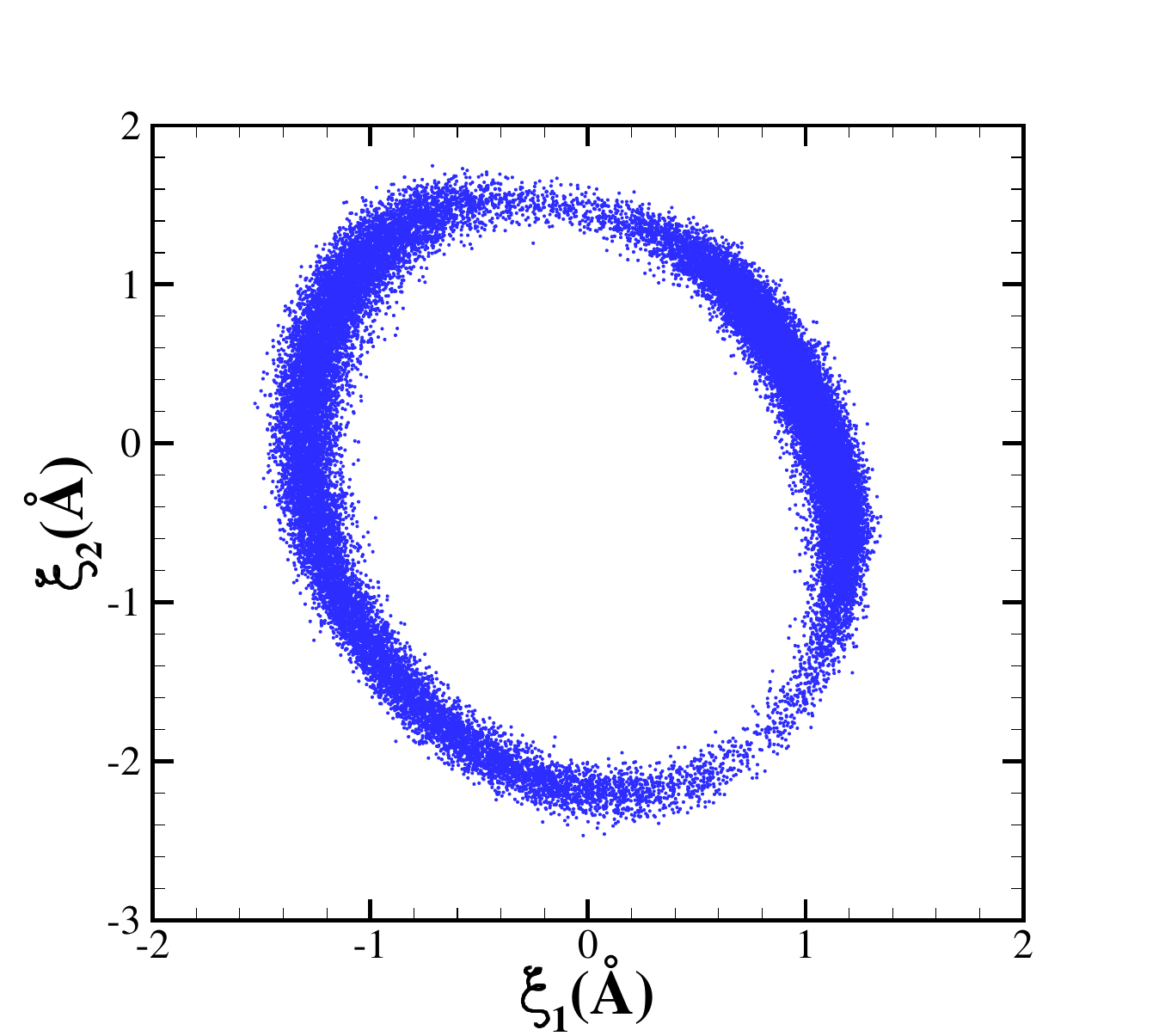}}
  \subfigure[]{\includegraphics*[scale=0.25]{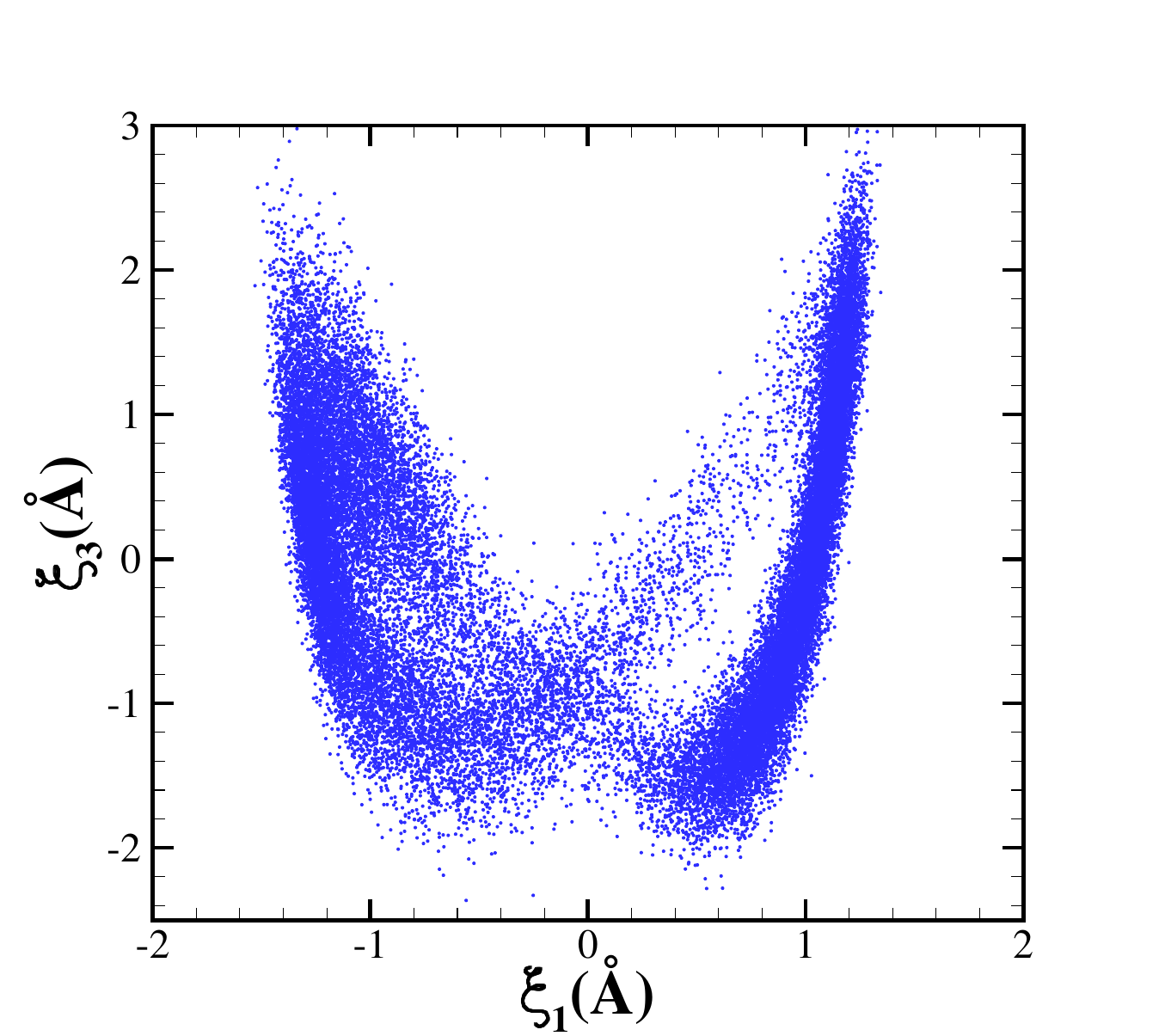}} \\
  \subfigure[]{\includegraphics*[scale=0.25]{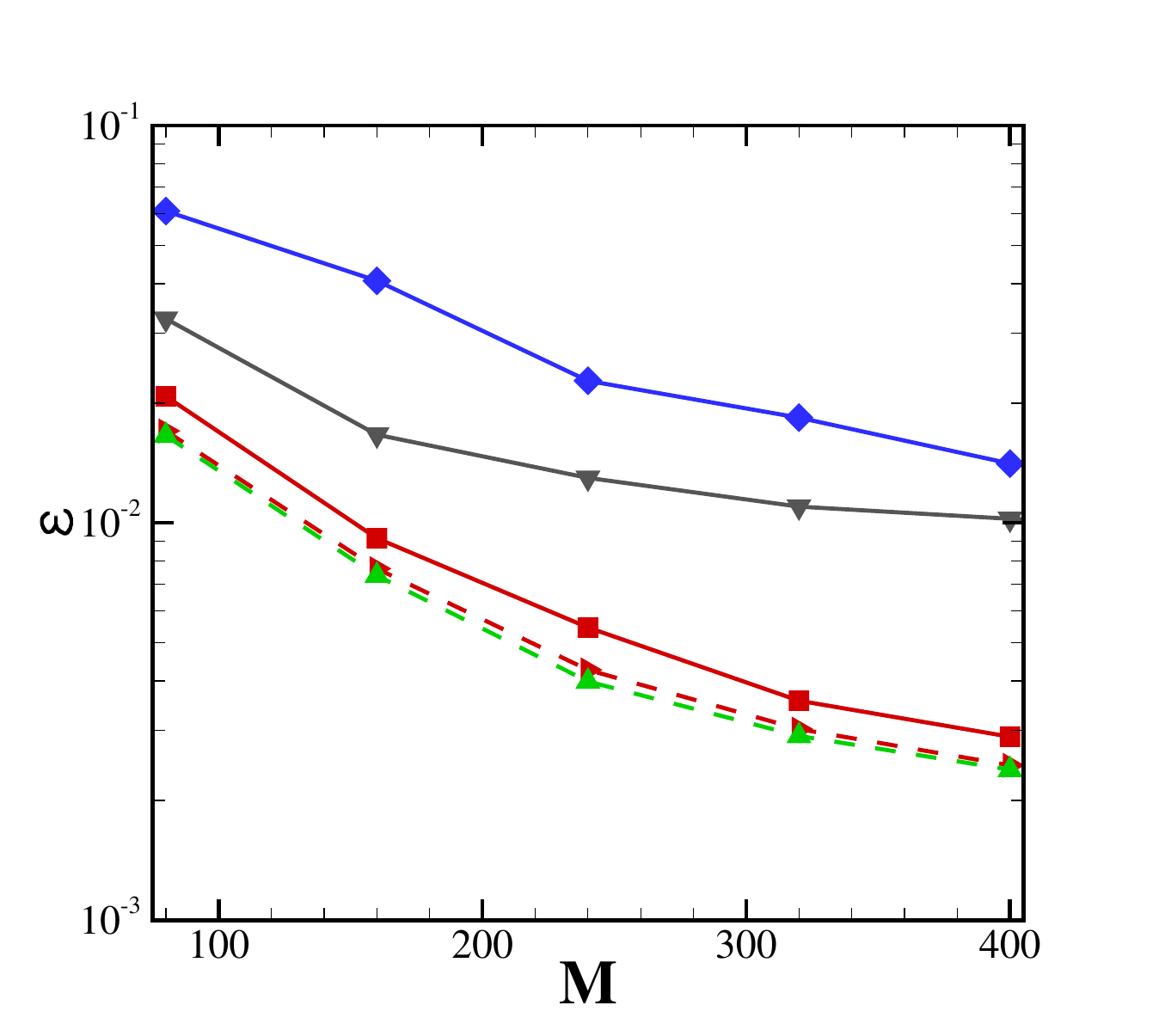}}
  \subfigure[]{\includegraphics*[scale=0.25]{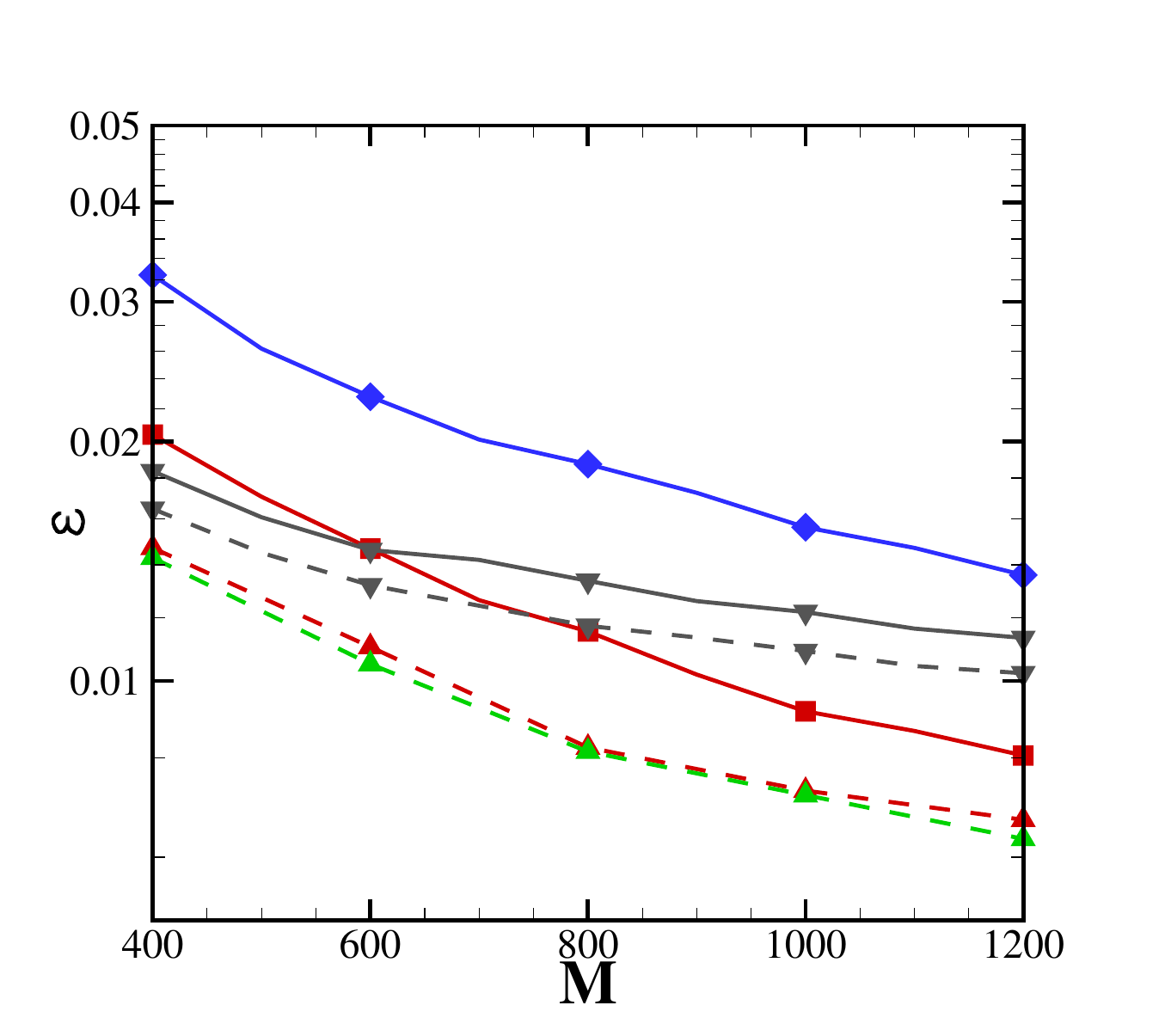}}
  \caption{{The present method based on data-driven basis construction and sparsity 
    enhancement rotation yields the most accurate surrogate
  model for molecular systems with mutually dependent non-Gaussian density distributions.} 
  (a-b) Sampling points representing the joint distributions $\left(\xi_1, \xi_2\right)$ (left) and $\left(\xi_1, \xi_3\right)$ (right).
  (c-d) Relative $l_2$ error of the polar solvation energy (left) and the local \ac{SASA} (right) of an individual atom
  (the H9 atom attached to the ortho-carbon atom) obtained with different numbers of training data $M$:
  the exact \black{\ac{amdP}} orthonormal basis with respect to $\bx$ 
    (``\textcolor{red}{\protect\rectanglesolidline}'') and 
  $\bm\chi$ (``\textcolor{red}{\protect\rectangledashline}''); 
  %near-orthonormal basis with respect to $\bx$ (``\textcolor{green}{\protect\trianglesolidline}'') and
  %$\bm\chi$ (``\textcolor{green}{\protect\triangledashline}'');
  Hermite basis with respect to $\bx$ (``\textcolor{gray}{\protect\downtrianglesolidline}'')
  and $\bm\chi$ (``\textcolor{gray}{\protect\downtriangledashline}'');
  Legendre basis with respect to $\bx$ (``\textcolor{blue}{\protect\diamondsolidline}'');
  the near-orthonormal \black{\ac{amdP}} basis with respect to $\bm\chi$ (``\textcolor{green}{\protect\triangledashline}'').
  } \label{fig:err_mol_dim_12_p_4}
\end{figure}
%%%
Although the individual components of $\bx$ are uncorrelated, the joint density distributions are mutually dependent and deviate from the standard Gaussian distributions. 

We chose the polar solvation energy and \ac{SASA} as the target \acp{QoI} for this system.
The polar solvation energy was modeled by the Poisson-Boltzmann equation \cite{ren_biomolecular_2012,baker_biomolecular_2005}
\begin{equation}
  -\nabla \cdot (\epsilon_f(\bm x; \bx) \nabla \varphi(\bm x; \bx)) = \rho_f(\bm x; \bx)
  \label{eq:PB}
\end{equation}
which relates the electrostatic potential $\varphi$ to a dielectric coefficient $\epsilon_f$ and a fixed charge distribution $\rho_f$. 
Equation~\eqref{eq:PB} is typically solved with Dirichlet boundary conditions set to an analytical asymptotic solution of the equation for an infinite domain. 
The dielectric coefficient $\epsilon_f$ implicitly represents the boundary between the atoms of the molecule and the surrounding solvent:  the coefficient changes rapidly across this boundary from a low dielectric value in the molecular interior to a high dielectric value in the solvent.
The charge distribution $\rho_f$ is generally modeled as a collection of $\delta$-like functions centered on the atoms of the molecule with magnitudes proportional to the atomic partial charges.
Both $\epsilon_f$ and $\rho_f$ are dependent on the instantaneous molecular structure (i.e., $\bx$).
The polar solvation energy was calculated from 
\begin{equation}
  G_p(\bx) = \int \rho_f(\bm x; \bx) \left( \varphi(\bm x; \bx) - \varphi_h(\bm x; \bx) \right) d \bm x
\end{equation}
where $\varphi_h$ is a reference potential obtained from solution of
\begin{equation}
  -\epsilon_h \nabla^2 \varphi_h(\bm x; \bx) = \rho_f(\bm x; \bx)
\end{equation}
where $\epsilon_h$ is a constant reference dielectric value.
We used the \ac{APBS} software to solve the equations above \cite{APBS_2018}.
Besides the solvation energy of the whole molecule, we also studied a local property like the \ac{SASA} of an individual atom 
(the H9 atom attached to the ortho-carbon atom of the benzyl bromide molecule, see Figure \ref{fig:mol_blb}) 
by the Shrake-Rupley algorithm \cite{Shrake_JMB_1973} using \ac{APBS}.
Details of the \ac{APBS} calculations are presented in \ref{app:sim}.

Figures~\ref{fig:err_mol_dim_12_p_4}(c) and (d) show the relative $l_2$ error of the constructed surrogate model $\tilde{f}(\bx)$ for the solvation energy and \ac{SASA} using a \nth{4}-order \ac{gPC} expansion with $N = 1820$ basis functions.
{For both \acp{QoI}, the near-orthonormal and orthonormal
 bases with respect to the rotated variable $\bm\chi$ (dashed curves) 
 yield similar error which is much smaller than the error of 
 Legendre and Hermite bases. A possible explanation for the 
 similar performance of the near-orthonormal and orthonormal bases is the
 closeness of the basis bound estimates for these two bases (see 
 Table \ref{tab:BioMol_d_12_p_4} in \ref{app:basis_bound}).

\begin{figure}[tbp]
  \center
  \subfigure[]{\includegraphics*[scale=0.25]{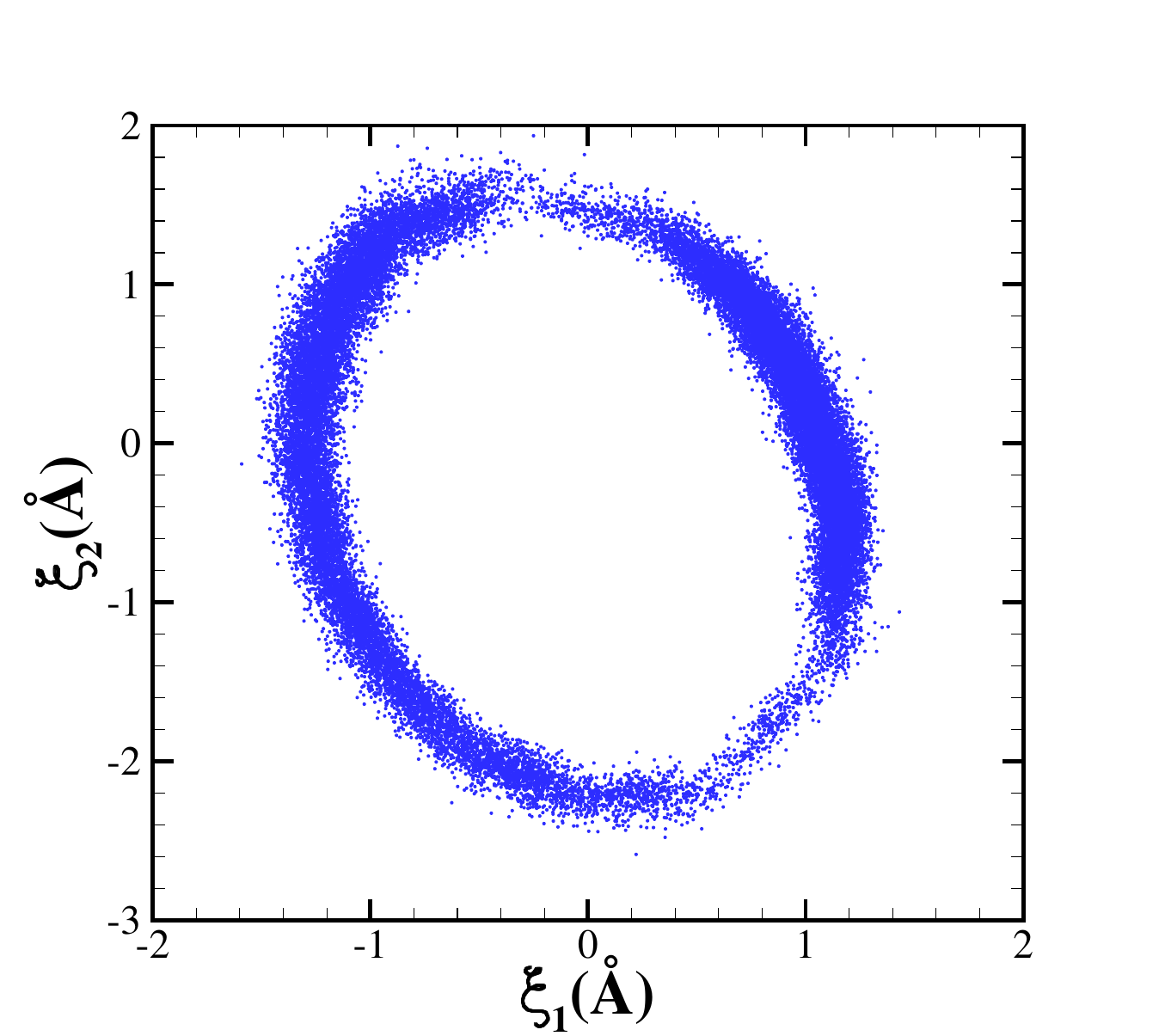}}
  \subfigure[]{\includegraphics*[scale=0.25]{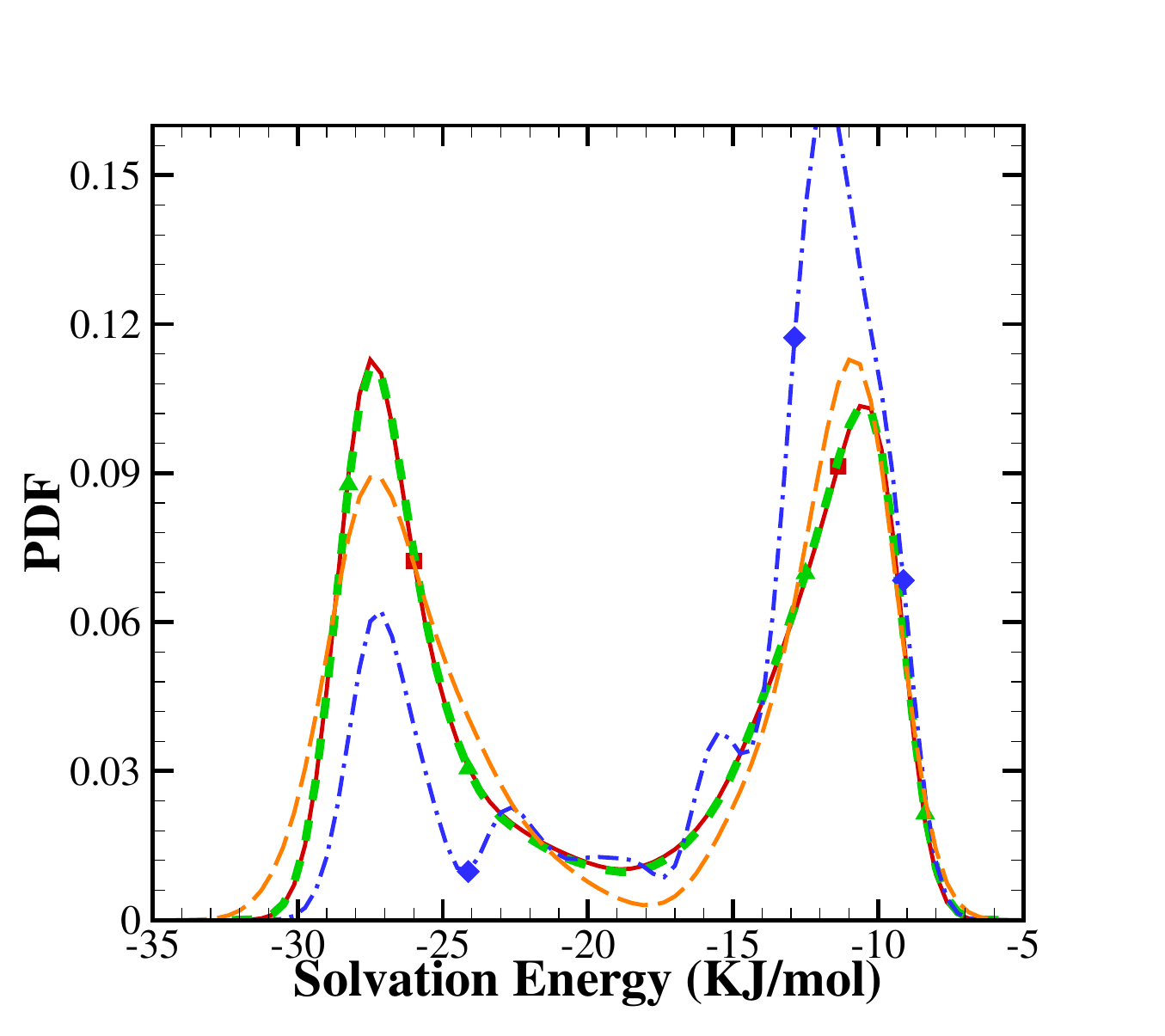}}
  \caption{{The present method yields the most accurate prediction on the \ac{PDF} of the QoI for the molecular systems.
  Direct fitting of the underlying density $\omega(\bx)$ using Gaussian Mixture model may induce biased error to
    the \ac{PDF} prediction.}
  (a) Fitted random variables $(\xi_1,\xi_2)$ with Gaussian mixture models.
  (b) \ac{PDF} of the solvation energy obtained with the Gaussian Mixture model and the present data-driven approach.
    ``\textcolor{red}{\protect\rectanglesolidline}'': reference solution obtained from $2\times10^5$ \ac{MC} samples;
    ``\textcolor{blue}{\protect\diamonddashdotline}'': direct \ac{MC} sampling using the same set of $200$ samples; 
    ``\textcolor{green}{\protect\triangledashline}'': present method using the same set of $200$ samples; 
    ``\textcolor{orange}{\protect\dashline}'': fitting Gaussian Mixture model using $800$ samples.
   }
   \label{fig:pdf_mol_dim_12_p_4}
\end{figure}
%%%

Instead of the direct construction of $\tilde{f}(\bx)$ using data-driven basis functions, another possible 
approach to characterize the uncertainty of the molecular system is to fit the distribution density $\omega(\bx)$ with 
a distribution model such as a Gaussian Mixture model. Figure~\ref{fig:pdf_mol_dim_12_p_4} (a) shows a scatter plot of the joint distribution $\left(\xi_1,\xi_2\right)$ extracted from the fitted Gaussian mixture distribution $\tilde{\omega}(\bx)$ using $7$ Gaussian modes.
Accordingly, we can construct the surrogate model for each Gaussian mode using standard Hermite basis function.
However, it is well-known that accurate construction of $\omega(\bx)$ is a numerically challenging problem for $d > 4$.
As shown in Figure~\ref{fig:pdf_mol_dim_12_p_4}(b), direct fitting $\omega(\bx)$ by $\tilde{\omega}(\bx)$ 
induces non-negligible error and leads to biased prediction of the \ac{PDF} of the solvation energy. 
\black{
Furthermore, we lose the one-to-one mapping between the individual conformation state $\bx$ and the \acp{QoI} through the 
constructed surrogate model $\tilde{f}(\bx)$.}

%%%

\section{Summary}
\label{sec:summary}

In this study, we have developed a \ac{DSRAR} framework for constructing surrogate models \black{irrespective of the mutual dependence between the  components 
of random inputs} using limited training points. 
\black{
To the best of our knowledge, this problem has not been
addressed by  previous \ac{UQ} studies based on polynomial 
chaos expansions. The \ac{DSRAR} framework does not assume  mutual 
independence between the components of random inputs and therefore can be
applied to \ac{UQ} in complex systems
where information about the underlying random distribution can be implicit.
}
%
%\black{
%In particular, we do not assume that the random components are mutually independent. To the best 
%of our knowledge, the present framework takes a different starting point from the most of 
%the previous \ac{UQ} studies based on polynomial chaos expansion. Accordingly, the present method
%can be well-suited for uncertainty quantification of complex systems,
%where the knowledge of the underlying randomness can be implicit.
%}
%
To construct the surrogate model, this framework uses data-driven 
\black{\ac{amdP}} basis 
construction and a sparsity-enhancing rotation procedure which leads 
to more accurate recovery of the sparse representation of the target function. 
%
%In particular, the present method does not assume that the random components
%are mutually independent. To the best of our knowledge, this is different from the majority of 
%the previous \ac{UQ} studies
%
%
%
%
The method benefits from both the orthonormal 
basis expansion and the enhanced sparsity of the expansion coefficients. 
\black{With the assumption that there exists a sparse representation of the surrogate model},
the \ac{DSRAR} approach can be applied to challenging \ac{UQ} problems 
under two widely encountered situations: (\Rmnum{1}) probability measure 
implicitly represented by a large collection of samples and (\Rmnum{2}) 
non-Gaussian probability measures with explicit (analytical) forms.
For systems with explicit knowledge of the probability measure, our 
method exploits sparser representations of \acp{QoI} while retaining 
proper orthogonality with respect to rotated variables.
For systems with randomness implicitly represented by a large collection 
of random samples, we also proposed a heuristic method to construct a 
\emph{near-orthonormal} basis in addition to the exact orthonormal basis 
with respect to the discrete measure. The near-orthonormal basis 
shows {\color{black}a} smaller basis bound and empirically yields more accurate 
representations. 
The numerical examples show the 
effectiveness of our method for realistic problems on \black{quantifying 
uncertainty propagation in molecular system under 
conformational fluctuations} as well 
as \acp{PDE} with arbitrary underlying probability measures.

For future study, we note that several issues not considered in the 
present work could further improve the performance of the present \ac{DSRAR} framework.
The heuristic approach to constructing near-orthonormal basis introduced in 
this study yields smaller basis bounds and more accurate representations than 
existing methods. However, we do not have the theoretical analysis to formally 
show that the near-orthonormal basis is optimal and to establish the conditions 
under which it outperforms the exact orthonormal basis. It would be 
interesting to investigate 
different approaches of data-driven basis construction to further improve the 
properties of measurement matrix for \ac{CS} purposes. \textcolor{black}{For instance,
if new data becomes available after the surrogate construction, it is worth
exploring how to use the new information to design more sophisticated 
(cross-validation) strategies to optimize the orthonormal threshold 
values and the basis construction procedure.} 
Furthermore, our study used a standard $\ell_1$ minimization approach 
for relaxing the \ac{CS} 
problem and recovering a sparse solution of the under-determined system.
However, other optimization approaches can be employed when the measurement 
matrix is highly coherent when $\ell_1$ minimization is not necessarily optimal.
Finally, it would be interesting to employ the developed \ac{DSRAR} approach for
\ac{UQ} study in other complex biological systems \cite{Bajaj_ACM_2016, Bajaj_JCB_2018}.
Such results will be presented in a future publication.

\appendix
\section{The proof of Theorem ~\ref{thm:exact_recovery}} \label{app:exact_recovery}

\begin{proof}
  Let $\bm v \in \rm{Ker}~\bm A$ and $\bm x \neq \bm c$ another solution of $\bm A \bm x = \bm b$. 
  To show that $\bm c$ is the unique $l_1$ minimizer of $\bm A \bm c  = \bm b$, it is sufficient if 
  \begin{equation}
    \Vert \bm v_{T_{\ba}}\Vert_1 < \Vert \bm v_{T_{\ba}^c}\Vert_1,
    \label{eq:null_condition}
  \end{equation}
  which gives
  \begin{equation}
    \begin{split}
      \Vert \bm c \Vert_1 &\le \Vert \bm c - \bm x_{T_{\ba}}\Vert_1 + \Vert  \bm x_{T_{\ba}}\Vert_1
      = \Vert \bm c_{T_{\ba}} - \bm x_{T_{\ba}}\Vert_1 + \Vert  \bm x_{T_{\ba}}\Vert_1
      = \Vert \bm v_{T_{\ba}}\Vert_1 + \Vert  \bm x_{T_{\ba}}\Vert_1  \\
      &< \Vert \bm v_{T_{\ba}^c}\Vert_1 + \Vert  \bm x_{T_{\ba}}\Vert_1 = \Vert \bm x \Vert_1.
    \end{split} 
  \end{equation}
  To satisfy \eqref{eq:null_condition}, we partition $T_{\ba}^c$ into $T_{\ba}^c =  T_{\ba,1}^c \bigcup  T_{\ba,2}^c \bigcup \cdots$, where $T_{\ba,1}^c$ is the index set of $s$ largest absolute entries of $\bm v$ in $T_{\ba}^c$, $T_{\ba,2 }^c$ is the index set of $s$ largest absolute entries of $\bm v$ in $T_{\ba}^c T_{\ba,1}^c$.
  Accordingly, 
  \begin{equation}
    \Vert \bm v_{T_{\ba}}\Vert_2^2 \le \frac{1}{1-\delta_s} \Vert \bm A \bm v_{T_{\ba}}\Vert_2 
    = \frac{1}{1-\delta_s} \sum_{k=1} \left\langle \bm A \bm v_{T_{\ba}}, \bm A (-\bm v_{T_{\ba,k}^c}) \right\rangle \le \frac{\theta_s}{1-\delta_s} \sum_{k=1} \Vert \bm v_{T_{\ba}}\Vert_2 \Vert \bm v_{T_{\ba,k}^c}\Vert_2,
  \end{equation}
  which gives $\Vert \bm v_{T_{\ba}}\Vert_2 \le \frac{\theta_s}{1-\delta_s} \sum_{k=1}  \Vert \bm v_{T_{\ba,k}^c}\Vert_2$.
  The remaining of the proof is straightforward and follows {\rm Theorem 2.6} of Rauhut~\cite{Rauhut_2010CsSM}.
  By the Cauchy-Schwarz inequality, we obtain 
  \begin{equation}
    \Vert \bm v_{T_{\ba}}\Vert_1 \le  \frac{\theta_s}{1-\delta_s} \left( \Vert \bm v_{T_{\ba}}\Vert_1 + \Vert \bm v_{T_{\ba}^c}\Vert_1 \right).
  \end{equation}
  Equation~\eqref{eq:null_condition} follows if $\frac{\theta_s}{1-\delta_s} < 0.5$. 
  %\left\langle %\bm\Psi \bm v_{T_{\ba}}, \bm\Psi (-\bm v T_{\ba,1}^c) + \bm\Psi (-\bm v T_{\ba,2}^c) + \cdots
  %\right\rangle
\end{proof}

\begin{rem}
  We emphasize that Theorem~\ref{thm:exact_recovery} holds only for the given index set $T_{\ba}$; it provides a metric to examine the recovery accuracy with respect to measurement matrix $\bm A$ and should not be viewed as the sufficient condition for exact recovery of \emph{arbitrary $s$-sparse vector} via $l_1$-minimization (see canonical references \cite{Candes_2005Decoding, Candes_2006Stablesrec,Rauhut_2010CsSM} for details). 
  Theorem~\ref{thm:exact_recovery} also indicates that, for the given index set $T_{\ba}$, small $\|\bm{A}^*_{T_{\ba}}\bm{A}_{T_{\ba}}-I\|_2$ will promote the recover of $\bm v_{T_{\ba}}$.
\end{rem}

\section{Measurement matrix and basis bounds}\label{app:basis_bound}
\subsection{Null space of measurement matrix from Section \ref{sec:basis_comparison}}
Let $\tilde{\bm c} = \bm c + \bm v$, $\bm v \in \textrm{Ker}~\bm{A}$ where $\bm{A}$ is the measurement matrix defined in \eqref{eq:matA}. 
From the null space property \cite{Rauhut_2010CsSM}, $\tilde{\bm c}$ does not fully recover $\bm c$ by 
$\ell_1$ minimization (i.e., equation \eqref{eq:ape_L1}) only if $\Vert  \tilde{\bm c} \Vert_1 < \Vert \bm c\Vert_1$. As a \emph{necessary 
condition} for the failure of recovery, it requires 
\begin{equation}
\left \Vert \bm v_{T_{\ba}}\right \Vert_1 > \left \Vert \bm v_{T_{\ba}^c}\right\Vert_1,
\label{eq:l1_failure}
\end{equation}  
where $T_{\ba}^c$ refers to the complement of $T_{\ba}$. Accordingly, different null space of measurement matrix $\bm{A}$ 
generally leads to different recovery error.

We examined the above necessary condition \eqref{eq:l1_failure} for different measurement matrices by randomly choosing a non-zero index set 
$T_{\ba}$ with $\left\vert T_{\ba}\right\vert = 50$ and $M = 180$.
For $\bm A$ constructed by both basis sets, we collected $1000$ normalized  $\bm v\in\rm{Ker}~\bm A $ that satisfy $\left \Vert \bm v_{T_{\ba}}\right \Vert_1 > \left \Vert \bm v_{T_{\ba}^c}\right\Vert_1$.
Figure~\ref{fig:contour_null_vector} shows the density contour of individual component $\left\vert \bm v_{i'}\right\vert$ in log-scale, where $i'$ refers to the index sorted by magnitude in descending order.
%%%
\begin{figure}[tbp]
  \center
  \subfigure[]{\includegraphics*[scale=0.25]{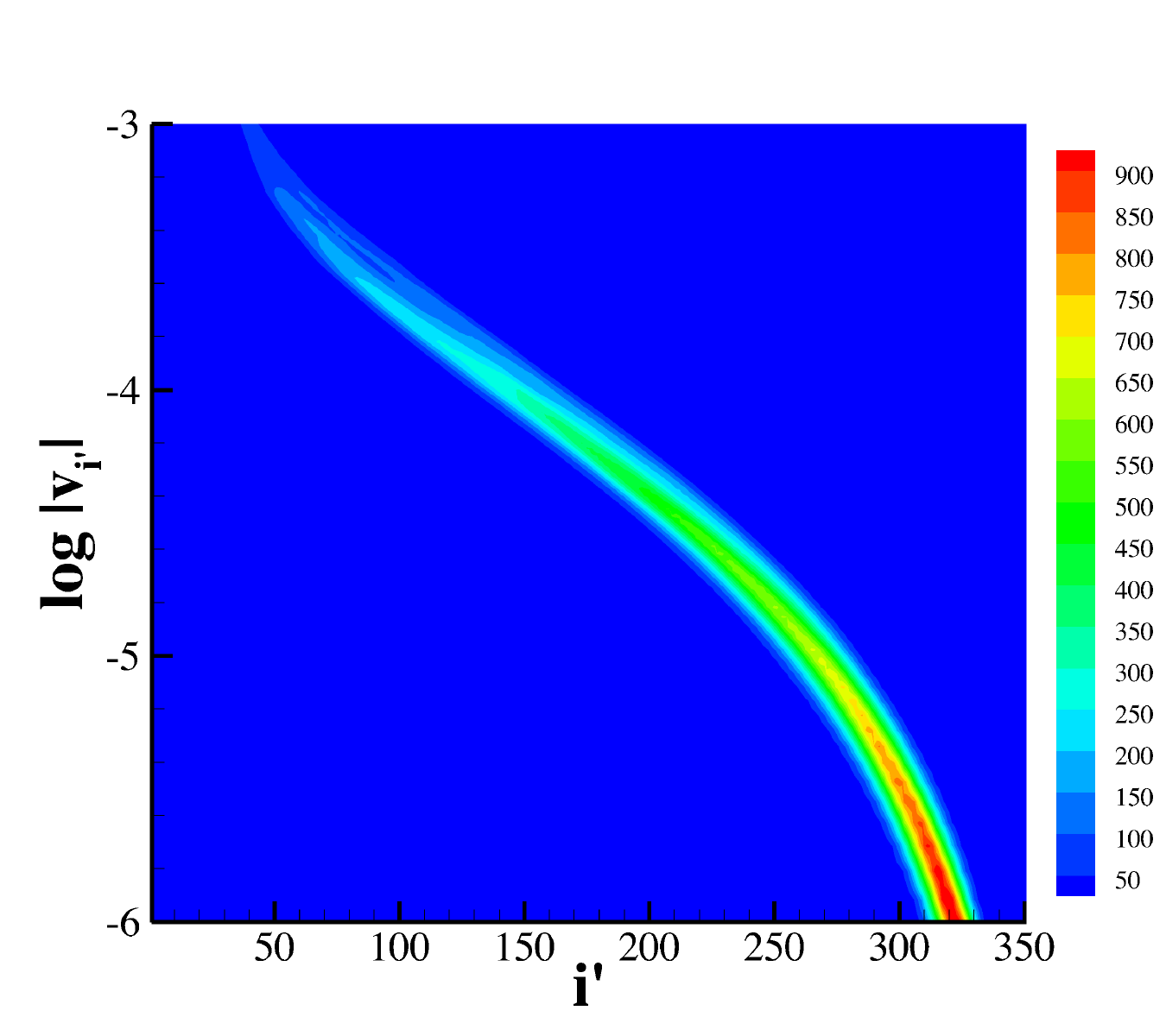}}
  \subfigure[]{\includegraphics*[scale=0.25]{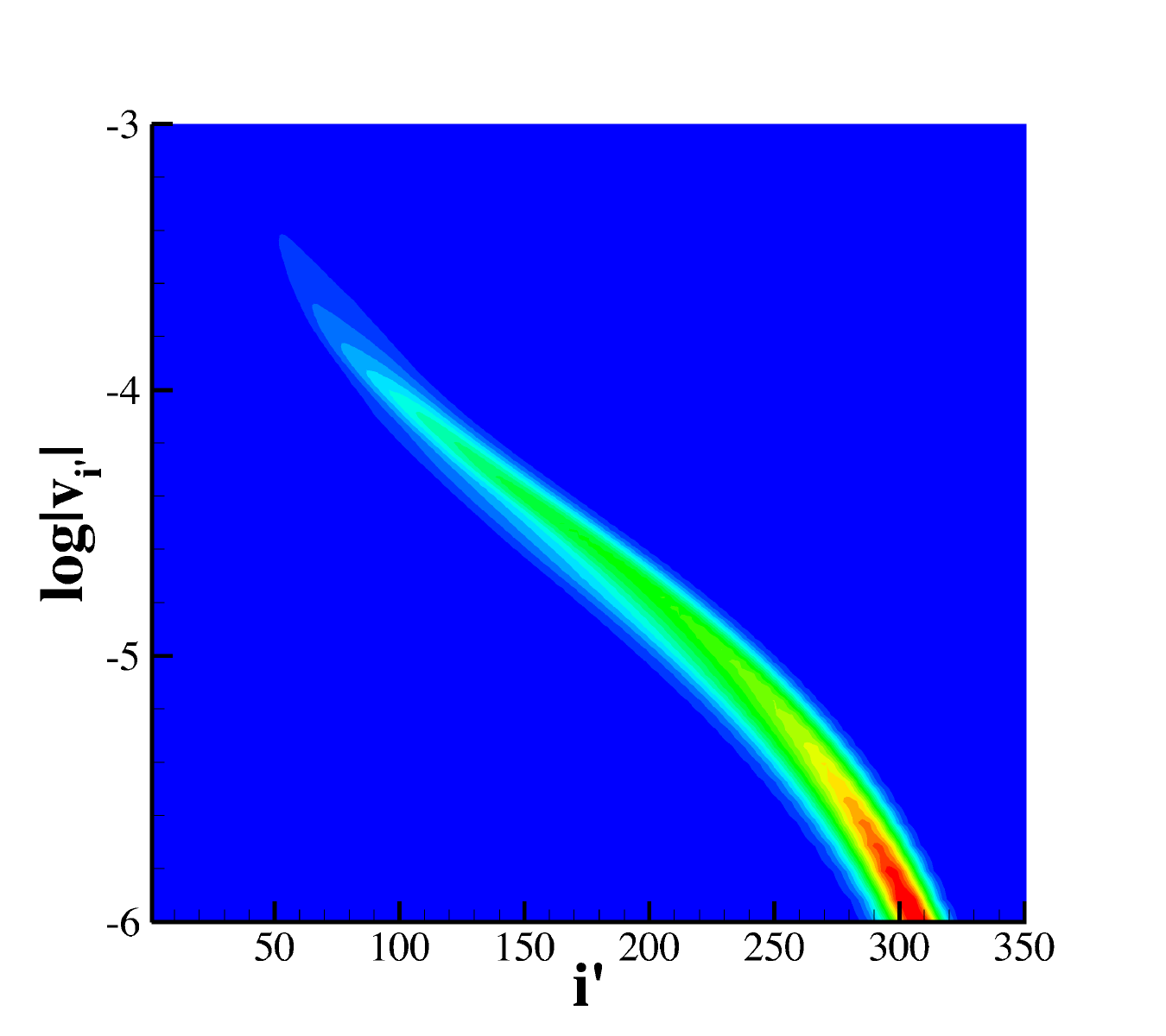}}
  \caption{{The null spaces of measurement matrices constructed by the exact and near-orthonormal bases are different under $\left \Vert \bm v_{T_{\ba}}\right \Vert_1 > \left \Vert \bm v_{T_{\ba}^c}\right\Vert_1$---a necessary
    condition for $\bm c$ not being recoverable exactly.}
  Density contour of the normalized null space vector component $\log\vert \rm {v}_{i'}\vert $ (sorted by magnitude) of the measurement matrix $\bm A$ constructed by orthogonal (a) and near-orthogonal basis functions (b) that satisfy $\left \Vert \bm v_{T_{\ba}}\right \Vert_1 > \left \Vert \bm v_{T_{\ba}^c}\right\Vert_1$ and $\Vert \bm v\Vert_2 = 1$.} \label{fig:contour_null_vector}
\end{figure}
%%%
The two basis sets demonstrate different distributions of $\log\left\vert \bm v_{i'}\right\vert$, which likely contribute 
to the different recovery errors shown in Figure~\ref{fig:err_no_sparse_vector}.

\subsection{Basis bounds}
%It is useful to introduce the ``basis bound'' for the numerical results from our method.
The lower bound of the required number of samples $M$ given in Theorem \ \ref{thm:boundedBOS} suggests that bases with smaller basis bounds $K$ are preferred.
We expect that smaller basis bounds will correlate with higher accuracy representations.
For the constructed basis set $\psi_i(\bx), i = 1,\cdots, N$, we define the basis 
bound $\tilde{K}$ on the given data set $S$ by
\begin{equation}\label{eq:basis_bound}
\tilde{K} := \frac{1}{\vert S_{M_{\sigma}}\vert} \sum_{\bx\in S_{M_{\sigma}}}\vert k(\bm\xi)\vert,
\end{equation}
%\end{equation}
where the set $S_{M_{\sigma}}$ is defined by $S_{M_{\sigma}} = \left
\{\bx\Big\vert \vert k(\bx) - \mathbb{E}\left[k\right]\vert > M_{\sigma} \sigma\left[k\right], \bx \in S\right\}$. 
Here $\displaystyle k(\bx) :=\mathop{\max}_{i} \vert \psi_i(\bx) \vert$ denotes the maximum magnitude for 
an individual sampling point $\bx$, $\mathbb{E}\left[k\right]$ and $\sigma\left[k\right]$  represent the 
mean and the standard deviation of $k(\bx)$ on $S$ with respect to the discrete measure $\nu_S$.
%and $M_{\sigma}$ defines the spanned range for sampling $K$. 
In this study, we present $\tilde{K}$ as an indication of the difference between the exact and near-orthonormal basis function. 
In compressive sensing, the measurement matrix only consists of limited number of samples. Therefore, we 
employ the mean of the tails in the basis bounds as an indicator of the upper bound of the largest 
entry values from the measurement matrix. $M_{\sigma}$ defines the range of this tail set. 
We choose $M_{\sigma} = 5$ if not specified otherwise.

\begin{table}[!h]
\centering
\caption{$\tilde{K}$ of constructed basis set for Gaussian mixture system $d = 25$, $p = 2$ and $N_s = 1\times 10^5$.}
\begin{tabular}{C{8em}|C{6em} C{6em} C{6em} C{6em} C{6em}}
\hline\hline
$M_{\sigma}$ & $3$ & $4$ & $5$ & $6$ &$ \displaystyle \mathop{\max}_{\bx\in S} k(\bx)$ \\
\hline
$\tilde{K}_{\rm orth}$ & 10.359 & 12.048 & 13.895 & 15.513 & 22.208\\ 
$\tilde{K}_{\rm near-orth}$ & 9.622 & 11.196 & 12.867 & 14.448 & 18.790\\ 
\hline\hline
\end{tabular}
\label{tab:GM_d_25_p_2}
\end{table}

Following the definition by Equation \eqref{eq:basis_bound}, we examine the basis bound $\tilde{K}$ of the numerical
examples presented in this study. Table \ref{tab:GM_d_25_p_2} shows the results of Gaussian mixture system  
$\left\{\bx^{(i)}\right\}, i = 1, \cdots, N_s$ with $N_s = 1\times 10^5$, $d = 25$ and $p = 2$ which is defined in Section \ref{sec:basis_comparison}. For different
values of $M_{\sigma}$, $\tilde{K}$ of the near orthogonal basis shows consistently smaller values than
the values of the exact orthogonal basis set. 

Table \ref{tab:GM_d_25_p_3} shows the basis bound $\tilde{K}$ of the Gaussian mixture system which is studied 
in Section \ref{sec:high_d_poly} with $N_s = 2\times 10^5$, 
$d = 25$ and $p = 3$ . The values of $\tilde{K}$ for the near orthogonal basis are consistently smaller than 
the value for the exact orthogonal basis set no matter on the original random sample set or the rotated sample set. 
Furthermore, we present the basis bounds on the rotated sampling set $\left\{\bm\chi_M^{\left(i\right)}\right\}_{i=1}^{N_s}$, where
the subscript ``$M$'' refers to the different number of training points utilized to construct the surrogate model
$X(\bx)$. The near-orthogonal basis yields smaller $\tilde{K}$ than the exact orthogonal basis in each case. 

Similarly, Table \ref{tab:BioMol_d_12_p_4} shows $\tilde{K}$ 
of the constructed basis for uncertainty quantification of the molecular solvation energy 
($d = 12$, $p=4$ and $N_s = 2\times 10^5$), which is studied in Section \ref{sec:mole_example}. The near-orthogonal basis yields smaller values consistently for different number ($\bm{\chi}_M$) of training points. 

\begin{table}[!h]
\centering
\caption{$\tilde{K}$ of constructed basis set for Gaussian mixture system $d = 25$, $p = 3$ and $N_s = 2\times 10^5$.}
\begin{tabular}{C{8em}|C{6em} C{6em} C{6em} C{6em} C{6em}}
\hline\hline
 & $\bx$ & $\bm{\chi}_{M=400}$ & $\bm{\chi}_{M=1200}$ & $\bm{\chi}_{M=1600}$ & $\bm{\chi}_{M=2400}$ \\
\hline
$\tilde{K}_{\rm orth}$ & 32.497 & 32.522 & 32.079 & 33.142 & 32.308 \\ 
$\tilde{K}_{\rm near-orth}$ & 28.320 & 29.811 & 29.407 & 29.512 & 29.192\\ 
\hline\hline
\end{tabular}
\label{tab:GM_d_25_p_3}
\end{table}

\begin{table}[!h]
\centering
\caption{$\tilde{K}$ of constructed basis set for molecular system $d = 12$, $p = 4$ and $N_s = 2\times 10^5$.}
\begin{tabular}{C{8em}|C{6em} C{6em} C{6em} C{6em} C{6em}}
\hline\hline
 & $\bm{\chi}_{M=80}$ & $\bm{\chi}_{M=160}$ & $\bm{\chi}_{M=240}$ & $\bm{\chi}_{M=320}$ & $\bm{\chi}_{M=400}$  \\
\hline
$\tilde{K}_{\rm orth}$ & 40.596 & 39.914 & 39.789 & 39.218 & 39.142 \\ 
$\tilde{K}_{\rm near-orth}$ & 39.970 & 39.278 & 39.290 & 38.528 & 38.631\\ 
\hline\hline
\end{tabular}
\label{tab:BioMol_d_12_p_4}
\end{table}

\section{Other metrics of surrogate model}
\black{
Besides the relative $l_2$ error, we have also computed the predictivity coefficients $Q_2$ for 
the test cases of Gaussian Mixture systems (with $d = 25$ and
$p = 3$) and the molecular systems. Similar to Ref. \cite{Marrel_Q2_2009}, $Q_2$ is defined by
\begin{equation}
Q_2 = 1 - \int(f(\bx) - \tilde{f}(\bx))^2 \dif \nu_{S_2}(\bx) \big/ \int \left(f(\bx) - \bar{f}\right)^2 \dif \nu_{S_2}(\bx) ,
\end{equation}
where $\bar{f}$ represents the mean of QoI on $S_2$. The results are listed in Tab. \ref{tab:Q_2_GM_mol}, where the surrogate models are
constructed by the present data-driven basis approach. }

\black{
\begin{table}[!h]
\centering
\caption{The predictivity coefficient $Q_2$ for polynomial function with 
  Gaussian Mixture measure ($d = 25$ and $p = 3$) and the molecular system for
    solvation energy and SASA of atom $\rm{H9}$.}
\begin{tabular}{C{8em} C{3em} C{4em} C{4em} C{4em} C{4em} C{4em}}
\hline\hline
molecule solvation & $M$ & 80 & 160 &  240 & 320 & 400 \\
 & $Q_2$ & 0.995715 & 0.999132 &  0.999731 & 0.999864 & 0.999911 \\
\hline
molecule SASA & $M$ & 200 & 300 &  400 & 500 & 600 \\
 & $Q_2$ & 0.988675 & 0.996069 &  0.998272 & 0.998709 & 0.999027 \\
\hline
Gaussian Mixture & $M$ & 200 & 300 &  400 & 500 & 600 \\
 & $Q_2$ & 0.998372 & 0.999347 &  0.999844 & 0.999892 & 0.999941 \\
\hline\hline
\end{tabular}
\label{tab:Q_2_GM_mol}
\end{table}
}
\black{
With the constructed surrogate model, we can further compute the Sobol' sensitivity indices for 
\ac{QoI} with dependent random variables. In brief, $f(\bx)$ is expanded by
\begin{equation}
f(\bx) = \eta_0(\bx) + \sum_{\bm\beta \in \Theta^d} \eta_{\bm\beta}(\bx),
\end{equation}
where $\Theta^d$ represents the collection of all subsets of $[1~:~d]$ and $\eta_{\bm\beta}(\bx)$
satisfies 
%\begin{equation}
$\mathbb{E}\left[\eta_{\bm\alpha}, \eta_{\bm\beta}\right] = 0$, if $\bm\alpha \subset \bm\beta$.
The sensitivity index $S_{\bm\beta}$ is given by
\begin{equation}
S_{\bm\beta} = \frac{\mathbb{V}(\eta_{\bm\beta}) + \sum_{\ba \cap \bm\beta \neq \ba, \bm\beta} \rm{Cov}(\eta_{\ba}, \beta_{\bm\beta})}
    {\mathbb{V}(f)}
\end{equation}
where $\mathbb{V}(\cdot)$ refers to the variance on $\nu_S$. 
We refer to Ref. \cite{Chastaing_sobol_2014} for the details. Fig. \ref{tab:Q_2_GM_mol} shows the first 
order sensitivity indices for the test cases of Gaussian Mixture systems 
($d = 25$, $p = 3$) and the biomolecular systems, where the surrogate models are
constructed by the present data-driven basis approach using $M = 800$, $M = 240$ and $M = 600$ 
training points, respectively. Based on the analysis, it is shown the dominant components are on the 
dimensions $(1, 2, 3, 6, 11, 13, 14, 15, 16, 20, 22, 24, 25)$, $(1, 2, 5)$ and $(1, 2, 4, 5, 7)$ ($90\%$ of 
total variance).
}
%%%
\begin{figure}[tbp]
  \center
  \subfigure[]{\includegraphics*[scale=0.21]{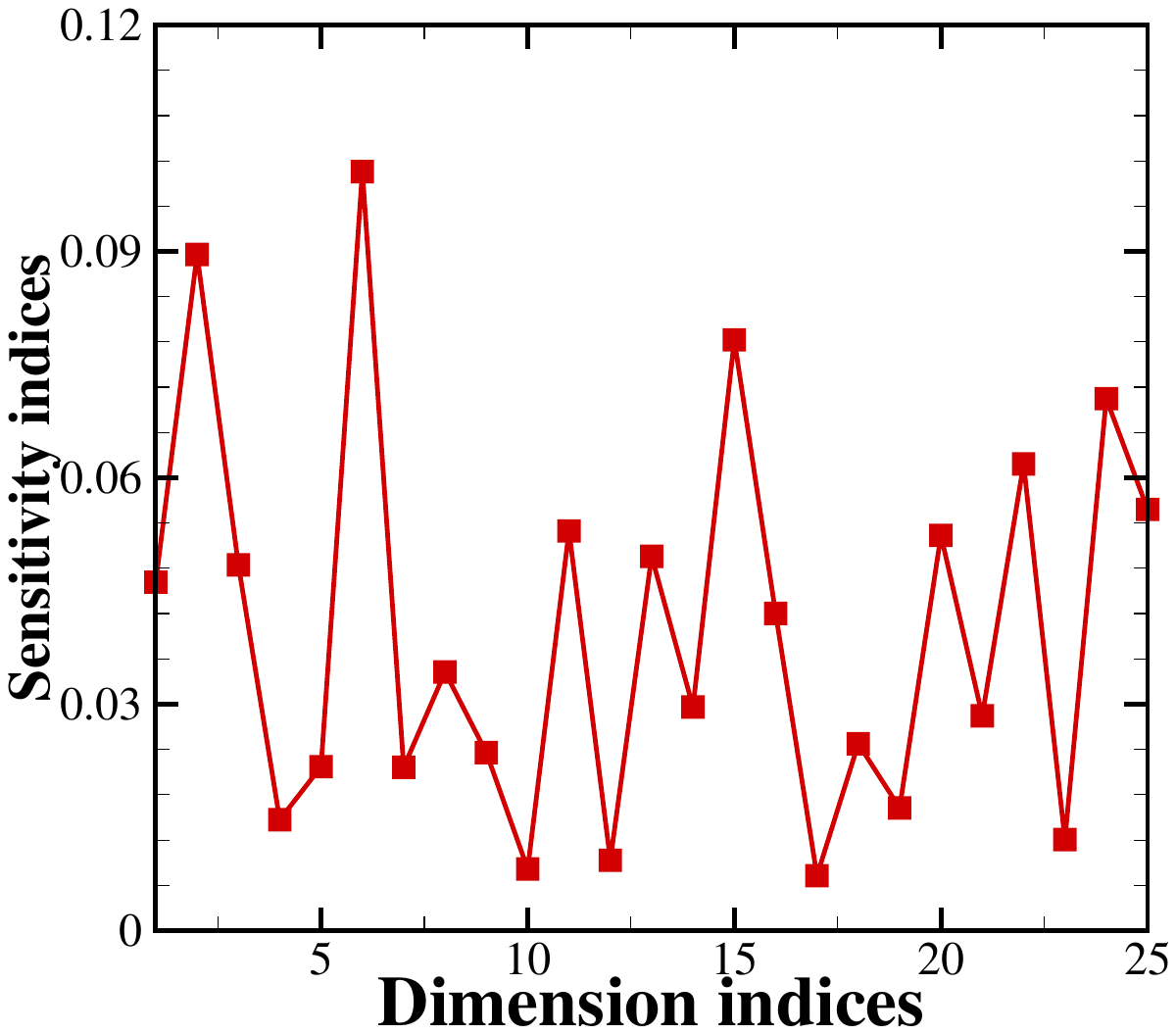}}
  \subfigure[]{\includegraphics*[scale=0.21]{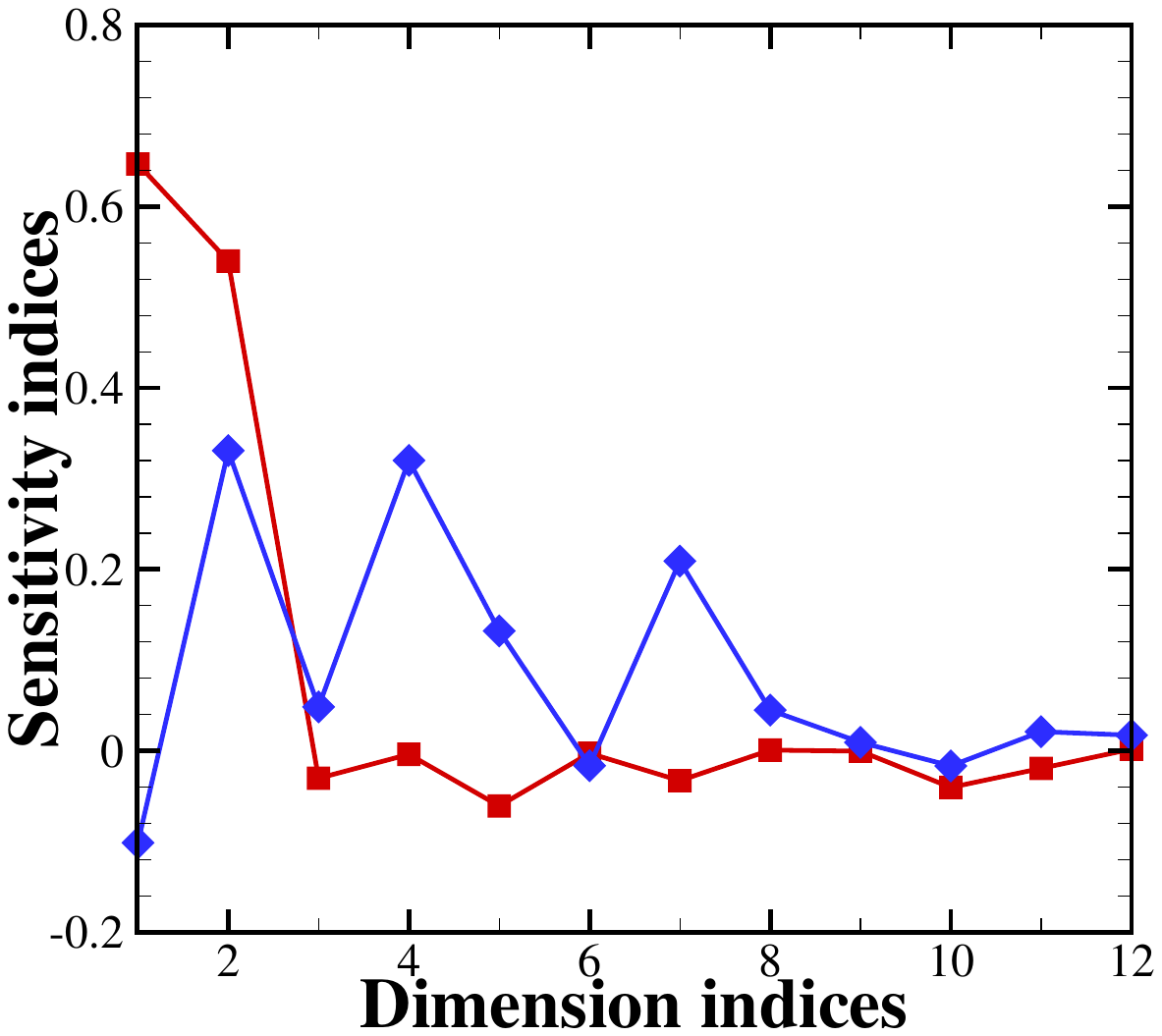}}
  \caption{The first-order Sobol' sensitivity indices for (a) polynomial function with 
  Gaussian Mixture mesure ($d = 25$, $p = 3$) (b) molecular system for
    solvation energy (``\textcolor{red}{\protect\rectanglesolidline}'')
    and SASA of atom $\rm{H9}$(``\textcolor{blue}{\protect\diamonddashdotline}'').
   } \label{fig:sobol_GM_mol}
\end{figure}
%%%

\section{Generation of Gaussian Mixture data set}
\label{app:GM_generation}
\black{
We employ Matlab to generate the Gaussian Mixture data set in Sec. \ref{sec:basis_comparison}
by calling the function gmdistribution($\bm\mu$, $\left\{\Sigma_i\right\}_{i=1}^3$, $\bm a$). Here
$\bm a = (0.5358, 0.1281, 0.3361)$. $\bm\mu$ is a $25\times 3$ random 
matrix with i.i.d.\ entries on $U[-2.5, 2.5]$. $\left\{\Sigma_i\right\}_{i=1}^3$ is a 
$25\times25\times3$ array where $\Sigma_i$ is defined by
\begin{equation}
  \bm\Sigma_i = (\bm \Upsilon_i {\bm \Upsilon_i}^T + \mathbf{I})/4,
\end{equation}
where $\bm\Upsilon_i$ is a random matrix with i.i.d.\ entries from $\mathcal{U}[0,1]$ for $i=1,2,3$.
$\bm\mu$ and $\Upsilon_i$ are generated by calling Matlab function rand() consequently with
random number seed $200$.
}

\section{Molecular Dynamics simulation and calculation details}
\label{app:sim}
We performed all-atom MD simulation of benzyl bromide in water using GROMACS 5.1.2 \cite{GROMACS}.
The simulation system included a benzyl bromide molecule (see Figure \ref{fig:mol_blb} for the molecule structure) 
 and 1011 water molecules.
The \ac{GAFF} \cite{RN1} was used for benzyl bromide parameters.
The partial charges of benzyl bromide molecule were calculated by RESP method \cite{RN2}.
Bond lengths of benzyl bromide were constrained using the LINCS algorithm \cite{RN3}.
The water molecule was modeled with the rigid TIP3P water model \cite{RN4}.
The bond lengths and angles were held constant through the SETTLE algorithm \cite{RN5}.
The system was equilibrated in the isothermal-isobaric ensemble for 10 ns at 300K and 1 bar after energy minimization.
The van der Waals cut-off radii was 1.0 nm.
Long-range electrostatics were calculated using a Particle Mesh Ewald (PME) summation with grid spacing of 0.12 nm.
The time step was 2 fs.
Isobaric-isothermal simulations were equilibrated using a V-rescale thermostat and Berendsen barostat.
Following equilibration, the simulation was run for a production period of 20 $\mu$s in a NVT ensemble with a Nos\'e-Hoover thermostat.
The trajectory was stored every 10000 time steps.

\begin{figure}[htbp]
\center
\includegraphics*[scale=0.6]{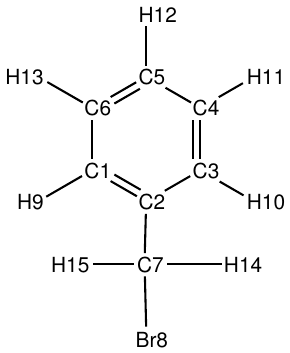}
\caption{Sketch of the molecule benzyl bromide with labeled atoms.}
\label{fig:mol_blb}
\end{figure}

APBS calculations \cite{BakerSJHM01, APBS_2018} were performed with $129^3$ grid points over a $40 \times 40 \times 40$ \AA$^3$ coarse grid domain with focusing to a $14 \times 14 \times 14$ \AA$^3$ fine grid domain with the grid origin located at the geometric center of the molecule.
The Poisson equation was solved with Dirichlet boundary conditions based on the asymptotic behavior of multiple point charges in a homogeneous dielectric medium.
The dielectric coefficient inside the domain used a van der Waals molecular volume definition with a dielectric value of 2.0 inside the molecule and 78.0 outside the molecule.
Charges were modeled by Dirac delta functions but discretized to the finite difference grid points using a cubic spline approximation.

\begin{acknowledgments}
We thank %\fxwarning*{Can you please add ORCID IDs or institutional affiliations for the people you are acknowledging?}
Ling Guo (Shanghai Normal University), Lei Wu (Princeton University), Wen Zhou (Colorado State University), and David Sept 
(University of Michigan, ORCID:0000-0003-3719-2483) for helpful discussions.
This work was supported by the U.S.\ Department of Energy, Office of Science, Office of Advanced Scientific Computing Research as part of the Collaboratory on Mathematics for Mesoscopic Modeling of Materials (CM4)
and by the National Institutes of Health grant R01 GM069702. The research was performed using resources available through Research Computing at Pacific Northwest National Laboratory.
%\fxwarning*{I think there is some official text you should use to acknowledge LDRD support.}
HL acknowledges grant support from AMS Simons Post-doctoral Travel Grant and PNNL Laboratory Directed Research \& Development (LDRD) under project
``Development of physics-compatible stochastic models for multiphysics systems''.
\end{acknowledgments}

\bibliographystyle{unsrt}
\bibliography{ab_uq,physics_system}

\begin{thebibliography}{100}

\bibitem{Lei_Yang_MMS_2015}
H.~Lei, X.~Yang, B.~Zheng, G.~Lin, and N.~A. Baker.
\newblock Constructing surrogate models of complex systems with enhanced
  sparsity: Quantifying the influence of conformational uncertainty in
  biomolecular solvation.
\newblock {\em SIAM Multiscale Model. Simul.}, 13(4):1327--1353, 2015.

\bibitem{OLADYSHKIN2012}
S.~Oladyshkin and W.~Nowak.
\newblock Data-driven uncertainty quantification using the arbitrary polynomial
  chaos expansion.
\newblock {\em Reliability Engineering \& System Safety}, 106:179 -- 190, 2012.

\bibitem{Saltelli_book_2008}
A.~Saltelli.
\newblock {\em \black{Global sensitivity analysis: the primer}}.
\newblock John Wiley, 2008.

\bibitem{Bishop_2006}
Christopher~M. Bishop.
\newblock {\em \black{Pattern Recognition and Machine Learning (Information
  Science and Statistics)}}.
\newblock Springer-Verlag, Berlin, Heidelberg, 2006.

\bibitem{Laio_Parrinello_PNAS_2002}
Alessandro Laio and Michele Parrinello.
\newblock Escaping free-energy minima.
\newblock {\em Proceedings of the National Academy of Sciences},
  99(20):12562--12566, 2002.

\bibitem{Fishman96}
G.S. Fishman.
\newblock {\em {Monte Carlo: Concepts, Algorithms, and Applications}}.
\newblock Springer-Verlag New York, Inc., 1996.

\bibitem{Kucherenko_2015}
Sergei Kucherenko, Daniel Albrecht, and Andrea Saltelli.
\newblock \black{Exploring multi-dimensional spaces: a Comparison of Latin
  Hypercube and Quasi Monte Carlo Sampling Techniques}, 2015.

\bibitem{Giles_2015}
Michael~B. Giles.
\newblock \black{Multilevel Monte Carlo methods}.
\newblock {\em Acta Numerica}, 24:259 -- 328, 2015.

\bibitem{Heinrich_MLMC_2001}
Stefan Heinrich.
\newblock \black{Multilevel Monte Carlo Methods}.
\newblock In Svetozar Margenov, Jerzy Wa{\'{s}}niewski, and Plamen Yalamov,
  editors, {\em Large-Scale Scientific Computing}, pages 58--67, Berlin,
  Heidelberg, 2001. Springer Berlin Heidelberg.

\bibitem{Pisaroni_Nobile_CMAME_2017}
M.~Pisaroni, F.~Nobile, and P.~Leyland.
\newblock \black{A Continuation Multi Level Monte Carlo (C-MLMC) method for
  uncertainty quantification in compressible inviscid aerodynamics}.
\newblock {\em Computer Methods in Applied Mechanics and Engineering}, 326:20
  -- 50, 2017.

\bibitem{Koutsourelakis_SISC_2009}
P.~Koutsourelakis.
\newblock \black{Accurate Uncertainty Quantification Using Inaccurate
  Computational Models}.
\newblock {\em SIAM Journal on Scientific Computing}, 31(5):3274--3300, 2009.

\bibitem{Peherstorfer_Willcox_SISC_2016}
B.~Peherstorfer, K.~Willcox, and M.~Gunzburger.
\newblock \black{Optimal Model Management for Multifidelity Monte Carlo
  Estimation}.
\newblock {\em SIAM Journal on Scientific Computing}, 38(5):A3163--A3194, 2016.

\bibitem{Fox99}
B.L. Fox.
\newblock {\em Strategies for {Quasi-Monte Carlo}}.
\newblock Kluwer Academic Pub., 1999.

\bibitem{Niedeereiter92}
H.~Niederreiter.
\newblock {\em Random number generation and {Quasi-Monte Carlo} methods}.
\newblock SIAM, 1992.

\bibitem{NiederreiterHLZ98}
H.~Niederreiter, P.~Hellekalek, G.~Larcher, and P.~Zinterhof.
\newblock {\em {Monte Carlo} and {Quasi-Monte Carlo} methods 1996}.
\newblock Springer-Verlag, 1998.

\bibitem{Mckay_Beckman_Technometrics_1979}
M.~D. McKay, R.~J. Beckman, and W.~J. Conover.
\newblock \black{Comparison of Three Methods for Selecting Values of Input
  Variables in the Analysis of Output from a Computer Code}.
\newblock {\em Technometrics}, 21(2):239--245, 1979.

\bibitem{Stein_tech87}
M.~Stein.
\newblock Large sample properties of simulations using {Latin Hypercube
  Sampling}.
\newblock {\em Technometrics}, 29(2):143--151, 1987.

\bibitem{Loh_AS96}
W.L. Loh.
\newblock On {Latin} hypercube sampling.
\newblock {\em Ann. Stat.}, 24(5):2058--2080, 1996.

\bibitem{Sack_Welch_Stat_1989}
Jerome Sacks, William~J. Welch, Toby~J. Mitchell, and Henry~P. Wynn.
\newblock \black{Design and Analysis of Computer Experiments}.
\newblock {\em Statist. Sci.}, 4(4):409--423, 11 1989.

\bibitem{Kennedy_Hagan_JRSS_2001}
Marc~C. Kennedy and Anthony O'Hagan.
\newblock \black{Bayesian calibration of computer models}.
\newblock {\em Journal of the Royal Statistical Society: Series B (Statistical
  Methodology)}, 63(3):425--464, 2001.

\bibitem{GP_book_Rasmussen_2006}
CE. Rasmussen and CKI. Williams.
\newblock {\em \black{Gaussian Processes for Machine Learning}}.
\newblock Adaptive Computation and Machine Learning. MIT Press, Cambridge, MA,
  USA, January 2006.

\bibitem{Wiener38}
N.~Wiener.
\newblock {The homogeneous chaos}.
\newblock {\em Amer. J. Math.}, 60:897--936, 1938.

\bibitem{Ghanem_1991Spectralappro}
R.~Ghanem and P.~Spanos.
\newblock {\em Stochastic Finite Elements: A Spectral Approach}.
\newblock Springer-Verlag, 1991.

\bibitem{Xiu_2002Wiener}
D.~Xiu and G.~E. Karniadakis.
\newblock The wiener-askey polynomial chaos for stochastic differential
  equations.
\newblock {\em SIAM J. Sci. Comput.}, 24:619--644, 2002.

\bibitem{Qian_Wu_2008}
Peter Z.~G. Qian and C.~F.~Jeff Wu.
\newblock \black{Bayesian Hierarchical Modeling for Integrating Low-Accuracy
  and High-Accuracy Experiments}.
\newblock {\em Technometrics}, 50(2):192--204, 2008.

\bibitem{Williams_2006}
Brian Williams, Dave Higdon, Jim Gattiker, Leslie Moore, Michael McKay, and
  Sallie Keller-McNulty.
\newblock \black{Combining experimental data and computer simulations, with an
  application to flyer plate experiments}.
\newblock {\em Bayesian Anal.}, 1(4):765--792, 12 2006.

\bibitem{Oakley_OHagan_Biometrika_2002}
Jeremy Oakley and Anthony O'Hagan.
\newblock \black{Bayesian inference for the uncertainty distribution of
  computer model outputs}.
\newblock {\em Biometrika}, 89(4):769--784, 2002.

\bibitem{Lockwood_Anitescu_2012}
Brian~A. Lockwood and Mihai Anitescu.
\newblock \black{Gradient-Enhanced Universal Kriging for Uncertainty
  Propagation}.
\newblock {\em Nuclear Science and Engineering}, 170(2):168--195, 2012.

\bibitem{XiuK_JCP03}
D.~Xiu and G.E. Karniadakis.
\newblock Modeling uncertainty in flow simulations via generalized polynomial
  chaos.
\newblock {\em J. Comput. Phys.}, 187:137--167, 2003.

\bibitem{GhanemMPW_CMAME05}
R.~Ghanem, S.~Masri, M.~Pellissetti, and R.~Wolfe.
\newblock Identification and prediction of stochastic dynamical systems in a
  polynomial chaos basis.
\newblock {\em Comput. Meth. Appl. Math. Engrg.}, 194:1641--1654, 2005.

\bibitem{Knio_CFD06}
O.M. Knio and O.P. {Le Maitre}.
\newblock Uncertainty propagation in {CFD} using polynomial chaos
  decomposition.
\newblock {\em Fluid Dyn. Res.}, 38(9):616--640, 2006.

\bibitem{Sudret2008}
Bruno Sudret.
\newblock Global sensitivity analysis using polynomial chaos expansions.
\newblock {\em Reliability Engineering \& System Safety}, 93(7):964 -- 979,
  2008.

\bibitem{LiXiu_JCP09}
J.~Li and D.~Xiu.
\newblock A generalized polynomial chaos based ensemble {Kalman} filter with
  high accuracy.
\newblock {\em J. Comput. Phys.}, 228:5454--5469, 2009.

\bibitem{MarzoukXiu2009}
Youssef Marzouk and Dongbin Xiu.
\newblock A stochastic collocation approach to bayesian inference in inverse
  problems.
\newblock {\em Communications in Computational Physics}, 6(4):826--847, 10
  2009.

\bibitem{Li2010}
Jing Li and Dongbin Xiu.
\newblock Evaluation of failure probability via surrogate models.
\newblock {\em J. Comput. Phys.}, 229:8966--8980, November 2010.

\bibitem{Li2018}
Jing Li and Panos Stinis.
\newblock Mori-zwanzig reduced models for uncertainty quantification.
\newblock {\em arXiv:1803.02826}, 2018.

\bibitem{Schobi_2015}
Roland Schobi, Bruno Sudret, and Joe Wiart.
\newblock \black{Polynomial-chaos-based Kriging}.
\newblock {\em International Journal for Uncertainty Quantification},
  5(2):171--193, 2015.

\bibitem{Gratiet_2016}
Lo{\"i}c~Le Gratiet, Stefano Marelli, and Bruno Sudret.
\newblock {\em \black{Metamodel-Based Sensitivity Analysis: Polynomial Chaos
  Expansions and Gaussian Processes}}, pages 1--37.
\newblock Springer International Publishing, Cham, 2016.

\bibitem{Owen_Challenor_SIAM_2017}
N.~Owen, P.~Challenor, P.~Menon, and S.~Bennani.
\newblock \black{Comparison of Surrogate-Based Uncertainty Quantification
  Methods for Computationally Expensive Simulators}.
\newblock {\em SIAM/ASA Journal on Uncertainty Quantification}, 5(1):403--435,
  2017.

\bibitem{Roy_Mocayd_2018}
Pamphile~T. Roy, Nabil El~Mo{\c{c}}ayd, Sophie Ricci, Jean-Christophe Jouhaud,
  Nicole Goutal, Matthias De~Lozzo, and M{\'e}lanie~C. Rochoux.
\newblock \black{Comparison of polynomial chaos and Gaussian process surrogates
  for uncertainty quantification and correlation estimation of spatially
  distributed open-channel steady flows}.
\newblock {\em Stochastic Environmental Research and Risk Assessment},
  32(6):1723--1741, Jun 2018.

\bibitem{MathelinH_NASA03}
L.~Mathelin and M.Y. Hussaini.
\newblock A stochastic collocation algorithm for uncertainty analysis.
\newblock Technical Report NASA/CR-2003-212153, NASA Langley Research Center,
  2003.

\bibitem{XiuH_SISC05}
D.~Xiu and J.S. Hesthaven.
\newblock High-order collocation methods for differential equations with random
  inputs.
\newblock {\em SIAM J. Sci. Comput.}, 27(3):1118--1139, 2005.

\bibitem{Babuska2007}
Ivo Babu$\check{\rm{s}}$ka, Fabio Nobile, and Ra$\acute{\rm{u}}$l Tempone.
\newblock A stochastic collocation method for elliptic partial differential
  equations with random input data.
\newblock {\em SIAM J. Numer. Anal.}, 45(3):1005--1034, 2007.

\bibitem{nobile2008sparse}
Fabio Nobile, Ra{\'u}l Tempone, and Clayton~G Webster.
\newblock A sparse grid stochastic collocation method for partial differential
  equations with random input data.
\newblock {\em SIAM J. Numer. Anal.}, 46(5):2309--2345, 2008.

\bibitem{Ma2009}
X.~Ma and N.~Zabaras.
\newblock An adaptive hierarchical sparse grid collocation algorithm for the
  solution of stochastic differential equations.
\newblock {\em J. Comput. Phys.}, 228(8):3084--3113, 2009.

\bibitem{Foo2010}
J.~Foo and G.~Em. Karniadakis.
\newblock Multi-element probabilistic collocation method in high dimensions.
\newblock {\em J. Comput. Phys.}, 229(5):1536 -- 1557, 2010.

\bibitem{CONSTANTINE2012}
Paul~G. Constantine, Michael~S. Eldred, and Eric~T. Phipps.
\newblock Sparse pseudospectral approximation method.
\newblock {\em Computer Methods in Applied Mechanics and Engineering},
  229-232:1 -- 12, 2012.

\bibitem{jakemanG2013}
John~D Jakeman and Stephen~G Roberts.
\newblock Local and dimension adaptive stochastic collocation for uncertainty
  quantification.
\newblock In {\em Sparse grids and applications}, pages 181--203. Springer,
  2013.

\bibitem{Li2016}
Jing Li and Panos Stinis.
\newblock A unified framework for mesh refinement in random and physical space.
\newblock {\em Journal of Computational Physics}, 323:243 -- 264, 2016.

\bibitem{Doostan_2011nonadapted}
A.~Doostan and H.~Owhadi.
\newblock A non-adapted sparse approximation of pdes with stochastic inputs.
\newblock {\em J. Comput. Phys}, 230:3015--3034, 2011.

\bibitem{Yan_2012Sc}
L.~Yan, L.~Guo, and D.~Xiu.
\newblock Stochastic collocation algorithms using $l^1$ minimization.
\newblock {\em Inter. J. Uncertain Quantification}, 2:279--293, 2012.

\bibitem{Rauhut_2012sparseLegen}
H.~Rauhut and R.~Ward.
\newblock Sparse legendre expansions via $\ell_1$-minimization.
\newblock {\em J. Approx. Theory}, 164:517--533, 2012.

\bibitem{mathelin_gallivan_2012}
L.~Mathelin and K.~A. Gallivan.
\newblock A compressed sensing approach for partial differential equations with
  random input data.
\newblock {\em Communications in Computational Physics}, 12(4):919?954, 2012.

\bibitem{Yang_2013reweightedL1}
X.~Yang and G.~E. Karniadakis.
\newblock Reweighted $\ell_1$ minimization method for stochastic elliptic
  differential equations.
\newblock {\em J. Comput. Phys.}, 248:87--108, 2013.

\bibitem{Hampton_2015Cs}
J.~Hampton and A.~Doostan.
\newblock Compressive sampling of polynomial chaos expansions: Convergence
  analysis and sampling strategies.
\newblock {\em J. Comput. Phys}, 280:363--386, 2015.

\bibitem{Peng_2016gradientL1}
J.~Peng, J.~Hampton, and A.~Doostan.
\newblock On polynomial chaos expansion via gradient-enhanced
  $\ell_1$-minimization.
\newblock {\em J. Comput. Phys}, 310:440--458, 2016.

\bibitem{Yan_2017}
Liang Yan, Yeonjong Shin, and Dongbin Xiu.
\newblock Sparse approximation using $\ell_1-\ell_2$ minimization and its
  application to stochastic collocation.
\newblock {\em SIAM Journal on Scientific Computing}, 39(1):A229--A254, 2017.

\bibitem{Liu_2016QMCL1}
Y.~L. Liu and L.~Guo.
\newblock Stochastic collocation via l1-minimisation on low discrepancy point
  sets with application to uncertainty quantification.
\newblock {\em EAJAM}, 6:171--191, 2016.

\bibitem{Lei_Karniadakis_JCP_2017}
H.~Lei, X.~Yang, Z.~Li, and G.~E. Karniadakis.
\newblock Systematic parameter inference in stochastic mesoscopic modeling.
\newblock {\em {J. Comput. Phys.}}, 330(4):571--593, 2017.

\bibitem{ALEMAZKOOR2017}
Negin Alemazkoor and Hadi Meidani.
\newblock Divide and conquer: An incremental sparsity promoting compressive
  sampling approach for polynomial chaos expansions.
\newblock {\em Computer Methods in Applied Mechanics and Engineering}, 318:937
  -- 956, 2017.

\bibitem{DIAZ2018}
Paul Diaz, Alireza Doostan, and Jerrad Hampton.
\newblock Sparse polynomial chaos expansions via compressed sensing and
  d-optimal design.
\newblock {\em Computer Methods in Applied Mechanics and Engineering}, 336:640
  -- 666, 2018.

\bibitem{RAI2018}
P.~Rai, K.~Sargsyan, and H.~Najm.
\newblock Compressed sparse tensor based quadrature for vibrational quantum
  mechanics integrals.
\newblock {\em Computer Methods in Applied Mechanics and Engineering}, 336:471
  -- 484, 2018.

\bibitem{Huang_Protein_book_2005}
Kerson Huang.
\newblock {\em \black{Lectures on Statistical Physics and Protein Folding}}.
\newblock WORLD SCIENTIFIC, 2005.

\bibitem{Kumar_Kollman_JCC_1992}
Shankar Kumar, John~M. Rosenberg, Djamal Bouzida, Robert~H. Swendsen, and
  Peter~A. Kollman.
\newblock The weighted histogram analysis method for free-energy calculations
  on biomolecules. i. the method.
\newblock {\em Journal of Computational Chemistry}, 13(8):1011--1021, 1992.

\bibitem{Mar_Van_JCP_2008}
Luca Maragliano and Eric Vanden-Eijnden.
\newblock Single-sweep methods for free energy calculations.
\newblock {\em The Journal of Chemical Physics}, 128(18):184110, 2008.

\bibitem{ConstantineDW14}
Paul~G Constantine, Eric Dow, and Qiqi Wang.
\newblock Active subspace methods in theory and practice: Applications to
  kriging surfaces.
\newblock {\em SIAM J. Sci. Comput.}, 36(4):A1500--A1524, 2014.

\bibitem{LIweixuan2015}
Weixuan Li and Guang Lin.
\newblock An adaptive importance sampling algorithm for bayesian inversion with
  multimodal distributions.
\newblock {\em Journal of Computational Physics}, 294:173 -- 190, 2015.

\bibitem{Vittaldev2016}
Vivek Vittaldev, Ryan~P. Russell, and Richard Linares.
\newblock Spacecraft uncertainty propagation using gaussian mixture models and
  polynomial chaos expansions.
\newblock {\em Journal of Guidance, Control, and Dynamics}, 39(12):2615--2626,
  December 2016.

\bibitem{FEINBERG2015}
Jonathan Feinberg and Hans~Petter Langtangen.
\newblock Chaospy: An open source tool for designing methods of uncertainty
  quantification.
\newblock {\em Journal of Computational Science}, 11:46 -- 57, 2015.

\bibitem{Zabaras_2014}
Jiang Wan and Nicholas Zabaras.
\newblock A probabilistic graphical model based stochastic input model
  construction.
\newblock {\em Journal of Computational Physics}, 272:664 -- 685, 2014.

\bibitem{WanK_SISC06}
X.~Wan and G.E. Karniadakis.
\newblock Multi-element generalized polynomial chaos for arbitrary probability
  measures.
\newblock {\em SIAM J. Sci. Comput.}, 28:901--928, 2006.

\bibitem{Witteveen_Bijl_2006}
Jeroen A.~S. Witteveen and Hester Bijl.
\newblock Modeling arbitrary uncertainties using gram-schmidt polynomial chaos.
\newblock In {\em 44th AIAA Aerospace Sciences Meeting and Exhibit}, pages
  1706--1713. American Institute of Aeronautics and Astronautics, 2006.

\bibitem{Zheng_2015MEPCM}
M.~Zheng, X.~Wan, and G.~E. Karniadakis.
\newblock Adaptive multi-element polynomial chaos with discrete measure:
  Algorithms and application to spdes.
\newblock {\em Applied Numerical Mathematics}, 90:91--110, 2015.

\bibitem{YIN2018}
Shengwen Yin, Dejie Yu, Zhen Luo, and Baizhan Xia.
\newblock An arbitrary polynomial chaos expansion approach for response
  analysis of acoustic systems with epistemic uncertainty.
\newblock {\em Computer Methods in Applied Mechanics and Engineering}, 332:280
  -- 302, 2018.

\bibitem{AHLFELD2016}
R.~Ahlfeld, B.~Belkouchi, and F.~Montomoli.
\newblock Samba: Sparse approximation of moment-based arbitrary polynomial
  chaos.
\newblock {\em Journal of Computational Physics}, 320:1 -- 16, 2016.

\bibitem{dunkl_xu_2014}
Charles~F. Dunkl and Yuan Xu.
\newblock {\em Orthogonal Polynomials of Several Variables}.
\newblock Encyclopedia of Mathematics and its Applications. Cambridge
  University Press, 2 edition, 2014.

\bibitem{Candes_2005error}
E.~Cand\`{e}s, M.~Rudelson, T.~Tao, and R.~Vershynin.
\newblock Error correction via linear programming.
\newblock In {\em 46th Annual IEEE Symposium on Foundations of Computer Science
  (FOCS'05)}, pages 668--681, Oct 2005.

\bibitem{Candes_2008Rip}
E.~J. Cand\`{e}s.
\newblock The restricted isometry property and its implications for compressed
  sensing.
\newblock {\em C. R. Acad. Sci. Paris S\'{e}r. I Math.}, 346:589--592, 2008.

\bibitem{Davies_2010RICLp}
M.~E. Davies and R.~Gribonval.
\newblock Restricted isometry constants where $\ell_p$ sparse recovery can fail
  for $0<p\leq1$.
\newblock {\em IEEE Trans. Inf. Theory}, 55:2203--2214, 2010.

\bibitem{Donoho_2006srs}
D.~Donoho, M.~Elad, and V.~Temlyakov.
\newblock Stable recovery of sparse overcomplete representations in the
  presence of noise.
\newblock {\em Information Theory, IEEE Transactions on}, 52(1):6--18, 2006.

\bibitem{Berg_2007spgl}
E.~Van~Den Berg and M.~Friedlander.
\newblock Spgl1: A solver for large-scale sparse reconstruction.
\newblock {\em http:// www.cs.ubc.ca/labs/scl/spgl}, 2007.

\bibitem{Candes_2005Decoding}
E.~J. Cand\`{e}s and T.~Tao.
\newblock Decoding by linear programming.
\newblock {\em IEEE Trans. Inform. Theory}, 51:4203--4215, 2005.

\bibitem{Candes_2006Stablesrec}
E.~J. Cand\`{e}s, J.~Romberg, and T.~Tao.
\newblock Stable signal recovery from incomplete and inaccurate measurements.
\newblock {\em Comm. Pure Appl. Math.}, 56:1207--1223, 2006.

\bibitem{Natarajan_1995L0}
B.~K. Natarajan.
\newblock Sparse approximate solutions to linear systems.
\newblock {\em SIAM J. Sci. Comput.}, 2:227--234, 1995.

\bibitem{Russi_PHD_2001}
Trent~Michael Russi.
\newblock {\em \black{Uncertainty Quantification with Experimental Data and
  Complex System Models}}.
\newblock PhD thesis, University of California, Berkeley, 2001.

\bibitem{Yang_Lei_JCP_2016}
X.~Yang, H.~Lei, N.~A. Baker, and G.~Lin.
\newblock Enhancing sparsity of hermite polynomial expansions by iterative
  rotations.
\newblock {\em J. Comput. Phys.}, 307:94 -- 109, 2016.

\bibitem{JardakSK02}
M~Jardak, Chau-Hsing Su, and George~Em Karniadakis.
\newblock Spectral polynomial chaos solutions of the stochastic advection
  equation.
\newblock {\em J. Sci. Comput.}, 17(1-4):319--338, 2002.

\bibitem{Yang_Wan_rotation_2017}
Xiu Yang, Xiaoliang Wan, and Lin Lin.
\newblock A general framework of enhancing sparsity of generalized polynomial
  chaos expansions, 2017.

\bibitem{Ati_Bahar_BJ_2001}
A.~R. Atilgan, S.~R. Durell, R.~L. Jernigan, M.~C. Demirel, O.~Keskin, and
  I.~Bahar.
\newblock Anisotropy of fluctuation dynamics of proteins with an elastic
  network model.
\newblock {\em Biophysical Journal}, 80(1):505--515, 2001.

\bibitem{ren_biomolecular_2012}
Pengyu Ren, Jaehun Chun, Dennis~G. Thomas, Michael~J. Schnieders, Marcelo
  Marucho, Jiajing Zhang, and Nathan~A. Baker.
\newblock Biomolecular electrostatics and solvation: a computational
  perspective.
\newblock {\em Quarterly reviews of biophysics}, 45(4):427--491, November 2012.

\bibitem{baker_biomolecular_2005}
Nathan~A. Baker.
\newblock Biomolecular {Applications} of {Poisson}?{Boltzmann} {Methods}.
\newblock In Kenny~B. Lipkowitz, Raima Larter, and Thomas~R. Cundari, editors,
  {\em Reviews in {Computational} {Chemistry}}, pages 349--379. John Wiley \&
  Sons, Inc., 2005.

\bibitem{APBS_2018}
Elizabeth Jurrus, Dave Engel, Keith Star, Kyle Monson, Juan Brandi, Lisa~E.
  Felberg, David~H. Brookes, Leighton Wilson, Jiahui Chen, Karina Liles, Minju
  Chun, Peter Li, David~W. Gohara, Todd Dolinsky, Robert Konecny, David~R.
  Koes, Jens~Erik Nielsen, Teresa Head-Gordon, Weihua Geng, Robert Krasny,
  Guo-Wei Wei, Michael~J. Holst, J.~Andrew McCammon, and Nathan~A. Baker.
\newblock Improvements to the {APBS} biomolecular solvation software suite.
\newblock {\em Protein Science}, 27(1):112--128, January 2018.

\bibitem{Shrake_JMB_1973}
A.~Shrake and J.A. Rupley.
\newblock Environment and exposure to solvent of protein atoms. lysozyme and
  insulin.
\newblock {\em Journal of Molecular Biology}, 79(2):351 -- 371, 1973.

\bibitem{Bajaj_ACM_2016}
Muhibur Rasheed, Nathan Clement, Abhishek Bhowmick, and Chandrajit Bajaj.
\newblock \black{Statistical Framework for Uncertainty Quantification in
  Computational Molecular Modeling}.
\newblock In {\em Proceedings of the 7th ACM International Conference on
  Bioinformatics, Computational Biology, and Health Informatics}, pages
  146--155, New York, NY, USA, 2016. ACM.

\bibitem{Bajaj_JCB_2018}
Nathan Clement, Muhibur Rasheed, and Chandrajit~Lal Bajaj.
\newblock \black{Viral Capsid Assembly: A Quantified Uncertainty Approach}.
\newblock {\em Journal of Computational Biology}, 25(1):51--71, 2018.

\bibitem{Rauhut_2010CsSM}
H.~Rauhut.
\newblock Compressive sensing and structured random matrices.
\newblock {\em Radon Series Comp. Appl. Math.}, 9:1--92, 2010.

\bibitem{Marrel_Q2_2009}
Amandine Marrel, Bertrand Iooss, B\'{e}atrice Laurent, and Olivier Roustant.
\newblock \black{Calculations of Sobol indices for the Gaussian process
  metamodel}.
\newblock {\em Reliability Engineering \& System Safety}, 94(3):742 -- 751,
  2009.

\bibitem{Chastaing_sobol_2014}
G.~Chastaing, F.~Gamboa, and C.~Prieur.
\newblock \black{Generalized Sobol sensitivity indices for dependent variables:
  numerical methods}.
\newblock {\em Journal of Statistical Computation and Simulation},
  85(7):1306--1333, 2015.

\bibitem{GROMACS}
H.~J.~C. Berendsen, D.~van~der Spoel, and R.~van Drunen.
\newblock {GROMACS}: A message-passing parallel molecular dynamics
  implementation.
\newblock {\em Computer Physics Communications}, 91(1):43--56, September 1995.

\bibitem{RN1}
Junmei Wang, Romain~M. Wolf, James~W. Caldwell, Peter~A. Kollman, and David~A.
  Case.
\newblock Development and testing of a general amber force field.
\newblock {\em Journal of Computational Chemistry}, 25(9):1157--1174, 2004.

\bibitem{RN2}
Christopher~I. Bayly, Piotr Cieplak, Wendy Cornell, and Peter~A. Kollman.
\newblock A well-behaved electrostatic potential based method using charge
  restraints for deriving atomic charges: the resp model.
\newblock {\em The Journal of Physical Chemistry}, 97(40):10269--10280, 1993.

\bibitem{RN3}
Berk Hess, Henk Bekker, J.~C. Berendsen~Herman, and G.~E.~M. Fraaije~Johannes.
\newblock Lincs: A linear constraint solver for molecular simulations.
\newblock {\em Journal of Computational Chemistry}, 18(12):1463--1472, 1998.

\bibitem{RN4}
William~L. Jorgensen, Jayaraman Chandrasekhar, Jeffry~D. Madura, Roger~W.
  Impey, and Michael~L. Klein.
\newblock Comparison of simple potential functions for simulating liquid water.
\newblock {\em The Journal of Chemical Physics}, 79(2):926--935, 1983.

\bibitem{RN5}
Shuichi Miyamoto and A.~Kollman~Peter.
\newblock Settle: An analytical version of the shake and rattle algorithm for
  rigid water models.
\newblock {\em Journal of Computational Chemistry}, 13(8):952--962, 2004.

\bibitem{BakerSJHM01}
Nathan~A Baker, David Sept, Simpson Joseph, Michael~J Holst, and J~Andrew
  McCammon.
\newblock Electrostatics of nanosystems: application to microtubules and the
  ribosome.
\newblock {\em Proceedings of the National Academy of Sciences},
  98(18):10037--10041, 2001.

\end{thebibliography}
\end{document}